\documentclass[11pt]{amsart}
\usepackage{geometry}                
\usepackage{latexsym}
\usepackage{bbm}
\usepackage{amssymb}
\usepackage{amsmath}
\usepackage{mathtools}
\usepackage{comma}
\usepackage{graphicx}
\usepackage{cutwin}
\usepackage{pstricks}
\usepackage{pst-plot}
\usepackage{pgfplots}
\pgfplotsset{compat=1.7}
\numberwithin{equation}{section}

\newcommand{\defeq}{\vcentcolon=}
\newcommand{\eqdef}{=\vcentcolon}

\newcommand{\N}{\mathbb{N}}
\newcommand{\Z}{\mathbb{Z}}

\newcommand{\R}{\mathbb{R}}

\newcommand{\1}{\mathbbm{1}}

\newcommand{\x}{\mathsf{x}}
\newcommand{\y}{\mathsf{y}}

\newcommand{\Prm}{\mathrm{P}}
\newcommand{\Erm}{\mathrm{E}}
\newcommand{\Prmp}{\mathrm{P}_{\!p}}
\newcommand{\Ermp}{\mathrm{E}_{p}}
\newcommand{\Prob}{\mathbb{P}}
\newcommand{\Probi}{\mathrm{P}_{\!\mathit{i}}}
\newcommand{\Probnull}{\mathrm{P}_{\!0}}
\newcommand{\Var}{\mathrm{Var}}
\newcommand{\E}{\mathbb{E}}

\newcommand{\Cluster}{\mathcal{C}}

\newcommand{\eqdist}{\stackrel{\mathrm{law}}{=}}

\newcommand{\F}{\mathcal{F}}
\newcommand{\G}{\mathcal{G}}
\renewcommand{\H}{\mathcal{H}}

\newcommand{\dx}{\mathrm{d}\mathit{x}}

\newcommand{\vel}{\overline{\mathrm{v}}}

\newcommand{\lambdacrit}{\lambda_{\mathrm{c}}}
\newcommand{\pesc}{\mathit{p}_{\mathrm{esc}}}

\newcommand{\domega}{\mathrm{d}\omega}
\usepackage{changes}

\newcommand{\nina}[1]{\textcolor{violet}{ #1}}
\newcommand{\auskommentiert}[1]{\textcolor{cyan}{}}
\newtheorem{theorem}{Theorem}[section]
\newtheorem{lemma}[theorem]{Lemma}
\newtheorem{corollary}[theorem]{Corollary}
\newtheorem{proposition}[theorem]{Proposition}

\theoremstyle{definition}
\newtheorem{definition}[theorem]{Definition}
\theoremstyle{remark}
\newtheorem{remark}[theorem]{Remark}

\begin{document}

\title[Regularity of the speed of biased random walk in $1$-dim percolation]{Regularity of the speed of biased random walk in a one-dimensional percolation model}

\author{Nina Gantert}
\address{Fakult\"at f\"ur Mathematik
		Technische Universit\"at M\"unchen,
		85748 Garching bei M\"unchen, Germany.}
\email{gantert@ma.tum.de}		

\author{Matthias Meiners}
\address{Institut f\"ur Mathematik, Universit\"at Innsbruck, 6060 Innsbruck, Austria.}
\email{matthias.meiners@uibk.ac.at}

\author{Sebastian  M\"uller}

\address{ Aix Marseille Univ, CNRS, Centrale Marseille, I2M, Marseille, France.}
\email{sebastian.muller@univ-amu.fr}


\begin{abstract}
We consider biased random walks on the infinite cluster of a conditional bond percolation model on the infinite ladder graph.
Axelsson-Fisk and H\"aggstr\"om established for this model a phase transition for the asymptotic linear speed $\vel$ of the walk.
Namely, there exists some critical value $\lambdacrit>0$ such that $\vel>0$ if $\lambda\in (0,\lambdacrit)$ and $\vel=0$ if $\lambda>\lambdacrit$.

We show that the speed $\vel$ is continuous in $\lambda$ on the interval $(0,\lambdacrit)$
and differentiable on $(0,\lambdacrit/2)$.
Moreover, we characterize the derivative as a covariance.
For the proof of the differentiability of $\vel$ on $(0,\lambdacrit/2)$,
we require and prove a central limit theorem for the biased random walk.
Additionally, we prove that the central limit theorem fails to hold for $\lambda \geq \lambdacrit/2$.

\end{abstract}
\maketitle
\section{Introduction}

As a model for transport in an inhomogeneous medium, one may consider
a biased random walk on an (infinite) percolation cluster.
The bias, whose strength is given
by some parameter $\lambda>0$, favors the walk to move in a
pre-specified direction.  A very interesting phenomenon predicted first
by Barma and Dhar \cite{Barma+Dhar:1983} concerns the (asymptotic) linear speed.
Namely,  it was conjectured that there exists a critical bias
$\lambdacrit$ such that for $\lambda\in (0, \lambdacrit)$ the walk has
positive speed while for $\lambda>\lambdacrit$ the speed is zero. This
conjecture was partly proved by Berger, Gantert and Peres
\cite{Berger+Gantert+Peres:2003} and Sznitman \cite{Sznitman:2003}: they showed that when the bias
is small enough, the walk exhibits a positive  speed, while for large
bias the speed is zero. Eventually, Fribergh and Hammond  proved the
phase transition in \cite{Fribergh+Hammond:2014}.
 
The reason for these two different regimes is that
that the percolation cluster contains traps (or dead
ends) and the walk faces two competing effects.   When the bias becomes
larger the time spent in such traps (peninsulas stretching out in the
direction of the bias) increases while the time spent on the
backbone (consisting of infinite paths in the direction of the bias)
decreases. Once the bias is sufficiently large the expected time the
walk stays in a typical trap is infinite and hence the speed of the walk
is zero.
(In cases where there are no traps, the behaviour is different:
Deijfen and H\"aggstr\"om \cite{Deijfen+H"aggstr"om:2010} constructed an invariant percolation model
on $\Z^2$ such that biased random walk has zero speed for small
$\lambda$ and positive speed when $\lambda$ is large).

The same phenomenon is known 
for biased random
walks on supercritical Galton-Watson trees with leaves, the corresponding phase transition was 
proved by Lyons, Pemantle and Peres \cite{Lyons+Pemantle+Peres:1996}. (The bias is here assumed to point away from the root.)
The Galton-Watson trees with leaves can be interpreted, in some cases, as infinite percolation clusters on a regular tree.
Although the tree case is easier than the lattice $\Z^{d}$, mainly because there is a natural decomposition of the tree in
a backbone and traps, see the textbook of Athreya and Ney \cite[p.\;48]{Athreya+Ney:2004}, there are still many open questions.
For instance, one would like to know if the speed is continuous or differentiable as a function of the bias, and if it is a unimodal function. 

In the case of Galton-Watson trees without leaves, the speed is conjectured to be increasing as a function of the bias.
This conjecture is proved for large enough bias by Ben Arous, Fribergh and Sidoravicius in \cite{BenArous+al:2014}. 
A{\"{\i}}d{\'e}kon gave in \cite{Aidekon:2014} 
a formula for the speed of biased random walks on Galton-Watson trees, which allows to deduce monotonicity for a larger 
(but not the full) range of parameters. The Einstein relation, which relates the derivative of the speed
at the critical parameter with the diffusivity of the unperturbed model,
was derived by Ben Arous, Hu, Olla and Zeitouni in \cite{BenArous+Hu+Olla+Zeitouni:2013}.



In this paper we consider  biased random walk on a one-dimensional
percolation model and study the regularity of the  speed as a function
of the bias $\lambda$. The model was introduced by Axelson-Fisk and 
H\"aggstr\"om \cite{Axelson-Fisk+H"aggstr"om:2009b} as a tractable model that exhibits the
same phenomena as biased random walk on the supercritical percolation model in $\Z^{d}$.
In fact, Axelson-Fisk and  H\"aggstr\"om proved the above phase
transition for this model before the conjecture was settled on $\Z^{d}$.

Even though the model may be considered as one of the easiest
non-trivial models, explicit calculation for the speed could not be
carried out. The main result of our paper is that the speed (for fixed
percolation parameter $p$) is continuous in $\lambda$ on $(0,\infty)$,
see Theorem \ref{Thm:continuity of the speed}.
The continuity of the speed may seem obvious, but to our best knowledge,
it has not been proved for a biased random walk on a percolation cluster, and not even for biased random walk on Galton-Watson trees.
Moreover, we prove that the speed is differentiable in
$\lambda$ on $(0,\lambdacrit/2)$ and we characterize the derivative as the
covariance of a suitable two-dimensional Brownian motion, see Formula \eqref{eq:formula for the speed}.
(We hope to address the derivative at $\lambda =0$ in future work).
The main ingredient of the proof of the latter result is an invariance
principle for the biased random walk,
which holds for $\lambda < \lambdacrit/2$ and fails to hold for $\lambda \geq \lambdacrit/2$.

Let us remark that invariance principles  for random walks on infinite clusters of
supercritical i.i.d.\ percolation on $\Z^{d}$ are known for simple
random walks, see De Masi et al. \cite{DeMasi+al:1989}, Sidoravicius and
Sznitman \cite{Sidoravicius+Sznitman:2004}, Berger and Biskup \cite{Berger+Biskup:2007},
and Mathieu and Piatnitski \cite{Mathieu+Piatnitski:2007}.
The case of  Galton-Watson trees was addressed by Peres and Zeitouni in \cite{Peres+Zeitouni:2008}: they proved a 
quenched invariance principle for biased
random walks on supercritical Galton-Watson trees without leaves.
For biased random walk on percolation clusters on $\Z^{d}$, a central limit theorem was proved for 
$\lambda<\lambdacrit/2$ by Fribergh and Hammond, see \cite{Fribergh+Hammond:2014}.

\section{Preliminaries and main results}	\label{sec:main results}

In this section we give a brief review of the percolation and random walk model studied in this paper.

\subsection{Percolation on the ladder graph.}
Consider the infinite ladder graph $\mathcal{L} = (V,E)$.
The  vertex set $V$ is identified with $\Z \times \{0,1\}$.
Two vertices $v,w \in V$ share an edge
if they are at Euclidean distance one from each other.
In this case we either write ${\langle v,w\rangle \in E}$ or $v\sim w$, and say that $v$ and $w$ are neighbors.
Axelson-Fisk and H\"aggstr\"om \cite{Axelson-Fisk+H"aggstr"om:2009} introduced
a percolation model on this graph that may be labelled \lq\lq i.\,i.\,d.~bond percolation on the ladder graph conditioned on the existence of a bi-infinite path\rq\rq.

Let $\Omega \defeq \{0,1\}^E$.
The elements $\omega \in \Omega$ are called \emph{configurations} throughout the paper.
A path in $\mathcal{L}$ is a finite sequence of distinct edges connecting a finite sequence of neighboring vertices.
Given a configuration $\omega \in \Omega$, we call a path $\pi$ in $\mathcal{L}$ \emph{open} if $\omega(e)=1$ for each edge $e \in \pi$.
For a configuration $\omega$ and a vertex $v \in V$,
$\Cluster_{\omega}(v)$ denotes the connected component in $\omega$ that contains $v$,
i.\,e.,
\begin{equation*}
\Cluster_{\omega}(v)	=	\{w \in V: \text{there is an open path in } \omega \text{ connecting } v \text{ and } w\}.
\end{equation*}
We denote by $\x: V \to \Z$ and $\y:V \to \{0,1\}$ the projections from $V$ to $\Z$ and $\{0,1\}$, respectively.
Hence, for any $v \in V$, $v = (\x(v),\y(v))$.
We call $\x(v)$ the $\x$-coordinate of $v$, and $\y(v)$ the $\y$-coordinate of $v$.
For $N_1, N_2 \in \N$, let $\Omega_{N_1,N_2}$ be the event that
there exists an open path from some $v_1 \in V$ to some $v_2 \in V$ with $\x$-coordinates $-N_1$ and $N_2$, respectively,
and let $\Omega^* \defeq \bigcap_{N_1, N_2 \geq 0} \Omega_{N_1,N_2}$
be the event that there is an infinite path connecting $-\infty$ and $+\infty$.

Denote by $\F$ the $\sigma$-field on $\Omega$ generated by the projections $p_e: \Omega \to \{0,1\}$,
$\omega \mapsto \omega(e)$, $e \in E$.
For $p \in (0,1)$, let $\mu_p$ be the distribution of i.\,i.\,d.\ bond percolation on $(\Omega,\F)$ with $\mu_p(\omega(e)=1)=p$ for all $e \in E$.
The Borel-Cantelli lemma implies $\mu_p(\Omega^*)=0$.
Write $\Prm_{p,N_1,N_2}(\cdot) \defeq \mu_p(\cdot \cap \Omega_{N_1,N_2})/\mu_p(\Omega_{N_1,N_2})$
for the probability distribution on $\Omega$
that arises from conditioning on the existence of an open path from $\x$-coordinate $-N_1$ to $\x$-coordinate $N_2$.
The following result is Theorem 2.1 in \cite{Axelson-Fisk+H"aggstr"om:2009}:

\begin{theorem}
The probability measures $\Prm_{p,N_1,N_2}$ converge weakly as $N_1,N_2 \!\to\! \infty$
to a probability measure $\Prmp^*$ on $(\Omega,\F)$ with $\Prmp^*(\Omega^*)=1$.
\end{theorem}

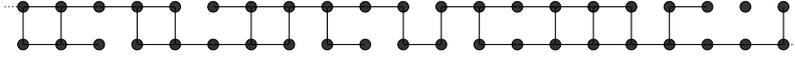
\begin{center}
\pgfmathsetseed{845133}
\begin{tikzpicture}[thin, scale=0.5,-,
                   shorten >=0pt+0.5*\pgflinewidth,
                   shorten <=0pt+0.5*\pgflinewidth,
                   every node/.style={circle,
                                      draw,
                                      fill          = black!80,
                                      inner sep     = 0pt,
                                      minimum width =4 pt}]

\def \p {0.75}

\foreach \x in {-10,-9,-8,-7,-6,-5,-4,-3,-2,-1,0,1,2,3,4,5,6,7,8,9,10}
\foreach \y in {0,1}
    \node at (\x,\y) {};

\foreach \x in {-10,-9,-8,-7,-6,-5,-4,-3,-2,-1,0,1,2,3,4,5,6,7,8,9}{
\foreach \y in {0,1}{
    \pgfmathparse{rnd}
    \let\dummynum=\pgfmathresult
    \ifdim\pgfmathresult pt < \p pt\relax \draw (\x,\y) -- (\x+1,\y);\fi
  }}
  
\foreach \x in {-10,-9,-8,-7,-6,-5,-4,-3,-2,-1,0,1,2,3,4,5,6,7,8,9,10}{
\foreach \y in {0}{
    \pgfmathparse{rnd}
    \let\dummynum=\pgfmathresult
    \ifdim\pgfmathresult pt < \p pt\relax \draw (\x,\y) -- (\x,\y+1);\fi
  }}

	\draw[densely dotted] (-10.5,1) -- (-10,1);
	\draw (-3,1) -- (-2,1);
	\draw (9,0) -- (10,0);
	\draw[densely dotted] (10,0) -- (10.5,0);
\end{tikzpicture}
\end{center}

Given $\omega \in \Omega^*$, denote by $\Cluster = \Cluster_\omega$ the a.\,s.\ unique infinite open cluster.
Define $\Omega_{\mathbf{0}} \defeq \{\omega \in \Omega^*: \mathbf{0} \in \Cluster\}$ and
$\Prmp(\cdot) \defeq \Prmp^*(\cdot | \Omega_{\mathbf{0}})$ where $\mathbf{0} \defeq (0,0)$.  
The measure $\Prmp$ will serve as the law of the percolation environment
for the random walk which is introduced next.

\subsection{Random walk in the infinite percolation cluster.}

We consider the random walk model introduced by Axelson-Fisk and H\"agg\-str\"om in \cite{Axelson-Fisk+H"aggstr"om:2009b}.
However, in order to be more consistent with other works on biased random walks we will use a different parametrization.
State and trajectory space of the walk are $V$ and $V^{\N_0}$, respectively.
By $Y_n: V^{\N_0} \to V$, we denote the projection from $V^{\N_0}$ onto the $n$th coordinate, $n \in \N_0$.
We equip $V^{\N_0}$ with the $\sigma$-field $\G = \sigma(Y_n: n \in \N_0)$.
Fix $\lambda \geq 0$. Given a configuration $\omega \in \Omega$,
let $P_{\omega,\lambda}$ denote the distribution on $V^{\N_0}$ that makes $Y \defeq (Y_n)_{n \in \N_0}$
a Markov chain on $V$ with initial position $\mathbf{0} \defeq (0,0)$
and transition probabilities
\begin{equation}	\label{eq:P_xi lazy}
p_{\omega,\lambda}(v,w)
=	P_{\omega,\lambda}(Y_{n+1} = w \mid Y_n = v)
=	\frac{e^{\lambda (\x(w)-\x(v))}}{e^{\lambda}+1+e^{-\lambda}} \1_{\{\omega(e)=1\}}
\end{equation}
for $v \sim w$ and
\begin{equation*}	\label{eq:P_xi lazy v to v}
p_{\omega,\lambda}(v,v)	=	P_{\omega,\lambda}(Y_{n+1} = v \mid Y_n = v)	=	1-\sum_{w \sim v} p_{\omega,\lambda}(v,w).
\end{equation*}
We write $P^{\mathbf{0}}_{\omega,\lambda}$ to emphasize the initial position $\mathbf{0}$,
and $P^v_{\omega,\lambda}$ for the distribution of the Markov chain with the same transition probabilities but initial position $v \in V$.
The joint distribution of $\omega$ and $(Y_n)_{n \in \N_0}$
when $\omega$ is drawn at random according to a probability distribution $Q$ on $(\Omega,\F)$
is denoted by $Q \times P^v_{\omega,\lambda} \eqdef \Prob_{Q,\lambda}^v$ where $v$ is the initial position of the walk.
Formally, it is defined by
\begin{equation}	\label{eq:|P_Q}
\Prob_{Q,\lambda}^v(F \times G)	~=~	\int_F P_{\omega,\lambda}^v(G) \, Q(\domega),	\quad	F \in \F,\;G \in \G.
\end{equation}
We fix $p \in (0,1)$ throughout this paper
and write $\Prob_{\lambda}^v$ for $\Prob_{\Prmp,\lambda}^v$
and $\Prob_{\lambda}$ for $\Prob_{\lambda}^{\mathbf{0}}$.
Then \eqref{eq:|P_Q} becomes
\begin{equation}	\label{eq:|P_p}
\Prob_{\lambda}(F \times G)	~=~	\int_F P_{\omega,\lambda}(G) \, \Prmp(\domega)	~=~	\Ermp[\1_{\{\omega \in F\}} P_{\omega,\lambda}(G)]
\end{equation}
where $\Ermp$ denotes expectation with respect to $\Prmp$.
We write $\Prob^*_{\lambda}$ for $\Prob_{\Prmp^*,\lambda}^{\mathbf{0}}$.

\subsection{The random walk revisited.}

We review two results from \cite{Axelson-Fisk+H"aggstr"om:2009b} that are important for the paper at hand.

\begin{proposition}[Proposition 3.1 in \cite{Axelson-Fisk+H"aggstr"om:2009b}]	\label{Prop:recurrence/transience}
The random walk $(Y_n)_{n \in \N_0}$ is recurrent under  $P^{\mathbf{0}}_{\omega,0}$
and transient under  $P^{\mathbf{0}}_{\omega,\lambda}$ for $\lambda \not = 0$, for  $\Prmp$-almost all $\omega$.
\end{proposition}

Define $X_n \defeq \x(Y_n)$, $n \in \N_0$ as the projection on the $\x$-coordinate.
In the biased case, a strong law of large numbers holds for $X_n$:

\begin{proposition}[Theorem 3.2 in \cite{Axelson-Fisk+H"aggstr"om:2009b}]	\label{Prop:SLLN}
For any $\lambda > 0$, there exists a deterministic constant $\vel(\lambda) = \vel(p,\lambda) \in [0,1]$ such that
\begin{equation*}	\textstyle
\frac{X_n}{n}	~\to~	 \vel(\lambda)	\quad	\Prob_{\lambda} \text{-a.\,s.\ as } n \to \infty.
\end{equation*}
Furthermore, there exists a critical value $\lambdacrit = \lambdacrit(p) > 0$ such that
\begin{equation*}
\vel(\lambda) > 0	\text{ for } 0 < \lambda < \lambdacrit
\quad	\text{ and }	\quad
\vel(\lambda) = 0	\text{ for } \lambda \geq \lambdacrit.
\end{equation*}
The critical value $\lambdacrit$ is
\begin{equation}	\label{eq:lambdacrit}	\textstyle
\lambdacrit	~=~	\frac{1}{2} \log\Big(2/\big(1+2p-2p^2-\sqrt{1+4p^2-8p^3+4p^4} \big)\Big).
\end{equation}
\end{proposition}

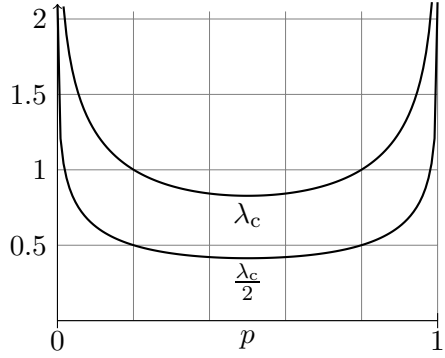
\begin{figure}
\begin{center}
\begin{tikzpicture}[scale=1]
		\draw [help lines] (0,0) grid (5,4.1); 
		\draw [-] (0,0) -- (5,0);
		\draw [->] (0,-0.1) -- (0,4.2);
		\draw [-] (5,-0.1) -- (5,0.1);
		\draw (0,0) node[below]{$0$};
		\draw (2.5,0) node[below]{$p$};
		\draw (5,0) node[below]{$1$};
		\draw (0,1) node[left]{$0.5$};
		\draw (0,2) node[left]{$1$};
		\draw (0,3) node[left]{$1.5$};
		\draw (0,4) node[left]{$2$};
		\draw [thick, domain=0.08:4.925, samples=128] plot(\x, {ln(2/(1+2*(\x/5)-2*(\x/5)^2-sqrt(1+4*(\x/5)^2-8*(\x/5)^3+4*(\x/5)^4)))});
		\draw (2.5,1.7) node[below]{$\lambdacrit$};
		\draw [thick, domain=0.0012:4.999, samples=128] plot(\x, {0.5*ln(2/(1+2*(\x/5)-2*(\x/5)^2-sqrt(1+4*(\x/5)^2-8*(\x/5)^3+4*(\x/5)^4)))});
		\draw (2.5,0.9) node[below]{$\frac{\lambdacrit}{2}$};
\end{tikzpicture}
\vspace{-0.3cm}
\caption{The figure shows $\lambdacrit$ and $\lambdacrit/2$ as functions of $p$.
The critical value $\lambdacrit$ is symmetric around $1/2$, \textit{i.e.}, $\lambdacrit(p) = \lambdacrit(1-p)$.}	
\label{fig:lambdacrit}
\end{center}
\end{figure}

\subsection{Regularity of the speed.}

Our first main result is the following theorem.

\begin{theorem}	\label{Thm:continuity of the speed}
The speed $\vel$ is continuous in $\lambda$ on the interval $(0,\infty)$.
Further, for any $\lambda^* \in (0,\lambdacrit)$ and any $1 < r < \frac{\lambdacrit}{\lambda^*} \wedge 2$,
we have
\begin{equation}	\label{eq:continuity of the speed}
\lim_{\lambda \to \lambda^*} \frac{\vel(\lambda)-\vel(\lambda^*)}{(\lambda-\lambda^*)^{r-1}} ~=~ 0.
\end{equation}
\end{theorem}

For $\lambda \in (0,\lambdacrit/2)$, we show a stronger statement:

\begin{theorem}	\label{Thm:differentiability of the speed}
The speed $\vel$ is differentiable in $\lambda$ on the interval $(0,\lambdacrit/2)$, and
the derivative is given in \eqref{eq:formula for the speed} below.
\end{theorem}

The differentiability of $\vel$ at $\lambda=0$ together with the statement $\vel'(0) = \sigma^2$
for the limiting variance $\sigma^2$ of $n^{-1/2} X_n$ under the distribution $\Prob_0$
is the Einstein relation for this model.
We will consider the Einstein relation in a follow-up paper.

\subsection{Sketch of the proof.}

Fix $\lambda^* \in (0,\lambdacrit)$ and let $1 < r < \lambdacrit/\lambda^*$ if $\lambda^* \geq \lambdacrit/2$,
and $r = 2$ if $\lambda^* < \lambdacrit/2$.
In order to prove Theorems \ref{Thm:continuity of the speed} and \ref{Thm:differentiability of the speed},
we show that
\begin{equation*}
\lim_{\lambda \to \lambda^*} \frac{\vel(\lambda)-\vel(\lambda^*)}{(\lambda-\lambda^*)^{r-1}}
=
\begin{cases}
0					&	\text{if } \lambda^* \geq \lambdacrit/2,	\\
\vel'(\lambda^*)	&	\text{if } \lambda^* < \lambdacrit/2.
\end{cases}
\end{equation*}
Since $\vel(\lambda) = \lim_{n \to \infty} \frac1n \E_{\lambda}[X_n]$ by Lebesgue's dominated convergence theorem,
we need to understand the quantity
\begin{equation*}
\frac{\E_{\lambda}[X_n]-\E_{\lambda^*}[X_n]}{n(\lambda-\lambda^*)^{r-1}}
\end{equation*}
as first $n \to \infty$ and then $\lambda \to \lambda^*$.
We follow ideas from \cite{Gantert+al:2012,Mathieu:2015}
and replace the double limit by a suitable simultaneous limit.
For instance, consider the case $\lambda^* < \lambdacrit/2$, i.\,e., $r=2$.
Then the expected difference between $X_n$ under $\Prob_\lambda$ and $\Prob_{\lambda^*}$
is of the order $n(\lambda-\lambda^*) \vel'(\lambda^*)$.
On the other hand, when a central limit theorem for $X_n$ with square-root scaling holds,
the fluctuations of $X_n$ are of order $\sqrt{n}$.
By matching these two scales, that is, $(\lambda-\lambda^*) \approx n^{-1/2}$,
we are able to apply a measure-change argument replacing
$\E_{\lambda}[X_n]$ by an expectation of the form $\E_{\lambda^*}[X_n f_{\lambda,n}]$ for a suitable density function $f_{\lambda,n}$.
In order to understand the limiting behavior of $\E_{\lambda^*}[X_n f_{\lambda,n}]$,
we use a joint central limit theorem for $X_n$ and the leading term in $f_{\lambda,n}$.
In the case $\lambda^* \geq \lambdacrit/2$,
we use Marcinkiewicz-Zygmund-type strong laws for $X_n$ and the leading term in $f_{\lambda,n}$ instead.

\subsection{Functional central limit theorem.}

As mentioned in the preceding paragraph, we will require a joint central limit theorem
for $X_n$ and the leading term of a suitable density. We will make this precise now. 

Fix $\lambda^* \geq 0$ and, for $v \in V$, let $N_{\omega}(v) \defeq \{w \in V: p_{\omega,0}(v,w) > 0\}$.
Notice that $N_{\omega}(v) \not = \varnothing$ even for isolated vertices.
For $w \in N_{\omega}(v)$, the function $\log p_{\omega,\lambda}(v,w)$ is differentiable at $\lambda^*$.
Hence, we can write a first-order Taylor expansion of $\log p_{\omega,\lambda}(v,w)$ as $\lambda\to \lambda^{*}$ in the form
\begin{equation}	\label{eq:Taylor expansion}
\log p_{\omega,\lambda}(v,w)
=\log p_{\omega,\lambda^*}(v,w)
+(\lambda\!-\!\lambda^{*}) \nu_{\omega,\lambda^{*}}(v,w)+(\lambda\!-\!\lambda^{*}) o_{\lambda^{*}}(\lambda\!-\!\lambda^*)
\end{equation}
where $\nu_{\omega,\lambda^{*}}(v,w)$ is the derivative of  $\log p_{\omega,\lambda}(v,w)$ at $\lambda^{*}$
and $o_{\lambda^{*}}(\lambda\!-\!\lambda^*)$ converges to $0$ as $\lambda\to\lambda^{*}$.
Since there is only a finite number of $1$-step transition probabilities,
$o_{\lambda^{*}}(\lambda\!-\!\lambda^*) \to 0$ as $\lambda \to \lambda^*$ uniformly
(in $v$, $w$ and $\omega$).

For all $v$ and all $\omega$, $p_{\omega,\lambda^{*}}(v,\cdot)$ is a probability measure
on $N_{\omega}(v)$ and hence
\begin{equation*}
\sum_{w\in N_{\omega}(v)} \nu_{\omega,\lambda^{*}}(v,w) p_{\omega, \lambda^{*}}(v,w) = 0.
\end{equation*}
Therefore, the sequence $(M^{\lambda^*}_n(\omega))_{n \geq 0}$ defined by $M^{\lambda^*}_0(\omega)=0$ and
\begin{equation}	\label{Mdef}
M^{\lambda^*}_n(\omega)=\sum_{k=1}^{n} \nu_{\omega,\lambda^{*}}(Y_{k-1},Y_{k}),	\quad	  n \in \N
\end{equation}
is a martingale under $P_{\omega,\lambda^{*}}$.
We write $M^{\lambda^*}_n$ for the random variable $M^{\lambda^{*}}_n(\cdot)$ on $\Omega \times V^{\N_0}$
and notice that the sequence $(M^{\lambda^*}_n)_{n \geq 0}$ is also a martingale under the annealed measure $\Prob_{\lambda^{*}}$.

For $t \geq 0$, denote by $\lfloor t \rfloor$ the largest integer $\leq t$.
For $\lambda \geq 0$ and $n \in \N$, put
\begin{equation*}	\textstyle
B_n(t) \defeq \frac{1}{\sqrt{n}}(X_{\lfloor n t \rfloor} - \lfloor n t \rfloor \vel(\lambda)),	\quad	0 \leq t \leq 1.
\end{equation*}
Then $B_n \defeq (B_n(t))_{0 \leq t \leq 1}$ takes values in the Skorokhod space $D[0,1]$
of real-valued right-continuous functions with finite left limits, see e.g.\ \cite[Chap.~3]{Billingsley:1999}.

\begin{theorem}	\label{Thm:joint CLT}
Let $\lambda \in (0,\lambdacrit/2)$. Then
\begin{equation}	\label{eq:joint invariance principle}
(B_n(t),n^{-1/2} M^{\lambda}_{\lfloor nt \rfloor)}		\Rightarrow	(B^{\lambda},M^{\lambda})
\quad	\text{under } \Prob_{\lambda}
\end{equation}
where $\Rightarrow$ denotes convergence in distribution in the Skorokhod space $D[0,1]$
and $(B^{\lambda},M^{\lambda})$ is a two-dimensional centered Brownian motion with covariance matrix
$\Sigma^{\lambda} = (\sigma_{ij}(\lambda))$.
Further,
\begin{equation}	\label{eq:sup kappa-moment X}	\textstyle
\sup_{n \geq 1} \E_{\lambda}[|B_n(1)|^{\kappa}] < \infty
\end{equation}
for some $\kappa = \kappa(\lambda) > 2$.
In particular,
\begin{eqnarray*}
\sigma_{11}(\lambda)	&=& \E_{\lambda}[B^{\lambda}(1)^2]	=		\lim_{n \to \infty} n^{-1} \E_{\lambda}[(X_n-n\vel(\lambda))^2],	\\
\sigma_{22}(\lambda)	&=& \E_{\lambda}[M^{\lambda}(1)^2]	=		\lim_{n \to \infty} n^{-1} \E_{\lambda}[(M^{\lambda}_n)^2],	\\
\sigma_{12}(\lambda)	&=& 
 \E_{\lambda}[B^{\lambda}(1)M^{\lambda}(1)]	=	\lim_{n \to \infty} n^{-1} \E_{\lambda}[(X_n-n\vel(\lambda))M^{\lambda}_n].
\end{eqnarray*}
If $\lambda \geq \lambdacrit/2$, then \eqref{eq:joint invariance principle} fails to hold,
and $B_n$ does not converge in distribution.
\end{theorem}

We do not only require a moment bound for $B_n(1)$, $n \geq 1$
as given in \eqref{eq:sup kappa-moment X},
but also a similar (but stronger) moment bound for the martingale
$M^{\lambda}_n$ for $\lambda \in (0,\lambdacrit)$.
The result we need is the following:

\begin{proposition}	\label{Prop:sup exp M}
Let $p \in (0,1)$, $\lambda \in (0,\lambdacrit)$.
Then, for every $t > 0$,
\begin{equation}	\label{eq:sup exp M}	\textstyle
\sup_{n \geq 1} \E_{\lambda}[e^{t n^{-1/2} M^{\lambda}_n}] < \infty.
\end{equation}
\end{proposition}

\subsection{Marcinkiewicz-Zygmund-type strong laws.}

Even though the central limit theorem for $X_n$ does not hold when $\lambda \geq \lambdacrit/2$,
we can give upper bounds on the fluctuations of $X_n$ around $n \vel(\lambda)$.

\begin{theorem}	\label{Thm:Marcinkiewicz-Zygmund}
Let $p \in (0,1)$, $\lambda \in (0,\lambdacrit)$ and $r < \frac{\lambdacrit}{\lambda} \wedge 2$.
Then
\begin{equation}	\label{eq:Marcinkiewicz-Zygmund}
\frac{X_n-n\vel(\lambda)}{n^{1/r}}	\to 0
\quad	\text{and}	\quad
n^{-1/r} M^{\lambda}_n	\to 0
\quad \Prob_{\lambda}\text{-a.\,s. and in } L^r(\Prob_{\lambda}).
\end{equation}
\end{theorem}

\subsection{Outline of the proofs.}
We continue with an outline of how the joint central limit theorem is used to derive the regularity of the speed.
First of all, for a fixed percolation configuration $\omega$, we have, by writing the Radon-Nikodym derivative,
\begin{equation}\label{eq:girsanov quenched}
E_{\omega,\lambda}[X_n]
~=~	E_{\omega,\lambda^{*}}\bigg[X_n \prod_{j=1}^{n} \frac{p_{\omega,\lambda}(Y_{j-1}, Y_{j})}{p_{\omega,\lambda^{*}}(Y_{j-1}, Y_{j})}\bigg]
\end{equation}
for $\lambda, \lambda^* \geq 0$.
Integration with respect to $\Prmp$ leads to
\begin{equation}\label{eq:girsanov annealed}
\E_{\lambda}[X_{n}]
~=~	\E_{\lambda^{*}}\bigg[X_{n} \prod_{j=1}^{n} \frac{p_{\omega,\lambda}(Y_{j-1},Y_{j})}{p_{\omega,\lambda^{*}}(Y_{j-1},Y_{j})}\bigg].
\end{equation}

As outlined above, we follow the strategy used in \cite{Mathieu:2015}
and prove the differentiability of $\vel$ in four steps:

\begin{enumerate}	\label{enum:program}
	\item
		We prove the joint central limit theorem, Theorem \ref{Thm:joint CLT}.
	\item
		We prove that, for $\lambda^* \in (0,\lambdacrit/2)$,
		\begin{equation}	\label{eq:supsecondmoment}	\textstyle
		\sup_{n \geq 1} \frac1n \E_{\lambda^*}[(X_n-n \vel(\lambda^*))^{2}] < \infty.
		\end{equation}
	\item
		Using the joint central limit theorem and \eqref{eq:supsecondmoment}, we show that, for $\alpha > 0$,
		\begin{equation}	\label{eq:speed approx by covariance}
		\lim_{\substack{\lambda \to \lambda^*,\\ (\lambda-\lambda^*)^2 n \to \alpha}} \frac{\E_{\lambda}[X_n]-\E_{\lambda^*}[X_n]}{(\lambda-\lambda^*)n}
		~=~ \E_{\lambda^*}[B^{\lambda^*}(1) M^{\lambda^*}(1)]
		~=~ \sigma_{12}(\lambda^*).
		\end{equation}
	\item
		We show that, for any $\lambda^* \in (0,\lambdacrit/2)$,
		\begin{equation}	\label{eq:2nd step}
		\lim_{\substack{\lambda \to \lambda^*,\\ (\lambda-\lambda^*)n \to \infty}}
		\bigg[\frac{\vel(\lambda)-\vel(\lambda^*)}{\lambda-\lambda^*} - \frac{\E_{\lambda}[X_n]-\E_{\lambda^*}[X_n]}{(\lambda-\lambda^*)n}\bigg]
		~=~ 0.
		\end{equation}
\end{enumerate}
Notice that
\eqref{eq:2nd step} and \eqref{eq:speed approx by covariance} imply
\begin{equation}	\label{eq:formula for the speed}
\vel'(\lambda^*) = \lim_{\substack{\lambda \to \lambda^*,\\ (\lambda-\lambda^*)^2 n \to \alpha}}
\frac{\E_{\lambda}[X_n]-\E_{\lambda^*}[X_n]}{(\lambda-\lambda^*)n}
= \E_{\lambda^*}[B^{\lambda^*}(1) M^{\lambda^*}(1)].
\end{equation}
The proof of the continuity of $\vel$ on $[\lambdacrit/2,\lambdacrit)$ follows a similar strategy,
where the use of the central limit theorem is replaced by the use of the Marcinkiewicz-Zygmund-type strong law
for $X_n$ and $M^{\lambda}_n$.

\section{Background on the percolation model}	\label{sec:background}

In this section we provide some basic results on the percolation model.
Most of the material presented here goes back to \cite{Axelson-Fisk+H"aggstr"om:2009b,Axelson-Fisk+H"aggstr"om:2009},
while some results are extensions that are taylor-made for our analysis.

\subsection{The percolation law.}

Let $E^{i,\leq}$ and $E^{i, \geq}$ be
the sets of edges (subsets of $E$),
with both endpoints having $\x$-coordinate $\leq i$ or $\geq i$, respectively.
Further, let $E^{i,<} \defeq E \setminus E^{i, \geq}$ and $E^{i,>} \defeq E \setminus E^{i,\leq}$.
Given $\omega \in \Omega$, we call a vertex $v \in V$ \emph{backwards communicating}
if there exists an infinite open path in $E^{\x(v),\leq}$ that contains $v$.
Analogously, we call $v$ \emph{forwards communicating}
if the same is true with $E^{\x(v),\leq}$ replaced by $E^{\x(v),\geq}$.
Loosely speaking, $v$ is backwards communicating
if one can move in $\omega$ from $v$ to $-\infty$ without ever visiting a vertex with $\x$-coordinate larger than $\x(v)$.
Now define
\begin{equation*}
{\tt T}_i	~\defeq~	\begin{cases}
			{\tt 00}	&	\text{if neither $(i,0)$ nor $(i,1)$ are backwards communicating;}	\\
			{\tt 01}	&	\text{if $(i,0)$ is not backwards communicating but $(i,1)$ is;}	\\
			{\tt 10}	&	\text{if $(i,0)$ is backwards communicating but $(i,1)$ is not;}	\\
			{\tt 11}	&	\text{if $(i,0)$ and $(i,1)$ are backwards communicating.}
			\end{cases}
\end{equation*}
We note that ${\tt T}_i$ is a function of $\omega$. When $\omega$ is drawn from $\Prmp^*$,
then ${\tt T} \defeq ({\tt T}_i)_{i \in \Z}$ is a Markov chain with state space $\{{\tt 10}, {\tt 01}, {\tt 11}\}$,
and the distribution of $\omega$ given ${\tt T}$ takes a simple form.
To describe it, we introduce the notion of \emph{compatibility}.
Let $E^{i} \defeq E^{i,\leq} \setminus E^{i-1,\leq}$.
A local configuration $\eta \in \{0,1\}^{E^{i}}$ is called $\texttt{ab}$-$\mathtt{cd}$-\emph{compatible}
for $\mathtt{ab}, \mathtt{cd} \in \{{\tt 00}, {\tt 10}, {\tt 01}, {\tt 11}\}$
if ${\tt T}_{i-1} = \mathtt{ab}$ and $\omega(E^{i})=\eta$ imply ${\tt T}_{i} = \mathtt{cd}$.

\begin{lemma}	\label{Lem:T_i,omega(E^i)}
Under $\Prmp^*$,
$({\tt T}_i)_{i \in \Z}$ is an irreducible and aperiodic time-homo\-geneous Markov chain.
Further, $({\tt T}_i)_{i \in \Z}$ is reversible and ergodic.
The conditional distribution of $(\omega(E^i))_{i \in \Z}$ given $({\tt T}_i)_{i \in \Z}$ is
\begin{equation}	\label{eq:omega's distribution}
\prod_{i \in \Z} \mathrm{P}_{\!p,{\tt T}_{i-1},{\tt T}_i}
\end{equation}
where, for $\mathtt{ab}, \mathtt{cd} \in \{{\tt 00}, {\tt 10}, {\tt 01}, {\tt 11}\}$,
\begin{equation*}
\mathrm{P}_{\!p,{\tt ab},{\tt cd}}(\{\eta\})
~=~	\frac{\1_{\{\eta \text{ is } {\tt ab}\text{-}{\tt cd} \text{-compatible}\}}}{Z_{p,{\tt ab},{\tt cd}}}	\prod_{e \in E^i} p^{\eta(e)} (1-p)^{1-\eta(e)}
\end{equation*}
with a norming constant $Z_{p,{\tt ab},{\tt cd}}$ such that $\mathrm{P}_{\!p,{\tt ab},{\tt cd}}$ is a probability distribution.
\end{lemma}
\begin{proof}
Theorems 3.1 and 3.2 in \cite{Axelson-Fisk+H"aggstr"om:2009} yield that $({\tt T}_i)_{i \in \Z}$ is a stationary time-homogeneous Markov chain.
Aperiodicity follows from the explicit form of the transition matrix $\mathbf{p}$ on pp.\;1111-1112 of the cited reference.
From this explicit form and the form of the invariant distribution $\pi$ given on p.\;1112 of \cite{Axelson-Fisk+H"aggstr"om:2009}
it is readily checked that $\pi$ and $\mathbf{p}$ are in detailed balance.
Hence, $({\tt T}_i)_{i \in \Z}$ is reversible.
Since the state space $\{{\tt 01}, {\tt 10}, {\tt 11}\}$ is finite, $\pi$ is the unique invariant distribution.
Consequently, $({\tt T}_i)_{i \in \Z}$ is ergodic.
The form of the conditional distribution given in \eqref{eq:omega's distribution} is (3.17) of \cite{Axelson-Fisk+H"aggstr"om:2009}.
\end{proof}

\subsection{Cyclic decomposition.}
Next, we introduce a decomposition of the percolation cluster into i.\,i.\,d.~cycles
originally introduced in \cite{Axelson-Fisk+H"aggstr"om:2009b}.
Cycles begin and end at horizontal levels $i$ such that $(i,1)$ is isolated in $\omega$.
A vertex $(i,0)$ such that $(i,1)$ is isolated in $\omega$ is called a pre-regeneration point.
We let $\ldots, R^{\mathrm{pre}}_{-2}, R^{\mathrm{pre}}_{-1}, R^{\mathrm{pre}}_0, R^{\mathrm{pre}}_1, R^{\mathrm{pre}}_2, \ldots$
be an enumeration of the pre-regeneration points such that
$\x(R^{\mathrm{pre}}_{-2}) < \x(R^{\mathrm{pre}}_{-1}) < 0 \leq \x(R^{\mathrm{pre}}_0) <  \x(R^{\mathrm{pre}}_1) < \x(R^{\mathrm{pre}}_2) \ldots$\,.
\smallskip

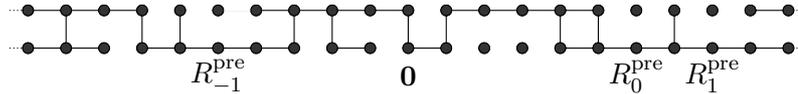
\begin{center}
\pgfmathsetseed{845133}
\begin{tikzpicture}[thin, scale=0.5,-,
                   shorten >=0pt+0.5*\pgflinewidth,
                   shorten <=0pt+0.5*\pgflinewidth,
                   every node/.style={circle,
                                      draw,
                                      fill          = black!80,
                                      inner sep     = 0pt,
                                      minimum width =4 pt}]

\def \p {0.5}

\foreach \x in {-10,-9,-8,-7,-6,-5,-4,-3,-2,-1,0,1,2,3,4,5,6,7,8,9,10}
\foreach \y in {0,1}
    \node at (\x,\y) {};

\foreach \x in {-10,-9,-8,-7,-6,-5,-4,-3,-2,-1,0,1,2,3,4,5,6,7,8,9}{
\foreach \y in {0,1}{
    \pgfmathparse{rnd}
    \let\dummynum=\pgfmathresult
    \ifdim\pgfmathresult pt < \p pt\relax \draw (\x,\y) -- (\x+1,\y);\fi
  }}
  
\foreach \x in {-10,-9,-8,-7,-6,-5,-4,-3,-2,-1,0,1,2,3,4,5,6,7,8,9,10}{
\foreach \y in {0}{
    \pgfmathparse{rnd}
    \let\dummynum=\pgfmathresult
    \ifdim\pgfmathresult pt < \p pt\relax \draw (\x,\y) -- (\x,\y+1);\fi
  }}

	\draw[densely dotted] (-10.5,1) -- (-10,1);
	\draw[densely dotted] (-10.5,0) -- (-10,0);
	\draw (-8,1) -- (-7,1);
	\draw (-3,1) -- (-2,1);
	\draw (1,1) -- (2,1);
	\draw (6,0) -- (7,0);
	\draw (9,0) -- (10,0);
	\draw[densely dotted] (10,0) -- (10.5,0);
	\draw[densely dotted] (10,1) -- (10.5,1);	
	\node[draw=none,fill=none] at (0,-0.75) {$\mathbf{0}$};
	\node[draw=none,fill=none] at (6,-0.75) {$R^{\mathrm{pre}}_{0}$};
	\node[draw=none,fill=none] at (8,-0.75) {$R^{\mathrm{pre}}_{1}$};

	\draw (9,1) -- (10,1);
	\draw[white, thick] (-5,1) -- (-4,1);
	\foreach \x in {-10,-9,-8,-7,-6,-5,-4,-3,-2,-1,0,1,2,3,4,5,6,7,8,9,10}
\foreach \y in {0,1}
    \node at (\x,\y) {};
    \node[draw=none,fill=none] at (-5,-0.75) {$R^{\mathrm{pre}}_{-1}$};
\end{tikzpicture}
\end{center}
We denote the subgraph of $\omega$ with vertex set $\{v \in V: a \leq \x(v) \leq b\}$
and edge set $\{e \in E^{a,\geq} \cap E^{b,<}: \omega(e)=1\}$
by $[a,b)$ and call $[a,b)$ a \emph{piece} or \emph{block} (of $\omega$).
The pre-regeneration points split the percolation cluster into blocks
\begin{center}
$\omega_n \defeq [\x(R_{n-1}^{\mathrm{pre}}), \x(R_n^{\mathrm{pre}})),	\quad	n \in \Z.$
\end{center}
The notation suggests that there are infinitely many pre-regeneration points to the left and right of $0$.
This is indeed the case and will be shown below.

Further, we call a piece $[a,b)$ with $a<b$ a \emph{trap piece} (in $\omega$) if it has the following properties:
\begin{itemize}
	\item[(i)]	the vertical edge $\langle(a,0),(a,1)\rangle$ is open, while all other vertical edges in $[a,b+1)$ are closed;
	\item[(ii)]	all horizontal edges in $[a,b)$ are open;
	\item[(iii)]	exactly one of the horizontal edges $\langle(b,i),(b+1,i) \rangle$, $i \in \{0,1\}$ is open.
\end{itemize}
We call $b-a$ the \emph{length} of the trap.
If $i$ is such that $\omega(\langle(b,i),(b+1,i)\rangle) = 1$, the vertex $(b+1,i)$ is called the \emph{trap end}.
In this situation, the induced line graph on the vertices $(a,1-i),\ldots,(b,1-i)$ is called \emph{trap} or \emph{dead end}
and the vertex $(a,1-i)$ is called the \emph{entrance of the trap}.
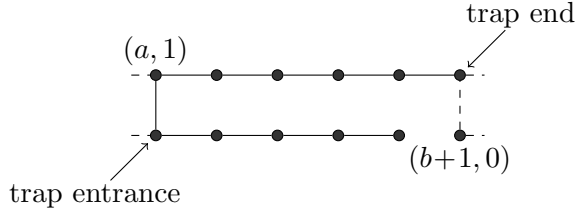
\begin{center}
\begin{tikzpicture}[thin, scale=0.8,-,
                   shorten >=2pt+0.5*\pgflinewidth,
                   shorten <=2pt+0.5*\pgflinewidth,
                   every node/.style={circle,
                                      draw,
                                      fill          = black!80,
                                      inner sep     = 0pt,
                                      minimum width =4 pt}]
\path[draw] 
       node at (0,0) {}  
       node at (0,1) {} 
       node at (1,0) {} 
       node at (1,1) {} 
       node at (2,0) {} 
       node at (2,1) {} 
       node at (3,0) {} 
       node at (3,1) {}
       node at (4,0) {} 
       node at (4,1) {}
       node at (5,0) {} 
       node at (5,1) {} 
        ; 

    \draw (0,0) -- (0,1) ;
    \draw (0,1) -- (1,1) ;
    \draw (0,0) -- (1,0) ;
    \draw (1,1) -- (2,1) ;
    \draw (1,0) -- (2,0) ;
    \draw (2,0) -- (3,0) ; 
    \draw (2,1) -- (3,1);
    \draw (3,0) -- (4,0) ; 
     \draw (3,1) -- (4,1);
    \draw (4,1) -- (5,1);
        
\begin{scope}[dashed]  
     
      \draw (-0.5,0) -- (-0,0)  ;
      \draw (-0.5,1) -- (-0,1)  ;
      
      \draw (5,0) -- (5,1)  ;
      \draw (5,0) -- (5.5,0)  ;
      \draw (5,1) -- (5.5,1);
     
\end{scope}

\node[draw=none,fill=none] at (0,1.4) {$(a,1)$};
\node[draw=none,fill=none] at (5,-0.4) {$(b\!+\!1,0)$};
\node[draw=none,fill=none] at (6,2) {$\mbox{trap end}$};
\begin{scope}   [->,shorten >=4pt+0.5*\pgflinewidth,
                   shorten <=8pt+0.5*\pgflinewidth,
]
	\draw (6,2) -- (5,1);
	\draw (-1,-1) -- (0,0);
	\node[draw=none,fill=none] at (-1,-1) {$\mbox{trap entrance}$};
 \end{scope}
\end{tikzpicture}
\end{center}
\emph{Non-trap pieces} are pieces $[a,b)$ such that every $v \in [a,b) \cap \Cluster_{\infty}$ is forwards communicating.

We enumerate the traps in $\omega$ as follows.
Let $L_1$ be the trap piece
that belongs to the trap entrance with the smallest nonnegative $\x$-coordinate.
We enumerate the remaining trap pieces such that $L_2$ is the next trap piece to the right of $L_1$ etc.
Analogously, $L_0$ is the first trap piece to the left of $L_1$ etc.

\begin{lemma}	\label{Lem:trap probability given vertical edge}
Under $\Prmp^*$, $(({\tt T}_i,\omega(E^i)))_{i \in \Z}$ is a (time-homogeneous) Markov chain with state space $\{{\tt 01},{\tt 10}, {\tt 11}\} \times \{0,1\}^3$.
Further, there exists a constant $\gamma(p) \in (0,1)$ such that, for every $i \in \Z$,
\begin{equation}	\label{eq:trap probability given vertical edge}
\Prmp^*(T_{i:i+m} \mid \omega(\langle(i,0),(i,1)\rangle)=1)	~=~	\gamma(p) e^{-2 \lambdacrit m},	\quad	m \in \N
\end{equation}
where $T_{a:b}$ denotes the event that $[a,b)$ is a trap piece ($a, b \in \Z$, $a<b$).
When $i \geq 0$, then \eqref{eq:trap probability given vertical edge} also holds with $\Prmp^*$ replaced by $\Prmp$.
\end{lemma}
\begin{proof}
From the last statement in Lemma \ref{Lem:T_i,omega(E^i)},
one infers that $(({\tt T}_i,\omega(E^i)))_{i \in \Z}$ is a Markov chain with state space $\{{\tt 01},{\tt 10}, {\tt 11}\} \times \{0,1\}^3$.
This Markov chain can be thought of as follows. Given all information up to and including time $i-1$,
one can first sample the value ${\tt T}_{i}$ using knowledge of the value of ${\tt T}_{i-1}$ only.
Then, independently of everything sampled before,
one can sample the value of $\omega(E^{i})$ from $\mathrm{P}_{\!p,{\tt T}_{i-1},{\tt T}_{i}}$.
Since $\Prmp^*$ is shift-invariant, it is enough to calculate $\lambda_m(p) \defeq \Prmp^*(T_{0:m} \mid \omega(\langle(0,0),(0,1)\rangle)=1)$.
This can be done as in \cite[pp.\;3403-3404]{Axelson-Fisk+H"aggstr"om:2009b}
and leads to
\begin{equation*}
\lambda_m(p)	=	\gamma(p) \left(\frac{1}{2} \left(1+2p-2p^2-\sqrt{1+4p^2-8p^3+4p^4} \right) \right)^{\!\!m}
=	\gamma(p) e^{-2 \lambdacrit m}
\end{equation*}
where $\gamma(p) = \Prmp^*(C_1 | {\tt T}_0 = {\tt 11}) \in (0,1)$
and $C_1$ is the event that precisely one of the horizontal edges with right endpoint at $\x$-coordinate $1$ is open,
while the other one and the vertical connection between $(1,0)$ and $(1,1)$ are closed.

\noindent
Finally, assume that $i \geq 0$.
Then \eqref{Lem:trap probability given vertical edge} for $\Prmp$
follows from the Markov property under $\Prmp$ at time $i$ for $(({\tt T}_j,\omega(E^j)))_{j \in \Z}$.
\end{proof}

For the formulation of the next lemma, we introduce the shift operators.
For $v \in V$, the shift $\theta^{v}$ is the translation possibly combined with a flip of the $\y$-coordinate
 that maps $v \in V$ to $\mathbf{0}$
and, in general, $w \in V$ to $(\x(w)-\x(v),\y(w)-\y(v))$.
The shift $\theta^{v}$ canonically extends to a mapping on the set of edges
and hence to a mapping on the configuration space $\Omega$.
For convenience, we denote all these mappings by $\theta^{v}$.
The mappings $\theta^{v}$ form a commutative group since $\theta^{v} \theta^{w} = \theta^{v+w}$
where addition $v+w$ is to be understood in $\Z \times \Z_2$. In particular, $(\x(v),1)+(\x(w),1)= (\x(v)+\x(w),0)$.

Next define
\begin{equation*}
E' \defeq \bigcup_{i \in \N} \{0,1\}^{E^{0, \geq} \cap E^{i,<}}
\quad	\text{and}	\quad
E^{\mathbf{0}} \defeq \bigcup_{i \in \N, \, j \in \N_0} \{0,1\}^{E^{-i, \geq} \cap E^{j,<}}.
\end{equation*}
The $\theta^{R^{\mathrm{pre}}_{n-1}} \omega_n$, $n \not = 0$ can be considered as random variables taking values in $E'$,
while $\omega_0$ is a random variable taking values in $E^{\mathbf{0}}$.
Let $C_0$ be the set of finite configurations $\eta \in E^{\mathbf{0}}$
for which $\mathbf{0}$ is on an open path connecting the left and right endpoints with $\y$-coordinate $0$ in $\eta$.
Then $\Prmp(\cdot) = \Prmp^*(\cdot \cap \{\omega_0 \in C_0\}) / \Prmp^*(\omega_0 \in C_0)$.

\begin{lemma}	\label{Lem:pre-regeneration point decomposition}
The following assertions hold true:
\begin{itemize}
	\item[(a)]
		With $\Prmp^*$-probability one, there are infinitely many pre-regeneration points to the right and to the left of zero.
	\item[(b)]
		There exists some $c = c(p) \in (0,1)$ with $\Prmp^*(\x(R^{\mathrm{pre}}_{1})-\x(R^{\mathrm{pre}}_{0}) > k) \leq c^k$ for all $k \in \N_0$.
	\item[(c)]
		Under $\Prmp^*$,
		$((\theta^{R^{\mathrm{pre}}_{n-1}} \omega_n, \x(R^{\mathrm{pre}}_{n})-\x(R^{\mathrm{pre}}_{n-1})))_{n \in \Z \setminus \{0\}}$
		is a family of i.i.d.~random variables independent of $\omega_0$.
\end{itemize}
All assertions also hold with $\Prmp^*$ replaced by $\Prmp$.
Further, the distribution of $$((\theta^{R^{\mathrm{pre}}_{n-1}} \omega_n, \x(R^{\mathrm{pre}}_{n})-\x(R^{\mathrm{pre}}_{n-1})))_{n \in \Z \setminus \{0\}}$$
under $\Prmp$ is the same as under $\Prmp^*$.
\end{lemma}
\begin{proof}
For the proof of this lemma, we consider the following auxiliary stochastic process
$(({\tt T}_i,\eta_i))_{i \in \Z} = (({\tt T}_i,\omega(E^{i-1,>} \cap E^{i+1,<})))_{i \in \Z}$.
At time $i$, it contains the information 
which of the vertices with $\x$-coordinate $i$ are backwards communicating,
encoded by the value of ${\tt T}_i$,
plus the information which edges adjacent to the vertices with $\x$-coordinate $i$ are open, encoded by the value of $\eta_i$. This process is a Markov chain.
Notice that $(({\tt T}_i,\eta_i))_{i \in \Z}$ has a finite state space
and that $(i,0)$ being a pre-regeneration point is equivalent to ${\tt T}_i = {\tt 10}$
and $\eta_i$ taking the particular value displayed in the figure below.
\\~
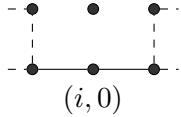
\begin{center}\begin{tikzpicture}[thin, scale=0.8,-,
                   shorten >=2pt+0.5*\pgflinewidth,
                   shorten <=2pt+0.5*\pgflinewidth,
                   every node/.style={circle,
                                      draw,
                                      fill          = black!80,
                                      inner sep     = 0pt,
                                      minimum width =4 pt}]
\foreach \x in {-1,0,1}
\foreach \y in {0,1}
	\node at (\x,\y) {}; 

	\draw (-1,0) -- (0,0);
	\draw (0,0) -- (1,0);
        
\begin{scope}[dashed]  
	\draw (-1.5,0) -- (-1,0);
	\draw (-1.5,1) -- (-1,1);
	\draw (-1,0) -- (-1,1);
	\draw (1,0) -- (1,1);
	\draw (1,0) -- (1.5,0);
	\draw (1,1) -- (1.5,1);
\end{scope}

\begin{scope}   [->,shorten >=4pt+0.5*\pgflinewidth,
                   shorten <=8pt+0.5*\pgflinewidth]
	\node[draw=none,fill=none] at (0,-0.5) {$(i,0)$};
 \end{scope}
\end{tikzpicture}
\end{center}
As this state is an accessible state for the chain and as the state space is finite, the chain hits it infinitely often, proving (a).
Further, a standard geometric trials argument gives (b).
Assertion (c) follows from the fact that the cycles between successive visits of a given state by the Markov chain
$(({\tt T}_i,\eta_i))_{i \in \Z}$ are i.i.d.
At first, this argument only applies to the cycles $\omega_1,\omega_2,\ldots$
and then extends by reflection ($\Prmp^*$ is symmetric by construction)
also to those that are on the negative half-axis.
The cycle straddling the origin still is independent of the other cycles by the Markov property,
but may have a different distribution.

Finally, one checks that (a), (b) and (c) hold with $\Prmp^*$ replaced by $\Prmp$.
\end{proof}

Using regeneration-time arguments will make it necessary at some points
to use a different percolation law than $\Prmp$ or $\Prmp^*$,
namely, the cycle-stationary percolation law $\Prmp^\circ$,
which is defined below.

\begin{definition}	\label{def:cycle-stationary percolation law}
The \emph{cycle-stationary percolation law} $\Prmp^\circ$
is defined to be the unique probability measure on $(\Omega,\F)$ such that the
cycles $\omega_n$, $n \in \Z$ are i.i.d.\ under $\Prmp^\circ$
and such that each $\omega_n$ has the same law under $\Prmp^\circ$
as $\omega_1$ under $\Prmp^*$.
\end{definition}

\subsection{The traps.}
The biased random walk will pass non-trap pieces in linear time,
while in traps, it will spend more time.
In the next step, we investigate the lengths of traps.
Let $\ell_n$ denote the length of the trap $L_n$, $n \in \Z$.

\begin{lemma}	\label{Lem:trap cost}
~
\begin{itemize}
\item[(a)] Under $\Prmp^*$, $(\ell_n)_{n \not = 0}$ is a family of i.i.d.\ nonnegative random variables independent of $\ell_0$ with $\Prmp^*(\ell_1 = m) = (e^{2 \lambdacrit}-1) e^{-2 \lambdacrit m}$, $m \in \N$.
	\item[(b)]
		There is a constant $\chi(p)$ such that $\Prmp^*(\ell_0 = m) \leq \chi(p) m e^{-2 \lambdacrit m}$, $m \in \N$.
\end{itemize}
\end{lemma}
\begin{proof}
Each trap begins at an open vertical edge.
By the strong Markov property, $(({\tt T}_i,\omega(E^i)))_{i \in \Z}$ starts afresh at every open vertical edge.
By \eqref{eq:trap probability given vertical edge},
the probability of having a trap of length $m$ following an open vertical edge is proportional to $e^{-2\lambdacrit m}$.
This implies assertion (a).

Assertion (b) is reminiscent of the fact that the distribution of the length of the cycle straddling the origin
in a two-sided renewal process is the size-biasing of the distribution of any other cycle.
This result is not directly applicable, but standard arguments
yield the estimate in (b).
\end{proof}

For later use, we derive an upper bound on the probability under the cycle-stationary percolation law
of the event that a certain piece of the ladder is part of a trap.

\begin{lemma}	\label{Lem:trap cover}
For $k,m \in \N_0$, $m > 0$, let $T'_{k:k+m}$ be the event that the piece $[k,k+m)$ is contained in a trap piece.
Then $\Prmp^{\circ}(T'_{k:k+m}) \leq e^{-2 \lambdacrit m}$.
\end{lemma}  
\begin{proof}
Notice that $T'_{k:k+m} \subseteq \{{\tt T}_k = {\tt 11}\} \cap \bigcap_{j=1}^m B_{k+j}$
where $B_j$ is the event that $\omega(\langle(j-1,i),(j,i)\rangle) = 1$ for $i=0,1$ and $\omega(\langle(j,0),(j,1)\rangle)=0$, $j \in \Z$.
Hence, arguing as in \cite[pp.\;3403--3404]{Axelson-Fisk+H"aggstr"om:2009b}, we obtain
\begin{equation*}	\textstyle
\Prmp^\circ(T'_{k:k+m})
\leq \Prmp^\circ({\tt T}_k = {\tt 11}) \Prmp^*\big(\bigcap_{j=1}^m B_{k+j} \,\big|\, {\tt T}_k = {\tt 11} \big)
\leq e^{-2 \lambdacrit m}.
\end{equation*}
\end{proof}

\section{Regeneration arguments}		\label{sec:biased random walk}

Throughout this section, we fix a bias $\lambda > 0$.
Hence, under $\Prob_{\lambda}$, $X_n \to \infty$ a.\,s.~as $n \to \infty$.
To deduce a central limit theorem or a Marcinkiewicz-Zygmund-type strong law for $X$,
information is needed about the time the walk spends in initial pieces of the percolation cluster.
To investigate these times, we introduce some additional terminology.

\subsection{The backbone.}

We call the subgraph $\mathcal{B}$ of the infinite cluster induced by all forwards communicating states the \emph{backbone}.
The backbone is obtained from $\Cluster_{\infty}$ by deleting the dead ends of all trap pieces.
Clearly, $\mathcal{B}$ is connected and contains all pre-regeneration points.

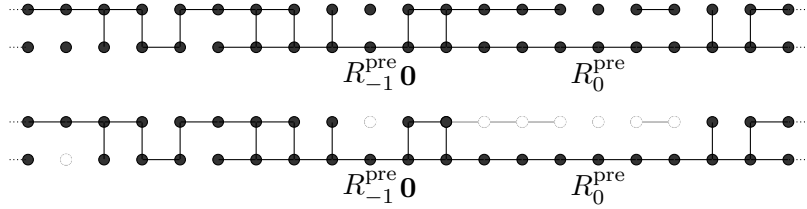
\begin{figure}[h]
\begin{center}
\pgfmathsetseed{216642}
\begin{tikzpicture}[thin, scale=0.5,-,
                   shorten >=0pt+0.5*\pgflinewidth,
                   shorten <=0pt+0.5*\pgflinewidth,
                   every node/.style={circle,
                                      draw,
                                      fill          = black!80,
                                      inner sep     = 0pt,
                                      minimum width =4 pt}]

\def \p {0.5}

\foreach \x in {-10,-9,-8,-7,-6,-5,-4,-3,-2,-1,0,1,2,3,4,5,6,7,8,9,10}
\foreach \y in {0,1}
    \node at (\x,\y) {};

\foreach \x in {-10,-9,-8,-7,-6,-5,-4,-3,-2,-1,0,1,2,3,4,5,6,7,8,9}{
\foreach \y in {1,0}{
    \pgfmathparse{rnd}
    \let\dummynum=\pgfmathresult
    \ifdim\pgfmathresult pt < \p pt\relax \draw (\x,\y) -- (\x+1,\y);\fi
  }}
  
\foreach \x in {-10,-9,-8,-7,-6,-5,-4,-3,-2,-1,0,1,2,3,4,5,6,7,8,9,10}{
\foreach \y in {0}{
    \pgfmathparse{rnd}
    \let\dummynum=\pgfmathresult
    \ifdim\pgfmathresult pt < \p pt\relax \draw (\x,\y) -- (\x,\y+1);\fi
  }}

	\draw	(-9,1) -- (-8,1);
	\draw	(-7,0) -- (-6,0);
	\draw	(-1,0) -- (0,0);
	\draw	(3,0) -- (4,0);
	\draw	(8,0) -- (9,0);

	\node[draw=none,fill=none] at (-1,-0.75) {$R^{\mathrm{pre}}_{-1}$};
	\node[draw=none,fill=none] at (0,-0.75) {$\mathbf{0}$};
	\node[draw=none,fill=none] at (5,-0.75) {$R^{\mathrm{pre}}_{0}$};

	\draw[densely dotted] (-10.5,0) -- (-10,0);
	\draw[densely dotted] (-10.5,1) -- (-10,1);
	\draw[densely dotted] (10,0) -- (10.5,0);
	\draw[densely dotted] (10,1) -- (10.5,1);
\end{tikzpicture}\smallskip

\pgfmathsetseed{216642}
\begin{tikzpicture}[thin, scale=0.5,-,
                   shorten >=0pt+0.5*\pgflinewidth,
                   shorten <=0pt+0.5*\pgflinewidth,
                   every node/.style={circle,
                                      draw,
                                      fill          = black!80,
                                      inner sep     = 0pt,
                                      minimum width =4 pt}]

\def \p {0.5}

\foreach \x in {-10,-9,-8,-7,-6,-5,-4,-3,-2,-1,0,1,2,3,4,5,6,7,8,9,10}
\foreach \y in {0,1}
    \node at (\x,\y) {};

\foreach \x in {-10,-9,-8,-7,-6,-5,-4,-3,-2,-1,0,1,2,3,4,5,6,7,8,9}{
\foreach \y in {1,0}{
    \pgfmathparse{rnd}
    \let\dummynum=\pgfmathresult
    \ifdim\pgfmathresult pt < \p pt\relax \draw (\x,\y) -- (\x+1,\y);\fi
  }}
  
\foreach \x in {-10,-9,-8,-7,-6,-5,-4,-3,-2,-1,0,1,2,3,4,5,6,7,8,9,10}{
\foreach \y in {0}{
    \pgfmathparse{rnd}
    \let\dummynum=\pgfmathresult
    \ifdim\pgfmathresult pt < \p pt\relax \draw (\x,\y) -- (\x,\y+1);\fi
  }}

	\draw	(-9,1) -- (-8,1);
	\draw	(-7,0) -- (-6,0);
	\draw	(-1,0) -- (0,0);
	\draw	(3,0) -- (4,0);
	\draw	(8,0) -- (9,0);
	
	\node[draw=none,fill=none] at (-1,-0.75) {$R^{\mathrm{pre}}_{-1}$};
	\node[draw=none,fill=none] at (0,-0.75) {$\mathbf{0}$};
	\node[draw=none,fill=none] at (5,-0.75) {$R^{\mathrm{pre}}_{0}$};

	\draw[densely dotted] (-10.5,0) -- (-10,0);
	\draw[densely dotted] (-10.5,1) -- (-10,1);
	\draw[densely dotted] (10,0) -- (10.5,0);
	\draw[densely dotted] (10,1) -- (10.5,1);
	
	\draw[white]	(1,1) -- (4,1);
	\draw[white]	(6,1) -- (7,1);
	\node at (1,1) {};
	\node[white] at (-9,0) {};

\foreach \x in {-1,2,3,4,5,6,7}
\foreach \y in {1}
    \node[white] at (\x,\y) {};
\end{tikzpicture}
\end{center}
\caption{The original percolation configuration and the backbone}	\label{Fig:backbone}
\end{figure}

\noindent
Let $(Z_0,Z_1,\ldots)$ be the agile walk corresponding to the walk $(Y_0,Y_1,\ldots)$,
that is, the walk obtained from $(Y_0,Y_1,\ldots)$ by removing all times at which the walk stays put.
Further, let $(Z_0^{\mathcal{B}},Z_1^{\mathcal{B}},\ldots)$ be the walk that is obtained from $(Z_0,Z_1,\ldots)$
by removing all steps in which the walk moves to or from a point outside $\mathcal{B}$.
By the strong Markov property, $(Z_n)_{n \geq 0}$ and $(Z_n^{\mathcal{B}})_{n \geq 0}$
are Markov chains on $\Cluster$ and $\mathcal{B}$, respectively,
under $P_{\omega,\lambda}$ for every $\omega \in \Omega_{\mathbf{0}}$ with $\mathbf{0} \in \mathcal{B}$.

\subsection{Regeneration points and times.}

Let $\mathcal{R}^{\mathrm{pre}} \defeq \{R_n^{\mathrm{pre}}: n \in \N_0\}$
denote the (random) set of all pre-regeneration points 
strictly to the right of $\x$-coordinate $0$.
A member of $\mathcal{R}^{\mathrm{pre}}$ is called a \emph{regeneration point}
if it is visited by the random walk $(Y_n)_{n \geq 0}$ precisely once.
The set of regeneration points will be denoted by $\mathcal{R} \subseteq \mathcal{R}^{\mathrm{pre}}$.
Let $R_0 \defeq \mathbf{0}$ and $R_1, R_2, \ldots$ be an enumeration of the regeneration points with increasing $\x$-coordinates.
Define $\tau_0 \defeq 0$ and, for $n \in \N$,  and let $\tau_n$ be the unique time at which $Y$ visits $R_n$.
Formally, the $\tau_n$ and \nina{$R_n$}, $n \in \N$ are given by:
\begin{equation}
\tau_n	\defeq	\inf\{k > \tau_{n-1}: Y_k \in \mathcal{R}^{\mathrm{pre}}, \, Y_j \not = Y_k \text{ for all } j \not = k \},
\quad	R_n	\defeq	Y_{\tau_n}.	\label{eq:tau_n}
\end{equation}
Since $\lambda>0$, the random walk is transient to the right.
This ensures that the $\tau_n$, $n \in \N_0$ are almost surely finite and form an increasing sequence.
The $\tau_{n}$, $n \in \N$ are no stopping times.
However, there is an analogue of the strong Markov property.
In order to formulate it, let $\rho_n \defeq \x(R_n)$  and denote by
\begin{equation*}
\H_n	~\defeq~	\sigma(\tau_1,\ldots, \tau_n, Y_0, \ldots, Y_{\tau_n}, \, \omega(\langle v,w\rangle):\;\x(v) < \rho_n, \, \x(w) \leq \rho_n)
\end{equation*}
the $\sigma$-field of the walk up to time $\tau_n$ and the environment up to $\rho_n$.
Further, for $e \in E$, let $p_e: \Omega \to \{0,1\}$, $\omega \mapsto \omega(e)$, and
\begin{equation*}
\F_{\geq}	~\defeq~	\sigma(p_{\langle v,w\rangle}:\, \x(v), \x(w) \geq 0).
\end{equation*}

\begin{lemma}	\label{Lem:iid regeneration times and points}	
For every $n \in \N$ and all measurable sets $F\! \in\! \F_{\geq}$, $G\! \in \! \G$,
we have
\begin{align}
\Prob_{\lambda} & ((\theta^{R_n} Y_{\tau_n+k})_{k \geq 0} \in G,\,\theta^{R_n} \omega \in F \mid \H_{n})	\notag	\\
&= \Prob^{\circ}_{\lambda}( (Y_{k})_{k \geq 0} \in G,\,\omega \in F \mid Y_k \neq \mathbf{0} \text{ for all } k \geq 1)	\label{eq:iid regeneration times and points}
\end{align}
where $\Prob^{\circ}_{\lambda} = \Prmp^{\circ} \times P_{\omega,\lambda}$.
In particular, the $(\tau_{n+1}-\tau_n, \rho_{n+1}-\rho_{n})$, $n \in \N$ are i.i.d.~pairs of random variables under $\Prob_{\lambda}$.
\end{lemma}

The proof is similar to the proof of Proposition 1.3 in \cite{Sznitman+Zerner:1999},
we refrain from providing details here.
The key result concerning the regeneration times is the following lemma, which is proved in Section \ref{sec:regeneration} below.

\begin{lemma}	\label{Lem:moments of rho and tau}
The following assertions hold:
\begin{itemize}
	\item[(a)]
		For every $\lambda > 0$, there exists some $\varepsilon > 0$ such that $\E_{\lambda}[e^{\varepsilon (\rho_2-\rho_1)}] < \infty$.
	\item[(b)]
		Let $\kappa \geq 1$.
		Then $\E_{\lambda}[(\tau_2-\tau_1)^{\kappa}] < \infty$ iff $\kappa < \frac{\lambdacrit}{\lambda}$.
\end{itemize}
\end{lemma}

\subsection{The Marcinkiewicz-Zygmund-type strong law.}

We now give a proof of Theorem \ref{Thm:Marcinkiewicz-Zygmund}
based on Lemmas \ref{Lem:iid regeneration times and points} and \ref{Lem:moments of rho and tau}.
For the reader's convenience, we restate the result here in a slightly extended version.

\begin{proposition}	\label{Prop:SLLN&MZlaw}
Let $p \in (0,1)$.
\begin{itemize}
	\item[(a)]
		If $\lambda > 0$, then
		\begin{equation}	\label{eq:velocity}	\textstyle
		\frac{X_n}{n} \to \frac{\E_{\lambda}[\rho_2-\rho_1]}{\E_{\lambda}[\tau_2-\tau_1]}	~\eqdef~	\vel(\lambda)
		\quad	\Prob_{\lambda}\text{-a.\,s.\ as } n \to \infty.
		\end{equation}
		In particular, $\vel(\lambda) > 0$ iff $\lambda<\lambdacrit$ and $\vel(\lambda) = 0$ iff $\lambda \geq \lambdacrit$.
	\item[(b)]	
		If $\lambda \in (0,\lambdacrit)$ and $1 < r < \frac{\lambdacrit}{\lambda} \wedge 2$, then
		\begin{equation*}	\tag{\ref{eq:Marcinkiewicz-Zygmund}}	\textstyle
		\frac{X_n-n\vel(\lambda)}{n^{1/r}}	\to 0
		\quad	\text{and}	\quad
		n^{-1/r} M^{\lambda}_n	\to 0
		\end{equation*}
		where the convergence in \eqref{eq:Marcinkiewicz-Zygmund} holds $\Prob_{\lambda}$-a.\,s.\ and in $L^r(\Prob_{\lambda})$.
\end{itemize}
\end{proposition}

Part (a) of this proposition implies Proposition \ref{Prop:SLLN}, part (b) implies Theorem \ref{Thm:Marcinkiewicz-Zygmund}.
A different formula for $\vel(\lambda)$ was given in \cite[p.\;3412]{Axelson-Fisk+H"aggstr"om:2009b}.

\begin{proof}
Let $\lambda > 0$. Further, let $r \in (1,\frac{\lambdacrit}{\lambda} \wedge 2)$ if $\lambda < \lambdacrit$, and $r=1$, otherwise.
By Lemmas \ref{Lem:iid regeneration times and points} and \ref{Lem:moments of rho and tau},
$(\rho_{n+1}-\rho_n)_{n \in \N}$ and $(\tau_{n+1}-\tau_n)_{n \in \N}$ are sequences of i.i.d.\ nonnegative random variables
with $\E_{\lambda}[(\rho_{2}-\rho_1)^r]<\infty$, $\E_{\lambda}[(\tau_{2}-\tau_1)^r] < \infty$ if $\lambda < \lambdacrit$
and $\E_{\lambda}[\tau_{2}-\tau_1] = \infty$ if $\lambda \geq \lambdacrit$.
The Marcinkiewicz-Zygmund strong law \cite[Theorems 6.7.1 and 6.10.3]{Gut:2005}
applied to $(\rho_{n+1}-\rho_{1})_{n \in \N_0}$, yields
\begin{equation}	\label{eq:MZ-law rho}
\frac{\rho_n - n\E_{\lambda}[\rho_2-\rho_1]}{n^{1/r}} ~\to~	0	\quad	\Prob_{\lambda} \text{-a.\,s.\ and in } L^r(\Prob_{\lambda}) \text{ as } n \to \infty.
\end{equation}
Analogously, if $\lambda < \lambdacrit$,
\begin{equation}	\label{eq:MZ-law tau}
\frac{\tau_n - n\E_{\lambda}[\tau_2-\tau_1]}{n^{1/r}} ~\to~	0	\quad	\Prob_{\lambda} \text{-a.\,s.\ and in } L^r(\Prob_{\lambda}) \text{ as } n \to \infty,
\end{equation}
while in any case, we have
\begin{equation}	\label{eq:SLLN tau}
\frac{\tau_n}{n} ~\to~	\E_{\lambda}[\tau_2-\tau_1]	\quad	\Prob_{\lambda} \text{-a.\,s.\ as } n \to \infty
\end{equation}
even in the case $\E_{\lambda}[\tau_2-\tau_1]=\infty$.
Define $\vel(\lambda) \defeq \E_{\lambda}[\rho_2-\rho_1]/\E_{\lambda}[\tau_2-\tau_1]$
and $k(n) \defeq \max\{k \in \N_0:\,\tau_{k} \leq n\}$. Clearly, $k(n) \to \infty$ as $n \to \infty$.
Further,
\begin{equation}	\label{eq:SLLN renewal counting}
\frac{k(n)}{n}	~\to~	\frac{1}{\E_{\lambda}[\tau_{2} - \tau_{1}]}	\quad	\Prob_{\lambda}\text{-a.\,s.\ as } n \to \infty
\end{equation}
by the strong law of large numbers for renewal counting processes.
Set $\nu(n) \defeq k(n)+1$.
Then $\nu(n)$ is a stopping time with respect to the canonical filtration of $((\tau_k,\rho_k))_{k \in \N_0}$
and $\nu(n) \leq n+1$.
Hence, the family $(\nu(n)/n)_{n \in \N}$ is uniformly integrable.
Thus \cite[Theorem 1.6.2]{Gut:2009} implies that
\begin{equation}	\label{eq:uirho_nu(n)}
\bigg(\bigg|\frac{\rho_{\nu(n)}-\nu(n) \E_{\lambda}[\rho_2-\rho_1]}{n^{1/r}}\bigg|^r\bigg)_{n \in \N_0}
\quad	\text{is uniformly integrable.}\!\footnote{
Notice that $\rho_1$ may have a different distribution under $\Prob_{\lambda}$
than the other increments $\rho_{n+1}-\rho_n$, $n \in \N$.
However, only minor changes are necessary to apply the results from \cite{Gut:2009}
anyway. This comment applies several times in this proof.}
\end{equation}
We write
\begin{align*}	\label{eq:decomposition for speed2}
&\frac{X_n-n\vel(\lambda)}{n^{1/r}}	\\
&~= \frac{X_n-\rho_{\nu(n)}}{n^{1/r}} + \frac{\rho_{\nu(n)}-\nu(n)\E_{\lambda}[\rho_2-\rho_1]}{n^{1/r}} + 
\frac{\nu(n)\E_{\lambda}[\rho_2-\rho_1] - n \vel(\lambda)}{n^{1/r}}.
\end{align*}
The absolute value of the first summand is bounded by
$(\rho_{\nu(n)}-\rho_{k(n)})/n^{1/r}$, which tends to $0$ $\Prob_\lambda$-a.\,s.\ and in $L^r(\Prob_{\lambda})$
by \cite[Theorem 1.8.1]{Gut:2009}.
The second summand tends to $0$ $\Prob_{\lambda}$-a.\,s.\ and in $L^r(\Prob_\lambda)$
by \eqref{eq:MZ-law rho}, \eqref{eq:SLLN renewal counting} and \eqref{eq:uirho_nu(n)}.
Further, we find that if $\lambda \geq \lambdacrit$, i.\,e., $r=1$,
then the third summand tends to $0$ $\Prob_{\lambda}$-a.\,s.\ by \eqref{eq:SLLN renewal counting}.
If $\lambda \in (0,\lambdacrit)$, then
\begin{equation}	\label{eq:estimate}	\textstyle
\left| \frac{\nu(n)\E_{\lambda}[\rho_2-\rho_1] - n \vel(\lambda)}{n^{1/r}}\right|
\leq\left| \vel(\lambda) \frac{\nu(n)\E_{\lambda}[\tau_2-\tau_1] - \tau_{\nu(n)}}{n^{1/r}}\right| + \frac{\vel(\lambda)(\tau_{\nu(n)}-\tau_{k(n)})}{n^{1/r}}.
\end{equation}
The first summand converges to $0$ $\Prob_{\lambda}$-a.\,s.\ by \eqref{eq:MZ-law tau} and \eqref{eq:SLLN renewal counting}.
A subsequent application of \cite[Theorem 1.6.2]{Gut:2009} guarantees
that this convergence also holds in $L^r(\Prob_{\lambda})$.
The second summand is bounded above by $\vel(\lambda)(\tau_{\nu(n)}-\tau_{k(n)})/n^{1/r}$,
which tends to $0$ $\Prob_\lambda$-a.\,s.\ and in $L^r(\Prob_{\lambda})$
again by \cite[Theorem 1.8.1]{Gut:2009}.

For the proof of the statement concerning $M^{\lambda}_n$ in \eqref{eq:Marcinkiewicz-Zygmund}, recall \eqref{Mdef} and define
\begin{equation*}
\eta_n
\defeq \nu_{\omega,\lambda}(Y_{\tau_{n-1}},Y_{\tau_{n-1}+1}) + \ldots + \nu_{\omega,\lambda}(Y_{\tau_{n}-1},Y_{\tau_{n}})
= M^{\lambda}_{\tau_n}-M^{\lambda}_{\tau_{n-1}}
\end{equation*}
for $n \in \N$.
The $\eta_n$, $n \geq 2$ are i.i.d.\ by Lemma \ref{Lem:iid regeneration times and points}.
There is a constant $C > 0$ such that $\sup_{\omega,v,w} |\nu_{\omega,\lambda}(v,w)| \leq C$.
As a consequence,
\begin{equation*}
|\nu_{\omega,\lambda}(Y_{\tau_{n-1}},Y_{\tau_{n-1}+1})| + \ldots + |\nu_{\omega,\lambda}(Y_{\tau_{n}-1},Y_{\tau_{n}})|
\leq C (\tau_n-\tau_{n-1})
\end{equation*}
for all $n \in \N$. Hence,
\begin{equation*}
M^{\lambda}_{\tau_{k(n)}} - C (\tau_n-\tau_{n-1}) \leq M^{\lambda}_n \leq M^{\lambda}_{\tau_{k(n)}} + C (\tau_n-\tau_{n-1})
\end{equation*}
for all $n \in \N$.
Similar arguments as those used for $X_n - n \vel(\lambda)$ now yield the second limit relation in \eqref{eq:Marcinkiewicz-Zygmund}.
\end{proof}

\subsection{The invariance principle.}

We now give a proof of Theorem \ref{Thm:joint CLT} based on regeneration times.
The same technique has been used e.\,g. in the proofs of Theorem 4.1 in \cite{Rassoul-Agha+Seppalainen:2006}
and Theorem 4.1 in \cite{Sznitman:2000}.

\begin{proof}[Proof of Theorem \ref{Thm:joint CLT}]
Assume that $\lambda \in (0,\lambdacrit/2)$. Then $\vel = \vel(\lambda) > 0$ by Proposition \ref{Prop:SLLN}.
For $n \in \N$, let
\begin{equation*}
\xi_n	~\defeq~	(\rho_{n}-\rho_{n-1}) - (\tau_{n}-\tau_{n-1}) \vel	
~=~	(X_{\tau_{n}}-X_{\tau_{n-1}}) - (\tau_{n}-\tau_{n-1}) \vel
\end{equation*} 
and, as in the proof of Proposition \ref{Prop:SLLN},
\begin{equation*}
\eta_n
\defeq \nu_{\omega,\lambda}(Y_{\tau_{n-1}},Y_{\tau_{n-1}+1}) + \ldots + \nu_{\omega,\lambda}(Y_{\tau_{n}-1},Y_{\tau_{n}})
= M^{\lambda}_{\tau_n}-M	^{\lambda}_{\tau_{n-1}}
\end{equation*}
According to Lemma \ref{Lem:iid regeneration times and points},
$((\xi_n,\eta_n))_{n \geq 2}$ is a sequence of centered 2-dimensional i.i.d.\ random variables.
Due to Lemma \ref{Lem:moments of rho and tau}
and since $\nu_{\cdot,\lambda}(\cdot,\cdot)$ is uniformly bounded,
the covariance matrix $\tilde{\Sigma}^{\lambda}$ of $(\xi_2,\eta_2)$ has finite entries only.
Moreover, $\E_{\lambda}[\xi_{2}^2] > 0$ and $\E_{\lambda}[\eta_2^2] > 0$ since clearly, $\xi_2$ and $\eta_2$ are not a.\,s.\ constant.
Define $S_0 \defeq (0,0)$ and 
\begin{equation}	\label{eq:S_n}
S_{n}	~\defeq~	(\xi_{1},\eta_1)+\ldots+(\xi_n,\eta_n)	~=~	(X_{\tau_n}- \tau_n \vel, M^{\lambda}_{\tau_n})	\quad	n \in \N.
\end{equation}
Since the contribution of the first term $(\xi_1,\eta_1)$ is negligible as $n \to \infty$,
Donsker's invariance principle \cite[Theorem 14.1]{Billingsley:1999} implies that
\begin{equation}	\label{eq:invariance principle S_n}
(n^{-1/2} S_{\lfloor nt \rfloor})_{0 \leq t \leq 1}
~\underset{n \to \infty}{\Rightarrow}~ (\tilde{B}^{\lambda},\tilde{M}^{\lambda})
\end{equation}
in the Skorokhod space $D[0,1]$ for a two-dimensional centered Brownian motion with covariance matrix $\tilde{\Sigma}^{\lambda}$.
For $u \geq 0$, let $k(u) = \max\{k \in \N_0:\,\tau_{k} \leq u\}$.
By monotonicity, $\lim_{n \to \infty} \frac{n}{k(n)} = \E_{\lambda}[\tau_{2} - \tau_{1}]$ $\Prob_{\lambda}$-a.\,s.\ extends to
\begin{equation}	\label{eq:supk(n)/n}
\sup_{0 \leq t \leq 1} \left| \frac{k(nt)}{n} - \frac{t}{\E_{\lambda} [\tau_2-\tau_1]}	\right|	~\underset{n \to \infty}{\to}~	0	\quad	\Prob_{\lambda} \text{-a.\,s.}
\end{equation}
The idea is to use \eqref{eq:supk(n)/n} to transfer \eqref{eq:invariance principle S_n}
to $(n^{-1/2} S_{k(nt)})_{0 \leq t \leq 1}$ (Step 1).
Then we show that the latter process is close to $(B_n,n^{-1/2}M^{\lambda}_n)$
and, thereby, establish the convergence of $(B_n,n^{-1/2}M^{\lambda}_n)$ (Step 2).\smallskip

\noindent
\emph{Step 1:}
As Brownian motion has almost surely continuous paths,
convergence to Brownian motion in the Skorokhod space implies convergence of the finite-dimensional distributions,
see e.g.\ \cite[Section 13]{Billingsley:1999}.
Hence, for $t > 0$, \eqref{eq:invariance principle S_n}, \eqref{eq:supk(n)/n}
and Anscombe's theorem \cite[Theorem 1.3.1]{Gut:2009} imply
\begin{equation*}
n^{-1/2} S_{k(nt)}	~\stackrel{\mathcal{D}}{\to}~	(B^{\lambda}(t),M^{\lambda}(t))		\quad	\text{as } n \to \infty
\end{equation*}
where $(B^{\lambda}(t),M^{\lambda}(t)) = (\E_{\lambda}[\tau_2-\tau_2])^{-1/2} (\tilde{B}^{\lambda}(t),\tilde{M}^{\lambda}(t))$.

Moreover, by inspecting the proof of \cite[Theorem 1.3.1]{Gut:2009},
this convergence can be strengthened to finite-dimensional convergence.
According to \cite[Theorem 13.1]{Billingsley:1999},
in order to prove convergence of $(n^{-1/2} S_{k(nt)})_{0 \leq t \leq 1}$ to $(B^{\lambda},M^{\lambda})$ in the Skorokhod space,
it suffices to check that $((n^{-1/2} S_{k(nt)})_{0 \leq t \leq 1})_{n \geq 1}$ is tight.
To this end, we invoke \cite[Theorem 13.2]{Billingsley:1999},
which yields tightness, once we have verified the conditions of the theorem.
For a function $f:[0,1] \to \R^2$, we write $\| f \|$ for $\sup_{t \in [0,1]} |f(t)|$
where $|f(t)|$ denotes the Euclidean norm of $f(t)$.
Sometimes, we write $\|f(t)\|$ for $\| f \|$.
To verify the first condition of \cite[Eq.\ (13.4) in Theorem 13.2]{Billingsley:1999},
we first notice that \cite[Theorem 14.4]{Billingsley:1999} and Slutsky's theorem imply
\begin{equation}	\label{eq:invariance principle for S_k(n)}
n^{-1/2} S_{\lfloor k(n)t\rfloor}	~\Rightarrow~	(B^{\lambda}(t),M^{\lambda}(t))		\quad	\text{as } n \to \infty.
\end{equation}
Using this and \eqref{eq:supk(n)/n}, we conclude that
\begin{align*}
\lim_{a \to \infty} & \limsup_{n \to \infty} \Prob_{\lambda} \left[\|n^{-1/2} S_{k(nt)}\| \geq a\right]		\\
&=~ \lim_{a \to \infty} \limsup_{n \to \infty} \Prob_{\lambda} \left[\Big(\frac{k(n)}{n}\Big)^{1/2} \max_{j=0,\ldots,k(n)} \frac{|S_j|}{k(n)^{1/2}} \geq a\right]	\\
&=~ \lim_{a \to \infty} \Prob_{\lambda}[\|(B^{\lambda},M^{\lambda})\| \geq a]	~=~	0.
\end{align*}
Turning to the second condition, we need to estimate terms of the form $|S_{k(nt)}-S_{k(ns)}|$ uniformly in $|t-s| \leq \delta$
for some $\delta \in (0,1)$ that will ultimately tend to $0$.
Using the triangular inequality, we obtain
\begin{eqnarray*}
|S_{k(nt)} -S_{k(ns)}|
& \leq &
|S_{k(nt)}-S_{\lfloor nt / \E_{\lambda}[\tau_2-\tau_1]\rfloor}| + |S_{k(ns)}-S_{\lfloor ns / \E_{\lambda}[\tau_2-\tau_1]\rfloor}|	\\
& & + |S_{\lfloor nt / \E_{\lambda}[\tau_2-\tau_1]\rfloor} - S_{\lfloor ns / \E_{\lambda}[\tau_2-\tau_1]\rfloor}|.
\end{eqnarray*}
Since $n^{-1/2} S_{\lfloor nt / \E_{\lambda}[\tau_2-\tau_1]\rfloor}$ converges in distribution on $D[0,1]$ by \eqref{eq:invariance principle S_n},
it is in particular tight and satisfies the second condition of Theorem 13.2 in \cite{Billingsley:1999}.
Therefore, it is enough to consider the first two terms on the right-hand side of the last inequality.
By symmetry, it suffices to consider one of them. Let $\varepsilon > 0$. Then, for arbitrary $c > 0$,
\begin{align*}
& \Prob_{\lambda} \left[n^{-1/2}\|S_{k(nt)}-S_{\lfloor nt / \E_{\lambda}[\tau_2-\tau_1]\rfloor}\| \geq \varepsilon \right]	\\
& \leq \Prob_{\lambda}\left[\left\| k(nt) \! - \! \left\lfloor \frac{nt}{\E_{\lambda} [\tau_2-\tau_1]}	 \right\rfloor\right\| > nc \right]
+ \Prob_{\lambda}\left[ \sup_{j \leq k(n)} \max_{|i-j| \leq nc} \frac1{\sqrt{n}}|S_{j}-S_{i}| \geq \varepsilon \right]\!\!.
\end{align*}
The first term tends to $0$ as $n \to \infty$ for any given $c > 0$ by \eqref{eq:supk(n)/n}.
By \eqref{eq:invariance principle for S_k(n)} and the continuous mapping theorem, the second term tends to
\begin{equation*}
\Prob_{\lambda}\left[ \sup_{0 \leq t \leq 1} \max_{|t-s| \leq c} |(B^{\lambda}(t),M^{\lambda}(t))-(B^{\lambda}(s),M^{\lambda}(s))| \geq \varepsilon \right]
\end{equation*}
which tends to $0$ as $c \to 0$, since Brownian motion is a.\,s.\ continuous (hence, uniformly continuous on compact intervals).
Therefore,
\begin{equation*}
\limsup_{n \to \infty} \Prob_{\lambda} \left[n^{-1/2}\|S_{k(nt)}-S_{\lfloor nt / \E_{\lambda}[\tau_2-\tau_1]\rfloor}\| \geq \varepsilon \right] ~=~	0.
\end{equation*}\smallskip

\noindent
\emph{Step 2:}
With $\| \cdot \|$ denoting the supremum norm of one- or two-dimensional functions, respectively,
the distance between $(B_n(\cdot),n^{-1/2} M^{\lambda}_{\lfloor n \cdot \rfloor)}$ and $S_{k(n \cdot)}$ can be estimated as follows:
\begin{align*}
\|&(B_n(t),M^{\lambda}_{\lfloor nt \rfloor}/\sqrt{n}) - S_{k(nt)}/\sqrt{n}\|	\\
& \leq~
n^{-1/2} \big( \|X_{\lfloor nt \rfloor} - \lfloor nt \rfloor \vel - (X_{\tau_{k(nt)}} - \tau_{k(nt)} \vel)\|
+ \|M^{\lambda}_{\lfloor nt \rfloor} - M^{\lambda}_{\tau_{k(nt)}} \| \big)	\\
& \leq~
n^{-1/2} \big(\|X_{\lfloor nt \rfloor} - X_{\tau_{k(nt)}} \| + \vel \| \tau_{k(nt)} - \lfloor nt \rfloor \| + \|M^{\lambda}_{\lfloor nt \rfloor} - M^{\lambda}_{\tau_{k(nt)}} \|\big).
\end{align*}
Here, for the first term, we find
\begin{equation*}
\|X_{\lfloor nt \rfloor} \! - \! X_{\tau_{k(nt)}} \|
~\leq~	\|X_{\tau_{k(nt)+1}} \! - \! X_{\tau_{k(nt)}}\|
~=~	\max_{j=0,\ldots,k(n)} (\rho_{j+1}\!-\!\rho_j).
\end{equation*}
Thus, for any $\varepsilon > 0$, using $k(n) \leq n$, the union bound and Chebychev's inequality give
\begin{align*}
\Prob_{\lambda}(n^{-1/2} & \|X_{\lfloor nt \rfloor} \! - \! X_{\tau_{k(nt)}} \| \geq \varepsilon)
~\leq~
\Prob_{\lambda}\Big(\max_{j \leq k(n)} (\rho_{j+1}\!-\!\rho_j) \geq \varepsilon \sqrt{n}\Big)	\\
& \leq~
\Prob_{\lambda}(\rho_1 \geq \varepsilon \sqrt{n} / 2) + n \Prob_{\lambda}(\rho_{2}\!-\!\rho_1 \geq \varepsilon \sqrt{n}/2)	\\
& \leq~
\Prob_{\lambda}(\rho_1 \geq \varepsilon \sqrt{n} / 2) + 4 \varepsilon^{-2} \E_{\lambda}[(\rho_{2}\!-\!\rho_1)^2 \1_{\{\rho_{2}\!-\!\rho_1 \geq \varepsilon \sqrt{n}/2\}}]
~\underset{n \to \infty}{\to}~	0.
\end{align*}
The other two terms are treated in a similar manner. Finally, we obtain
\begin{equation*}
\big\|(B_n(t),M^{\lambda}_{\lfloor nt \rfloor}/\sqrt{n}) - n^{-1/2} S_{k(nt)}\big\|	~\underset{n \to \infty}{\to}~	0	\quad	\text{ in } \Prob_{\lambda} \text{-probability.}
\end{equation*}
In view of Theorem 3.1 in \cite{Billingsley:1999}, the convergence of $n^{-1/2} S_{k(nt)}$ in $D[0,1]$
thus implies the convergence of $(B_n(t),M^{\lambda}_{\lfloor nt \rfloor}/\sqrt{n})$ in $D[0,1]$.\footnote{
In fact, one needs to show the above convergence in $\Prob_{\lambda}$-probability with the supremum norm replaced by a metric that induces the Skorokhod topology, for instance, the metric $d^{\circ}$ defined on p.\;125 of \cite{Billingsley:1999}. However, $d^{\circ}(\cdot,\cdot) \leq \|\cdot - \cdot\|$.}

Now we show \eqref{eq:sup kappa-moment X}.
To this end, pick $\kappa > 2$ with $\E_{\lambda}[(\tau_2-\tau_1)^{\kappa}] < \infty$.
The existence of $\kappa$ is guaranteed by Lemma \ref{Lem:moments of rho and tau}.
For $n \in \N$, observe that $\nu(n) \defeq \inf\{j \in \N: \tau_j > n\} = k(n)+1$ is a stopping time
w.r.t.\ the filtration $(\G_k)_{k \in \N_0}$ where $\G_k = \sigma((\rho_j,\tau_j): 1 \leq j \leq k)$.
Further, writing $\|\cdot\|_{\kappa}$ for the $\kappa$-norm w.r.t.\ $\Prob_{\lambda}$,
we infer from Minkowski's inequality that
\begin{align}
\|B_n(1)\|_{\kappa}&
~\leq~
\frac1{\sqrt{n}}\big(\|X_{\tau_{\nu(n)}}-\tau_{\nu(n)} \vel\|_{\kappa} + \|X_{\tau_{\nu(n)}}-X_{n}\|_{\kappa} + \vel \|\tau_{\nu(n)}-n\|_{\kappa}\big)	\notag	\\
& =~
\frac1{\sqrt{n}} \bigg\|\sum_{j=1}^{\nu(n)} \xi_j\bigg\|_{\kappa}
+ \frac1{\sqrt{n}} \|\rho_{\nu(n)}\!-\!\rho_{k(n)}\|_{\kappa} + \frac{\vel}{\sqrt{n}} \|\tau_{\nu(n)}\!-\!\tau_{k(n)}\|_{\kappa}.	\label{eq:three summands}
\end{align}
If $\xi_1, \xi_2, \ldots$ were i.i.d.\ under $\Prob_{\lambda}$,
boundedness of the first summand as $n \to \infty$ would follow from classical renewal theory as presented in \cite{Gut:2009}.
However, we have to incorporate the fact that, under $\Prob_{\lambda}$, $\xi_1$ has a different distribution than the $\xi_j$'s for $j \geq 2$.
Define $\nu'(k) = \inf\{j \in \N_0: \tau_{j+1}-\tau_1 > k\}$ and use Minkowski's inequality to obtain
\begin{equation*}	\textstyle
\big\|\sum_{j=1}^{\nu(n)} \xi_j\big\|_{\kappa}
\leq
\|\xi_1\|_{\kappa}
+\big\|\sum_{j=2}^{\nu'(n-\tau_1)} \xi_j\big\|_{\kappa}.
\end{equation*}
Condition w.r.t.\ $\G_1$ in the second summand to obtain
\begin{eqnarray*}
\E_{\lambda}\bigg[\bigg(\sum_{j=2}^{\nu'(n-\tau_1)} \xi_j\bigg)^{\!\!\kappa}\bigg]
& = &
\E_{\lambda}\bigg[\E_{\lambda}\bigg[\bigg(\sum_{j=2}^{\nu'(n-\tau_1)} \xi_j\bigg)^{\!\!\kappa} \,\bigg|\, \G_1\bigg]\bigg]	\\
& \leq &
\E_{\lambda}\Big[2 B_{\kappa} \E_{\lambda}\big[\big|\xi_2\big|^{\kappa}\big] \E_{\lambda}[\nu'(n-\tau_1)^{\kappa/2} \mid \G_1]\Big]	\\
& \leq &
2 B_{\kappa} \E_{\lambda}\big[\big|\xi_2\big|^{\kappa}\big] \E_{\lambda}[\nu'(n)^{\kappa/2}]
\end{eqnarray*}
where we have used \cite[Theorem 1.5.1]{Gut:2009} for the first inequality and where $B_{\kappa}$ is a finite constant depending only on $\kappa$.
Now take the $\kappa$th root to arrive at the corresponding bounds for the $\kappa$-norm and subsequently divide by $\sqrt{n}$.
Then, using that $n^{-1/2}(\E_{\lambda}[\nu'(n)^{\kappa/2}])^{1/\kappa} = \E_{\lambda}[(\nu'(n)/n)^{\kappa/2}]$
and the uniform integrability of $(\nu'(n)/n)^{\kappa/2}$, $n \in \N$ (see \cite[Formula (2.5.6)]{Gut:2009})
we conclude that the supremum over all $n \in \N$ of the first summand in \eqref{eq:three summands} is finite.
We now turn to the second and third summand in \eqref{eq:three summands}.
First observe that
$\E_{\lambda}[(\rho_2-\rho_1)^{\kappa}]<\infty$ and $\E[(\tau_2-\tau_1)^{\kappa}]<\infty$
by Lemma \ref{Lem:moments of rho and tau}.
Second, notice that
$\frac1n \nu(n) \to (\E_{\lambda}[\tau_2-\tau_1])^{-1}$ a.\,s.\ as $n \to \infty$
by the strong law of large numbers for renewal processes \cite[Theorem 2.5.1]{Gut:2009}
and that $(\frac1n\nu(n))_{n \in \N}$ is uniformly integrable, see \cite[Formula (2.5.6)]{Gut:2009}.
Therefore,
$\lim_{n \to \infty} n^{-1/2} \|\rho_{\nu(n)}\!-\!\rho_{k(n)}\|_{\kappa} = 0$
and $\lim_{n \to \infty} n^{-1/2} \|\tau_{\nu(n)}\!-\!\tau_{k(n)}\|_{\kappa} = 0$
is a consequence of \cite[Theorem 1.8.1]{Gut:2009}.

\noindent
Finally, fix $\lambda \in [\lambdacrit/2,\lambdacrit)$ and assume for a contradiction that \eqref{eq:joint invariance principle} holds.
Then $B_n = n^{-1/2}(X_n-n\vel) \to \sigma B(1)$ in distribution as $n \to \infty$ and, moreover,
\begin{eqnarray}	\label{eq:B_n-S_k(n)}
|B_n - n^{-1/2} S_{k(n)}|
& \leq &
n^{-1/2} \big(|X_n - X_{k(n)}| + \vel |n-\tau_{k(n)}| \big).
\end{eqnarray}
By the arguments given in the proof of Step 2 above, $n^{-1/2} (X_n - X_{k(n)}) \stackrel{\Prob_{\lambda}}{\to} 0$ as $n \to \infty$.
Further, $n-\tau_{k(n)}$ is the age at time $n$ of the (delayed) renewal process $(\tau_k)_{k \in \N_0}$.
By standard results from renewal theory, see e.g.~\cite[Corollary 10.1 on p.~76]{Thorisson:2000},
\begin{equation*}	\textstyle
\Prob_{\lambda}(n-\tau_{k(n)}=j)	~\underset{n \to \infty}{\to}~	\frac{1}{\E_{\lambda}[\tau_2-\tau_1]} \Prob_{\lambda}(\tau_2-\tau_1>j)
\end{equation*}
where $\lambda < \lambdacrit$ guarantees the finiteness of $\E_{\lambda}[\tau_2-\tau_1]$.
(Notice that the fact that $\tau_1$ has a different distribution than the $\tau_{n+1}-\tau_{n}$, $n \geq 1$
has no effect on this result.)
Hence, also $n^{-1/2}(n-\tau_{k(n)}) \to 0$ in $\Prob_{\lambda}$-probability as $n \to \infty$.
From \eqref{eq:B_n-S_k(n)} and Theorem 3.1 in \cite{Billingsley:1999},
we thus conclude that $n^{-1/2} S_{k(n)} \to \sigma B(1)$ in distribution as $n \to \infty$.
In particular, the sequence $(n^{-1/2} S_{k(n)})_{n \geq 1}$ is tight.
From Theorem 3.4 in \cite{Bednorz+Latuszynski+Latala:2008}
(notice that in the theorem, stochastic domination is assumed rather than tightness;
however, it is clear from the proof that tightness suffices),
we conclude that $\E_{\lambda}[\xi_2^2] < \infty$
which, in turn, gives $\E_{\lambda}[(\tau_2-\tau_1)^2] < \infty$.
This contradicts Lemma \ref{Lem:moments of rho and tau}.
\end{proof}

We continue with the proof of Proposition \ref{Prop:sup exp M}:

\begin{proof}[Proof of Proposition \ref{Prop:sup exp M}]
Choose an arbitrary $\theta > 0$.
By the Azuma-Hoeffding inequality \cite[E14.2]{Williams:1991},
with $c_{\lambda} \defeq \sup_{v,w,\omega} |\nu_{\omega,\lambda}(v,w)|$ where the supremum is over all $\omega \in \Omega$ and $v,w \in V$,
we have
\begin{equation*}	\textstyle
\Prob_{\lambda}(t n^{-1/2} M^{\lambda}_n \geq x)	~\leq~	\exp\big(-\frac{1}{2} \frac{x^2n}{t^2nc_{\lambda}^2}  \big)
~=~	\exp\big(-\frac{x^2}{2 t^2 c_{\lambda}^2}  \big)
\end{equation*}
for all $x > 0$.
This finishes the proof of \eqref{eq:sup exp M} because the bound on the right-hand side is independent of $n$.
\end{proof}

\section{Proof of Theorem \ref{Thm:differentiability of the speed}}

We carry out the program described on p\;\pageref{enum:program}.
The first two steps of the program
are contained in Theorem \ref{Thm:joint CLT} (the second step follows from \eqref{eq:sup kappa-moment X}).
We continue with Step 3.
It is based on a second order Taylor expansion for
$\sum_{j=1}^n \log \big(\frac{p_{\omega,\lambda}(Y_{j-1},Y_{j})}{p_{\omega,\lambda^{*}}(Y_{j-1},Y_{j})}\big)$
at $\lambda = \lambda^*$:
\begin{align}
& \textstyle
\sum_{j=1}^n \log \Big(\frac{p_{\omega,\lambda}(Y_{j-1},Y_{j})}{p_{\omega,\lambda^{*}}(Y_{j-1},Y_{j})}\Big)	\notag	\\
& \textstyle
~=	(\lambda\!-\!\lambda^*) M^{\lambda^*}_n	 + \frac{(\lambda\!-\!\lambda^*)^2}{2} \sum_{j=1}^n \! \Big( \!
\frac{p_{\omega,\lambda^*}''(Y_{j-1},Y_{j})}{p_{\omega,\lambda^*}(Y_{j-1},Y_{j})} - \nu_{\omega,\lambda^*}(Y_{j-1},Y_j)^2 \!\Big)	\notag	\\
& \textstyle
\hphantom{~=} + (\lambda\!-\!\lambda^*)^2 \sum_{j=1}^n r_{\omega,\lambda^*,Y_{j-1},Y_j}(\lambda)	\label{eq:Taylor}
\end{align}
where $r_{\omega,\lambda^*,v,w}(\lambda)$ tends to $0$ uniformly in $\omega \in \Omega$ and $v,w \in V$ as $\lambda \to \lambda^*$.
Set
\begin{equation*}	\textstyle
A_{\omega,\lambda^*}(n)
~=~ \frac12 \sum_{j=1}^{n} \Big(\nu_{\omega,\lambda^*}(Y_{j-1},Y_{j})^{2} - \frac{p_{\omega,\lambda^*}''(Y_{j-1},Y_j)}{p_{\omega,\lambda^*}(Y_{j-1},Y_j)}\Big)
\end{equation*}
and 
\begin{equation}	\label{eq:asymptotics R}	\textstyle
R_{\omega,\lambda^*,\lambda}(n)
=	(\lambda-\lambda^*)^2 \sum_{j=1}^{n} r_{\omega,\lambda^*,Y_{j-1},Y_j}(\lambda)
= (\lambda-\lambda^*)^2 \, n \, o(1),
\end{equation}
where $o(1)$ denotes a term that converges (uniformly) to $0$ as $\lambda \to \lambda^*$.

\begin{lemma}	\label{Lem:2nd order term}
Let $\lambda^{*} \in (0,\lambdacrit)$.
\begin{itemize}
	\item[(a)]
		If $\lambda^* \in (0,\lambdacrit/2)$, then
		\begin{equation}	\label{eq:A asymptotics}	\textstyle
		(\lambda-\lambda^*)^2 A_{\omega,\lambda^{*}}(n) \to \frac\alpha2 \E_{\lambda^*}[M^{\lambda^*}\!(1)^2]
		\quad	\Prob_{\lambda^*} \text{-a.\,s.\ and in } L^1(\Prob_{\lambda^*})
		\end{equation}
		if the limit $\lambda \to \lambda^{*}$ and $n\to\infty$ is such that $\lim_{n \to \infty}(\lambda-\lambda^*)^{2}n \eqdef \alpha > 0$.
	\item[(b)]
		If $\lambda^* \in (0,\lambdacrit)$ and $1 < r < \frac{\lambdacrit}{\lambda^*} \wedge 2$, then
		\begin{equation}	\label{eq:A asymptotics again}	\textstyle
		(\lambda-\lambda^*)^2 A_{\omega,\lambda^{*}}(n) \to 0
		\quad	\Prob_{\lambda^*} \text{-a.\,s.\ and in } L^1(\Prob_{\lambda^*})
		\end{equation}
		if the limit $\lambda \to \lambda^{*}$ and $n\to\infty$ is such that $\lim_{n \to \infty}(\lambda-\lambda^*)^{r}n \eqdef \alpha > 0$.
\end{itemize}
Further, $R_{\omega,\lambda^*,\lambda}(n) \to 0$ $\Prob_{\lambda^*}$-a.\,s.\ if the limits $\lambda \to \lambda^{*}$
and $n\to\infty$ are such that $\lim_{n \to \infty}(\lambda-\lambda^*)^{2} n < \infty$.
\end{lemma}
\begin{proof}
The convergence $R_{\omega,\lambda^*, \lambda}(n) \to 0$ if $\lambda \to \lambda^*$ and $n \to \infty$ such that
$\lim_{n \to \infty}(\lambda-\lambda^*)^{2} n < \infty$
follows immediately from \eqref{eq:asymptotics R}.

We now turn to assertions (a) and (b).
To this end, notice that
$A_{\omega,\lambda^*}(\tau_n) = \sum_{k=1}^n \xi_k$ where
\begin{equation*}
\xi_k	~\defeq~	\frac12 \sum_{j=\tau_{k-1}+1}^{\tau_k}
\bigg(\nu_{\omega,\lambda^*}(Y_{j-1},Y_{j})^{2} - \frac{p_{\omega,\lambda^*}''(Y_{j-1},Y_j)}{p_{\omega,\lambda^*}(Y_{j-1},Y_j)}\bigg),	\quad	k \in \N.
\end{equation*}
The $\xi_k$, $k \geq 2$ are i.i.d.\ by Lemma \ref{Lem:iid regeneration times and points}.
They are further integrable since the summands in the definition are uniformly bounded and $\E_{\lambda^*}[\tau_2-\tau_1]<\infty$.
The strong law of large numbers gives, as $n \to \infty$,
\begin{equation*}
\frac1n A_{\omega,\lambda*}(\tau_n) \to \frac12 \E_{\lambda^*}\bigg[\sum_{j=\tau_{1}+1}^{\tau_2}
\bigg(\nu_{\omega,\lambda^*}(Y_{j-1},Y_{j})^{2} - \frac{p_{\omega,\lambda^*}''(Y_{j-1},Y_j)}{p_{\omega,\lambda^*}(Y_{j-1},Y_j)}\bigg)\bigg]
\quad	\Prob_{\lambda^*}\text{-a.\,s.}
\end{equation*}
Using the sandwich argument from the proof of Proposition \ref{Prop:SLLN&MZlaw}(a), one infers
\begin{equation*}
\frac1n A_{\omega,\lambda*}(n) \to \frac12 \frac{\E_{\lambda^*}\big[\sum_{j=\tau_{1}+1}^{\tau_2}
\big(\nu_{\omega,\lambda^*}(Y_{j-1},Y_{j})^{2} - \frac{p_{\omega,\lambda^*}''(Y_{j-1},Y_j)}{p_{\omega,\lambda^*}(Y_{j-1},Y_j)}\big)\big]}
{\E_{\lambda^*}[\tau_2-\tau_1]}
\quad	\Prob_{\lambda^*}\text{-a.\,s.}
\end{equation*}
In the situation of (b), $(\lambda-\lambda^*)^2$ is of the order $n^{-2/r}$ with $2/r > 1$.
This implies that \eqref{eq:A asymptotics again} holds. 
In the situation of (a), we have $0 < \lambda^* < \lambdacrit/2$.
Since the $\nu_{\omega,\lambda^*}(Y_{j-1},Y_{j})^{2} - p_{\omega,\lambda^*}''(Y_{j-1},Y_j)/p_{\omega,\lambda^*}(Y_{j-1},Y_j)$, $j \in \N$
are bounded by a constant (depending on $\lambda^*$),
 $(\frac1n A_{\omega,\lambda*}(n))_{n \in \N}$ is a bounded sequence.
Thus,
$\E_{\lambda^*}[\lim_{n \to \infty}\frac 1n  A_{\omega,\lambda*}(n)] = \lim_{n \to \infty}\frac 1n \E_{\lambda^*}[A_{\omega,\lambda*}(n)]$
by the dominated convergence theorem,
and hence
\begin{equation*}	\textstyle
\lim_{n \to \infty}\frac 1n  A_{\omega,\lambda*}(n) = \lim_{n \to \infty}\frac 1n \E_{\lambda^*}[A_{\omega,\lambda*}(n)]
\quad \Prob_{\lambda^*}\text{-a.\,s.}
\end{equation*}
The latter limit can be calculated as follows.
For all $v$ and all $\omega$, $p_{\omega,\lambda^*}(v,\cdot)$
is a probability measure on the neighborhood
$N_{\omega}(v) = \{w \in V: p_{\omega,0}(v,w) > 0\}$ of $v$, hence
\begin{equation*}	\textstyle
\sum_{w\in N_{\omega}(v)} p_{\omega, \lambda^{*}}''(v,w) = 0.
\end{equation*}
This implies $E_{\omega,\lambda^*}\!\big[\frac{p_{\omega,\lambda^*}''(Y_{j-1},Y_j)}{p_{\omega,\lambda^*}(Y_{j-1},Y_j)}\big] = 0$
and also $\E_{\lambda^*}\!\big[\frac{p_{\omega,\lambda^*}''(Y_{j-1},Y_j)}{p_{\omega,\lambda^*}(Y_{j-1},Y_j)}\big] = 0$ for all $j \in \N$
and, thus,
\begin{align*}	\textstyle
\lim_{n \to \infty} \frac1n A_{\omega,\lambda^*}(n)
&=	\textstyle
\lim_{n \to \infty} \frac12 \frac1n \sum_{j=1}^{n} \E_{\lambda^*}\big[\nu_{\omega,\lambda^*}(Y_{j-1},Y_{j})^{2}\big]	\\
&=	\textstyle
\frac12 \lim_{n} \Var_{\lambda^*} [n^{-1/2}M^{\lambda^*}_n]
= \frac12 \E_{\lambda^*}[M^{\lambda^*}\!(1)^2]
\ \Prob_{\lambda^*} \text{-a.\,s.}
\end{align*}
where the second equality follows from the fact that the increments of square-integrable martingales are uncorrelated,
and the last equality follows from Theorem \ref{Thm:joint CLT}.
\end{proof}

\begin{proposition}	\label{Prop:3rd step}
Assume that $\lambda^* \in (0,\lambdacrit/2)$ and $\alpha > 0$.
Then
\begin{equation*}	\tag{\ref{eq:speed approx by covariance}}
\lim_{\substack{\lambda \to \lambda^*,\\ (\lambda-\lambda^*)^2n \to \alpha}} \frac{\E_{\lambda}[X_n]-\E_{\lambda^*}[X_n]}{(\lambda-\lambda^*)n}
= \E_{\lambda^*}[B^{\lambda^*}\!(1) M^{\lambda^*}\!(1)]
= \sigma_{12}(\lambda^*).
\end{equation*}
\end{proposition}
\begin{proof}
We have
\begin{equation}	\label{eq:step 3 decomposition}
\frac{\E_{\lambda}[X_n]-\E_{\lambda^*}[X_n]}{(\lambda-\lambda^*)n}
~=~	\frac{\E_{\lambda}[X_n]-n\vel(\lambda^*)}{(\lambda-\lambda^*)n} - \frac{\E_{\lambda^*}[X_n]-n\vel(\lambda^*)}{(\lambda-\lambda^*)n}.
\end{equation}
Regarding the second summand, Theorem \ref{Thm:joint CLT} implies that, under $\Prob_{\lambda^*}$,
\begin{equation*}
\frac{X_n-n\vel(\lambda^*)}{(\lambda-\lambda^*)n}	~=~	\frac{1}{(\lambda-\lambda^*) \sqrt{n}} \frac{X_n-n\vel(\lambda^*)}{\sqrt{n}}
~\to~	\frac{1}{\sqrt{\alpha}}B^{\lambda^*}(1)
\end{equation*}
in distribution as $n \to \infty$.
Further, \eqref{eq:supsecondmoment} implies convergence of the first moment.
Since $B^{\lambda^*}(1)$ is centered Gaussian, this means that the second summand in \eqref{eq:step 3 decomposition}
vanishes as $n \to \infty$.
It remains to show that
\begin{equation}	\label{eq:3rd step-remains to show}
\frac{\E_{\lambda}[X_n]-n\vel(\lambda^*)}{(\lambda-\lambda^*)n}	~\to~	\sigma_{12}(\lambda^*)	\quad
\text{as } \lambda \to \lambda^*, (\lambda-\lambda^*)^2n \to \alpha.
\end{equation}
To this end, we use the Radon-Nikod\'ym derivatives introduced in Section \ref{sec:main results}
and follow the end of the proof of Theorem 2.3 in \cite{Mathieu:2015}.
Indeed, using \eqref{eq:girsanov annealed} and \eqref{eq:Taylor}, we get
\begin{align*}
&\E_{\lambda}[X_n-n\vel(\lambda^*)]
=
\E_{\lambda^*}\bigg[(X_n\!-\!n\vel(\lambda^*))
\exp\bigg(\sum_{j=1}^{n} \log \frac{p_{\omega,\lambda}(Y_{j-1},Y_{j})}{p_{\omega,\lambda^{*}}(Y_{j-1},Y_{j})}\bigg)\bigg]	\\
&=
\E_{\lambda^*}\big[(X_n\!-\!n\vel(\lambda^*))
\exp\big((\lambda\!-\!\lambda^*)M^{\lambda^*}_n - (\lambda\!-\!\lambda^*)^2 A_{\omega,\lambda^*}(n) + R_{\omega,\lambda^*,\lambda}(n)\big)\big].
\end{align*}
Now divide by $(\lambda-\lambda^*)n \sim \sqrt{\alpha n}$ and use Theorem \ref{Thm:joint CLT}, Lemma \ref{Lem:2nd order term},
Slutsky's theorem and the continuous mapping theorem to conclude
\begin{align}
\frac{X_n-n\vel(\lambda^*)}{(\lambda-\lambda^*)n} &
\exp\big((\lambda\!-\!\lambda^*)M^{\lambda^*}_n - (\lambda\!-\!\lambda^*)^2 A_{\omega,\lambda^*}(n) + R_{\omega,\lambda^*,\lambda}(n)\big)	\notag	\\
&\stackrel{\mathcal{D}}{\to}~ \frac{1}{\sqrt{\alpha}} B^{\lambda^*}\!(1) \exp\Big(\sqrt{\alpha} M^{\lambda^*}\!(1) - \frac\alpha2 \E_{\lambda^*}[M^{\lambda^*}\!(1)^2]\Big).		\label{eq:convergence in distribution for 3rd step}
\end{align}
Suppose that along with convergence in distribution, convergence of the first moment holds.
Then we infer
\begin{align*}
\lim  \frac{\E_{\lambda}[X_n]-n\vel(\lambda^*)}{(\lambda-\lambda^*)n}
& =
\frac{1}{\sqrt{\alpha}}
\E_{\lambda^*}\Big[B^{\lambda^*}\!(1) \exp\Big(\sqrt{\alpha} M^{\lambda^*}\!(1) - \frac\alpha2 \E_{\lambda^*}[M^{\lambda^*}\!(1)^2]\Big)\Big]	\\
& =
\E_{\lambda^*}[B^{\lambda^*}\!(1) M^{\lambda^*}\!(1)]
= \sigma_{12}(\lambda)
\end{align*}
where the last step follows from the integration by parts formula for two-dimensional Gaussian vectors\footnote{
There are several proofs of this formula, for instance,
one can consider the bivariate moment generating function $\Phi(s,t) = \E_{\lambda^*}[\exp(sB^{\lambda^*}\!(1) + tM^{\lambda^*}\!(1))]$,
differentiate with respect to $s$ and evaluate at $(s,t)=(0,1)$.}
and the limit is as $\lambda \to \lambda^*, (\lambda-\lambda^*)^2n \to \alpha$.
It remains to show that the family on the left-hand side of \eqref{eq:convergence in distribution for 3rd step} is uniformly integrable.
To this end, use H\"older's inequality to obtain
\begin{align*}
\sup_{\lambda,n} \E_{\lambda^*} \! & \bigg[\bigg|\frac{X_n-n\vel(\lambda^*)}{(\lambda-\lambda^*)n}
e^{(\lambda\!-\!\lambda^*)M^{\lambda^*}_n - (\lambda\!-\!\lambda^*)^2 A_{\omega,\lambda^*}(n) + R_{\omega,\lambda^*,\lambda}(n)}\bigg|^{\frac65}\bigg]	\\
& \leq \sup_{\lambda,n} \E_{\lambda^*}\! \bigg[\bigg|\frac{X_n-n\vel(\lambda^*)}{(\lambda-\lambda^*)n}\bigg|^{2}\bigg]^{\frac{3}{5}}	\\
& \hphantom{\leq} \ \cdot
\sup_{\lambda,n} \E_{\lambda^*} \! \Big[e^{3(\lambda\!-\!\lambda^*)M^{\lambda^*}_n - 3(\lambda\!-\!\lambda^*)^2 A_{\omega,\lambda^*}(n)
+ 3 R_{\omega,\lambda^*,\lambda}(n)}\Big]^{\frac25}\!\!.
\end{align*}
By \eqref{eq:sup kappa-moment X}, the first supremum in the last line is finite.
To show finiteness of the second,
first notice that $(\lambda-\lambda^*)^2 A_{\omega,\lambda^*}(n)$ and $R_{\omega,\lambda^*,\lambda}(n)$ are
(for fixed $\lambda^*$) bounded sequences
when $(\lambda-\lambda^*)^2n$ stays bounded (see the proof of Lemma \ref{Lem:2nd order term} for details),
while $\sup_{\lambda,n} \E_{\lambda^*} \! [e^{3(\lambda\!-\!\lambda^*)M^{\lambda^*}_n}] < \infty$
follows from \eqref{eq:sup exp M}.
\end{proof}

For later use, we state here an analogous result used in the proof of Theorem \ref{Thm:continuity of the speed}.
Since the proof is an adaption of the proof of Proposition \ref{Prop:3rd step}
we refrain from giving the details here and only note that Theorem \ref{Thm:Marcinkiewicz-Zygmund} is used at this point
(instead of the central limit theorem).

\begin{proposition}	\label{Prop:3rd step again}
Assume that $\lambda^* \in (0,\lambdacrit)$ and let $1 < r < \frac{\lambdacrit}{\lambda^*} \wedge 2$.
Then, for arbitrary $\alpha > 0$,
\begin{equation*}	\label{eq:3rd step again}
\lim_{\substack{\lambda \to \lambda^*,\\ n(\lambda-\lambda^*)^r \to \alpha}}
\frac{\E_{\lambda}[X_n]-\E_{\lambda^*}[X_n]}{(\lambda-\lambda^*)^{r-1}n} = 0.
\end{equation*}
\end{proposition}
\auskommentiert{
\begin{proof}
We decompose as in the proof of Proposition \ref{Prop:3rd step}:
\begin{equation}	\label{eq:step 3 again decomposition}
\frac{\E_{\lambda}[X_n]-\E_{\lambda^*}[X_n]}{(\lambda-\lambda^*)^{r-1}n}
~=~	\frac{\E_{\lambda}[X_n]-n\vel(\lambda^*)}{(\lambda-\lambda^*)^{r-1}n} - \frac{\E_{\lambda^*}[X_n]-n\vel(\lambda^*)}{(\lambda-\lambda^*)^{r-1}n}.
\end{equation}
The second summand satisfies
\begin{equation*}
\frac{\E_{\lambda^*}[X_n]-n\vel(\lambda^*)}{(\lambda-\lambda^*)^{r-1}n}
= \frac{\E_{\lambda^*}[X_n]-n\vel(\lambda^*)}{n^{1/r}} \cdot \frac{n^{1/r}}{(\lambda-\lambda^*)^{r-1}n} \to 0
\end{equation*}
by Proposition \ref{Prop:SLLN&MZlaw} and the relation $n(\lambda-\lambda^*)^r \to \alpha$.
We deal with the first summand as in the proof of Proposition \ref{Prop:3rd step}:
\begin{align*}
&\E_{\lambda}[X_n-n\vel(\lambda^*)]
=
\E_{\lambda^*}\bigg[(X_n\!-\!n\vel(\lambda^*))
\exp\bigg(\sum_{j=1}^{n} \log \frac{p_{\omega,\lambda}(Y_{j-1},Y_{j})}{p_{\omega,\lambda^{*}}(Y_{j-1},Y_{j})}\bigg)\bigg]	\\
&=
\E_{\lambda^*}\big[(X_n\!-\!n\vel(\lambda^*))
\exp\big((\lambda\!-\!\lambda^*)M^{\lambda^*}_n - (\lambda\!-\!\lambda^*)^2 A_{\omega,\lambda^*}(n) + R_{\omega,\lambda^*,\lambda}(n)\big)\big].
\end{align*}
From $\lambda\!-\!\lambda^* \sim \alpha^{1/r} n^{-1/r}$ and \eqref{eq:Marcinkiewicz-Zygmund}
we conclude $(\lambda\!-\!\lambda^*)M^{\lambda^*}_n \to 0$ $\Prob_{\lambda}$-a.\,s.
From Lemma \ref{Lem:2nd order term} it follows that also
$(\lambda\!-\!\lambda^*)^2 A_{\omega,\lambda^*}(n) \to 0$ and $R_{\omega,\lambda^*,\lambda}(n) \to 0$ $\Prob_{\lambda^*}$-a.\,s.
On the other hand,
\begin{equation*}
\E_{\lambda^*}\bigg[\bigg|\frac{X_n\!-\!n\vel(\lambda^*)}{n^{1/r}}\bigg|\bigg] \to 0
\end{equation*}
again by \eqref{eq:Marcinkiewicz-Zygmund}. Thus, the claim follows
if we can check that the family
\begin{equation*}
\frac{X_n\!-\!n\vel(\lambda^*)}{n^{1/r}} \cdot 
\exp\bigg(\sum_{j=1}^{n} \log \frac{p_{\omega,\lambda}(Y_{j-1},Y_{j})}{p_{\omega,\lambda^{*}}(Y_{j-1},Y_{j})}\bigg),	\quad	n \in \N, \ \lambda \to \lambda^*
\end{equation*}
is uniformly integrable under $\Prob_{\lambda^*}$ where $\lambda \to \lambda^*$ such that $n(\lambda-\lambda^*)^r \to \alpha$.
Let $1 < p < r$ and use H\"older's inequality to infer
\begin{align*}
\sup_{\lambda,n} \E_{\lambda^*} \! & \bigg[\bigg|\frac{X_n-n\vel(\lambda^*)}{n^{1/r}}
e^{(\lambda\!-\!\lambda^*)M^{\lambda^*}_n - (\lambda\!-\!\lambda^*)^2 A_{\omega,\lambda^*}(n) + R_{\omega,\lambda^*,\lambda}(n)}\bigg|^{p}\bigg]	\\
& \leq \sup_{n} \E_{\lambda^*}\! \bigg[\bigg|\frac{X_n-n\vel(\lambda^*)}{n^{1/r}}\bigg|^{r}\bigg]^{\frac{p}{r}}	\\
& \hphantom{\leq} \ \cdot
\sup_{\lambda,n} \E_{\lambda^*} \! \Big[e^{\frac{r}{r-p}(\lambda\!-\!\lambda^*)M^{\lambda^*}_n - \frac{r}{r-p}(\lambda\!-\!\lambda^*)^2 A_{\omega,\lambda^*}(n)
+ \frac{r}{r-p} R_{\omega,\lambda^*,\lambda}(n)}\Big]^{\frac{r-p}{r}}\!\!.
\end{align*}
The first supremum is finite by Theorem \ref{Thm:Marcinkiewicz-Zygmund}.
The second supremum is finite by the arguments in the end of the proof of Proposition \ref{Prop:3rd step}. 
\end{proof}
}

We complete the fourth step of the program on p.\;\pageref{enum:program}
by proving the following two results.

\begin{lemma}	\label{Lem:uniform bound X_n-n vel}
Let $\lambda^*,\delta > 0$.
\begin{itemize}
	\item[(a)]
		If $[\lambda^*-\delta,\lambda^*+\delta] \subseteq (0,\lambdacrit/2)$,
		then there exists a constant $C(\lambda^*,\delta)$ with
		\begin{equation}	\label{eq:absolute bound for E_lambda[X_n]}
		|\E_{\lambda}[X_n] - n\vel(\lambda)| ~\leq~	C(\lambda^*,\delta)
		\end{equation}
		for all $\lambda \in [\lambda^*-\delta,\lambda^*+\delta]$ and all $n \in \N$.
	\item[(b)]
		If $[\lambda^*-\delta,\lambda^*+\delta] \subseteq (0,\lambdacrit)$ and $1 < r < \frac{\lambdacrit}{\lambda^*+\delta} \wedge 2$,
		then
		\begin{equation}	\label{eq:estimate for E_lambda[X_n]}	\textstyle
		n^{-1/r} \sup_{|\lambda - \lambda^*| \leq \delta} |\E_{\lambda}[X_n] - n\vel(\lambda)| ~\to~	0
		\quad	\text{as } n \to \infty.
		\end{equation}
\end{itemize}
\end{lemma}

The first part of the lemma has the following immediate corollary.

\begin{corollary}	\label{Cor:4th step}
Let $\lambda^* \in (0,\lambdacrit/2)$. Then
\begin{equation*}	\tag{\ref{eq:2nd step}}
\lim_{\substack{\lambda \to \lambda^*,\\ (\lambda-\lambda^*)n \to \infty}}
\bigg[\frac{\vel(\lambda)-\vel(\lambda^*)}{\lambda-\lambda^*} - \frac{\E_{\lambda}[X_n]-\E_{\lambda^*}[X_n]}{(\lambda-\lambda^*)n}\bigg]
~=~ 0.
\end{equation*}
\end{corollary}

\begin{proof}[Proof of Lemma \ref{Lem:uniform bound X_n-n vel}]
Choose $\delta > 0$ such that $0 < \lambda^*-\delta < \lambda^*+\delta < \lambdacrit$.
We first remind the reader that $\nu(n) = \inf\{j \in \N: \tau_j > n\} = k(n)+1$
is a stopping time with respect to the canonical filtration of $((\rho_j-\rho_{j-1},\tau_j-\tau_{j-1}))_{j \in \N}$.
For $n \in \N$, we decompose $\E_{\lambda}[X_n]$ in the form
\begin{equation}	\label{eq:E_lambda X_n sandwiched}
\E_{\lambda}[X_n]
~=~	\E_{\lambda}[X_{n}- \rho_{\nu(n)}] +\E_{\lambda}[\rho_{\nu(n)}],
\end{equation}
and estimate the two summands on the right-hand side separately.
The first summand in \eqref{eq:E_lambda X_n sandwiched}
is uniformly bounded in $\lambda \in [\lambda^*-\delta,\lambda^*+\delta]$
and $n \in \N_0$ by Lemma \ref{Lem:uniform regeneration estimates}(a).

In order to deal with the second summand, as in the proof of Theorem \ref{Thm:joint CLT},
we define $\nu'(k) = \inf\{j \in \N_0: \tau_{j+1}-\tau_1 > k\}$, $k \in \Z$. Then
\begin{equation*}	\textstyle
\rho_{\nu(n)}	~=~	 \rho_1 + \sum_{j=1}^{\nu'(n-\tau_1)} (\rho_{j+1}-\rho_{j}).
\end{equation*}
Now take expectation with respect to $\Prob_{\lambda}[\,\cdot\,| (\rho_1,\tau_1)]$, use Wald's equation
and then integrate with respect to $\Prob_{\lambda}$ to obtain
\begin{equation}
\E_{\lambda}[\rho_{\nu(n)}]
~=~
\E_{\lambda}[\rho_1] + \E_{\lambda}[\nu'(n-\tau_1)] \E_{\lambda}[\rho_2-\rho_1]	\label{eq:Wald's equation for E_lambda[rho_nu(n)]}.
\end{equation}
We use \eqref{eq:Wald's equation for E_lambda[rho_nu(n)]} to derive a lower bound for $\E_{\lambda}[\rho_{\nu(n)}]$.
For $j=1,\ldots,n$, Wald's equation gives $\E_{\lambda}[\nu'(n-j)] = \E_{\lambda}[\tau_{\nu'(n-j)+1}-\tau_1]/\E_{\lambda}[\tau_2-\tau_1]$.
Thus, the right-hand side of \eqref{eq:Wald's equation for E_lambda[rho_nu(n)]} can be bounded below by
\begin{eqnarray*}
\E_{\lambda}[\rho_{\nu(n)}]
& = &
\E_{\lambda}[\rho_1] + \sum_{j=1}^n \Prob_{\lambda}(\tau_1=j) \E_{\lambda}[\nu'(n-j)] \E_{\lambda}[\rho_2-\rho_1]	\\
& \geq &
\vel(\lambda) \sum_{j=1}^n \Prob_{\lambda}(\tau_1=j) \E_{\lambda}[\tau_{\nu'(n-j)+1}-\tau_1]	\\
& \geq &
\vel(\lambda) \sum_{j=1}^n \Prob_{\lambda}(\tau_1=j) (n-j)	\\
& \geq &
n \vel(\lambda) - \vel(\lambda) \E_{\lambda}[\tau_1] - n\vel(\lambda)\Prob_{\lambda}(\tau_1>n)	\\
& \geq &
n \vel(\lambda) - 2 \E_{\lambda}[\tau_1]
\end{eqnarray*}
where in the last step we have used $\vel(\lambda) \leq 1$ and $n\Prob_{\lambda}(\tau_1>n) \leq \E_{\lambda}[\tau_1]$.
Regarding the upper bound for $\E_{\lambda}[\rho_{\nu(n)}]$, we again use \eqref{eq:Wald's equation for E_lambda[rho_nu(n)]}
to conclude
\begin{eqnarray*}
\E_{\lambda}[\rho_{\nu(n)}]
& \leq &
\E_{\lambda}[\rho_1] + \E_{\lambda}[\nu'(n)] \E_{\lambda}[\rho_2-\rho_1]	\\
& = &
\E_{\lambda}[\rho_1] + \vel(\lambda) \E_{\lambda}[\tau_{\nu'(n)+1}-\tau_1]	\\
& = &
n \vel(\lambda) + \E_{\lambda}[\rho_1] + \vel(\lambda) \E_{\lambda}[(\tau_{\nu'(n)+1}-\tau_1 - n)]	\\
& \leq &
n \vel(\lambda) + \E_{\lambda}[\rho_1] + \E^{\circ}_{\lambda}[(\tau_{\nu(n)}- n) \mid Y_k \neq \mathbf{0} \text{ for all } k \geq 1].
\end{eqnarray*}
The estimates derived above together with Lemma \ref{Lem:uniform regeneration estimates}
yield assertions (a) and (b).
\end{proof}

Apart from the proofs of several lemmas we have referred to,
the proof of Theorem \ref{Thm:differentiability of the speed} is now complete.

\section{Regeneration estimates}	\label{sec:regeneration}

\subsection{The time spent in traps.}

We start by considering the discrete line segment $\{0,\ldots,m\}$ and a nearest-neighbor random walk $(S_n)_{n \geq 0}$
on this set starting at $i \in \{0,\ldots,m\}$ with transition probabilities
\begin{equation*}
\Probi(S_{k+1}=j+1 \mid S_k=j) =	1-\Probi(S_{k+1}=j-1 \mid S_k=j) = \frac{e^{\lambda}}{e^{-\lambda}+e^{\lambda}}
\end{equation*}
for $j=1,\ldots,m-1$ and
\begin{equation*}
\Probi(S_{k+1}=1 \mid S_k=0) = \Probi(S_{k+1}=m-1 \mid S_k=m) = 1.
\end{equation*}
For $i=0$, we are interested in $\tau_m \defeq \inf\{k \in \N: S_k = 0\}$, the time until the first return of the walk to the origin.
The stopping times $\tau_m$ will be used to estimate the time the agile walk $(Z_n)_{n \geq 0}$ spends
in a trap of length $m$ given that it steps into it.

\begin{lemma}	\label{Lem:Etau_m}
In the given situation, the following assertions hold true.
\begin{itemize}
	\item[(a)]
		For each $m \in \N$, we have $\Erm_0 [\tau_m] = 2 \frac{e^{2 \lambda m} -1}{e^{2 \lambda} -1}$.
	\item[(b)]
		For any $\kappa \geq 1$ and every $m \in \N$, we have
		\begin{equation*}	\textstyle
		2^{\kappa} e^{2 \kappa \lambda (m-1)}
		\leq \Erm_0 [\tau_m^{\kappa}] \leq c(\kappa,\lambda) m^{\kappa} e^{2 \kappa \lambda m}
		\end{equation*}
		where $c(\kappa,\lambda)
		= 2^{\kappa-1}(1 + 2(2(\frac{\kappa}{e})^{\kappa} + \Gamma(\kappa\!+\!1))
		(\frac{e^{2\lambda}+1}{e^{2\lambda}-1})^{\kappa})$.
	\item[(c)]
		Assume there is a sequence $G_1, G_2, \ldots$ of independent random variables
		defined on the same probability space as and independent of $(S_n)_{n \geq 0}$.
		Further, suppose that there is $r \in (0,1)$ such that for all $j \in \N$ and $n \in \N_0$, we have
		$\mathrm{P}_{\!0}(G_j > n) \leq r^n$.
		Then, for all $m \in \N$,
		\begin{equation*}
		\Erm_0\bigg[\bigg(\sum_{j=1}^{\tau_m} G_j\bigg)^{\!\!\kappa}\bigg]
		~\leq~	\frac{r}{|\log r|^{\kappa}} \bigg(2 \Big(\frac{\kappa}{e}\Big)^{\!\kappa}
		+ \frac{\Gamma(\kappa+1)}{|\log r|}\bigg) c(\kappa,\lambda) m^{\kappa} e^{2\kappa\lambda m}.
		\end{equation*}
\end{itemize}
\end{lemma}

Before we give the proof of Lemma \ref{Lem:Etau_m},
we remark that with some more effort, it would be possible to determine the exact order of $\Erm_0[\tau_m^\kappa]$.
However, the estimates in the lemma are precise enough for our purposes.

\begin{proof}
Clearly, $\tau_1 = 2$ and, for $m > 1$, by the strong Markov property,
\begin{equation}	\label{eq:trap recursion}
\tau_m	~\eqdist~	2+ \sum_{j=1}^G \tau_{m-1}^{(j)}	\quad	\text{under } \mathrm{P}_{\!0}
\end{equation}
where $\tau_{m-1}^{(j)}$, $j \in \N$ are i.i.d.~copies of $\tau_{m-1}$
and $G$ is an independent geometrically distributed random variable with
\begin{equation*}
\mathrm{P}_{\!0}(G \geq k) = \big(e^{\lambda}/(e^{-\lambda}+e^{\lambda})\big)^k,	\quad	k \in \N_0.
\end{equation*}
In particular, $\Erm_0 [G] = e^{2\lambda}$.
Using induction, Wald's equation and \eqref{eq:trap recursion}, we conclude (a).\smallskip

\noindent
We turn to assertion (b) and fix $\kappa \geq 1$.
Using Jensen's inequality, we infer
\begin{equation*}	\textstyle
\Erm_0 [\tau_m^{\kappa}]
\geq		\Erm_0 [\tau_m]^{\kappa}
=		\big(2 \sum_{j=0}^{m-1} e^{2 \lambda j}\big)^{\kappa}
\geq		2^{\kappa}  e^{2 \kappa \lambda (m-1)},
\end{equation*}
which is the lower bound.
For the upper bound, fix $m \geq 2$,
and let $V_{i} \defeq \sum_{k=1}^{\tau_m-1} \1_{\{S_k=i\}}$ be the number of visits to the point $i$ before the random walk returns to $0$, $i=1,\ldots,m$.
Then $\tau_m = 1+\sum_{i=1}^m V_i$ and, by Jensen's inequality,
\begin{equation}	\label{eq:Etau_m^kappa bound}
\Erm_0 [\tau_m^{\kappa}]
~=~	\Erm_0 \bigg[\bigg(1+\sum_{i=1}^m V_i\bigg)^{\!\!\kappa}\bigg]
~\leq~	(m+1)^{\kappa-1} \bigg(1 + \Erm_0 \bigg[\sum_{i=1}^m V_i^{\kappa}\bigg] \bigg).
\end{equation}
In order to investigate the $V_i$, $i=1,\ldots,m$, let
\begin{equation*}
\sigma_i \defeq \inf\{k \in \N: S_k = i\}
\quad	\text{and}	\quad
r_i \defeq \Probi(\sigma_i < \sigma_0).
\end{equation*}
Given $S_0=i$, when $S_1 = i+1$, then $\sigma_i < \sigma_0$.
When the walk moves to $i-1$ in its first step,
it starts afresh there and hits $i$ before $0$ with probability $\mathrm{P}_{\!\mathit{i}-1}(\sigma_i < \sigma_0)$.
Determining $\mathrm{P}_{\!\mathit{i}-1}(\sigma_i < \sigma_0)$ is the classical ruin problem, hence
\begin{equation}	\label{eq:r_i}
r_i
~=~
\begin{cases}
\frac{e^{\lambda}}{e^{-\lambda}+e^{\lambda}} + \frac{e^{-\lambda}}{e^{-\lambda}+e^{\lambda}} \big(1- \frac{e^{2\lambda}-1}{1-e^{-2\lambda i}} e^{-2\lambda i}\big)
&	\text{for } i=1,\ldots,m-1;	\\
1- \frac{e^{2\lambda}-1}{1-e^{-2\lambda m}} e^{-2\lambda m}
&	\text{for } i=m.
\end{cases}
\end{equation}
In particular, for $i=1,\ldots,m-1$, $r_i$ does not depend on $m$.
Moreover, we have $r_1 \leq r_2 \leq \ldots \leq r_{m-1}$ and $r_1 \leq r_m \leq r_{m-1}$.
By the strong Markov property, for $k \in \N$,
$\Probnull(V_i = k)
= \Probnull(\sigma_i < \sigma_0) \, r_i^{k-1} (1-r_i)$
and hence
\begin{eqnarray*}
\Erm_0 [V_i^{\kappa}]
~=~
\sum_{k \geq 1} k^{\kappa} \Probnull(V_i = k)
& \leq &	\frac{1-r_i}{r_i} \sum_{k \geq 1} k^{\kappa} r_i^{k}	\\
& \leq &
\frac{1-r_i}{r_i} \frac{1}{|\log r_i|^{\kappa}} \bigg(2 \Big(\frac{\kappa}{e}\Big)^{\!\kappa} + \frac{\Gamma(\kappa+1)}{|\log r_i|}\bigg)
\end{eqnarray*}
where \eqref{eq:kappa moment of geometric} has been used in the last step. 
Further, for $i=1,\ldots,m-1$,
\begin{equation*}
|\log r_i|
\geq 1-r_i
=	\frac{e^{-\lambda}}{e^{-\lambda}+e^{\lambda}} \frac{e^{2\lambda}-1}{1-e^{-2\lambda i}} e^{-2\lambda i}
\geq	\frac{e^{2\lambda}-1}{e^{2\lambda}+1} e^{-2\lambda i}.
\end{equation*}
Notice that the same bound also holds for $i=m$.
Using that $r_i^{-1} \leq r_1^{-1} \leq 2$, we conclude
\begin{align}
\Erm_0 [V_i^{\kappa}]
& \leq
\frac{1-r_i}{r_i(1-r_i)^{\kappa}} \Big(2 \Big(\frac{\kappa}{e}\Big)^{\!\kappa} \! + \frac{\Gamma(\kappa+1)}{1-r_i}\Big)
\leq
\frac{2}{(1-r_i)^{\kappa}} \Big(2 \Big(\frac{\kappa}{e}\Big)^{\!\kappa} \! + \Gamma(\kappa+1)\Big)	\notag	\\
& \leq
2 \Big(2 \Big(\frac{\kappa}{e}\Big)^{\!\kappa} \! + \Gamma(\kappa+1)\Big) \Big(\frac{e^{2\lambda}+1}{e^{2\lambda}-1}\Big)^{\!\kappa}
e^{2\lambda\kappa i}	\label{eq:E V_i^kappa}
\end{align}
for $i=1,\ldots,m$.
The upper bound in (b) now follows from \eqref{eq:Etau_m^kappa bound}, \eqref{eq:E V_i^kappa} and some elementary estimates.\smallskip

\noindent
Finally, regarding assertion (c), notice that by Jensen's inequality
\begin{align*}
\Erm_0 \bigg[\bigg(\sum_{i=1}^{\tau_m} G_i\bigg)^{\!\!\kappa}\bigg]
\leq \Erm_0 \bigg[\tau_m^{\kappa-1} \sum_{i=1}^{\tau_m} G_i^\kappa \bigg]
=	\sum_{n \geq 1} \mathrm{P}_{\!0}(\tau_m=n) n^{\kappa-1} \sum_{i=1}^{n}  \Erm_0 [G_i^{\kappa}]	\\
\leq		\Erm_0 [G^{\kappa}] \Erm_0[\tau_m^{\kappa}]
\leq		\frac{r}{|\log r|^{\kappa}} \bigg(2 \Big(\frac{\kappa}{e}\Big)^{\!\kappa} + \frac{\Gamma(\kappa+1)}{|\log r|}\bigg)
c(\kappa,\lambda) m^{\kappa} e^{2\kappa\lambda m}
\end{align*}
where we have used 
\eqref{eq:kappa moment of geometric} for the last inequality.
\end{proof}

From this lemma, we derive estimates for moments of the time the walk $(Y_n)_{n \geq 0}$ spends in the $i$th trap.
For reasons that will later become transparent,
we work with $\Prob^{\circ}_{\lambda} = \Prmp^{\circ} \times P_{\omega,\lambda}$
where $\Prmp^{\circ}$ is the cycle-stationary percolation law.

\begin{lemma}	\label{Lem:time spent in the ith trap}
Suppose that $0<\kappa < \lambdacrit/\lambda$.
For $i \in \N$, let $T_i$ be the time spent by the walk $Y$ in the $i$th trap.
Then there exist constants $C(p,\kappa,\lambda)$ such that,
for fixed $p$ and $\kappa$, $C(p,\kappa,\lambda)$ is bounded on compact $\lambda$-intervals $\subseteq (0,\lambdacrit/\kappa)$
and
\begin{equation}	\label{eq:time spent in the ith trap}
\E^{\circ}_{\lambda} [T_i^{\kappa}]	\leq	C(p,\kappa,\lambda)	\quad	\text{for all } i \in \N.
\end{equation}
\end{lemma}

\auskommentiert{
Before we prove the lemma, we give a second description of the regeneration point $R_1$ and the regeneration time $\tau_1$.
Let $D:V^{\N_0} \to \N_0 \cup \{\infty\}$ denote the time of the first return to the initial state,
that is, $D((y_n)_{n \in \N_0}) \defeq \inf\{n \in \N: y_n = y_0\}$
where, as usual, $\inf \varnothing \defeq \infty$.
Put $F_0 = E_0 = M_0 = 0$ and, for $n \in \N$,
\begin{eqnarray*}	\label{eq:F_n,E_n,M_n}
F_{n}	& \defeq &	\inf\{k \in \N_0:\,Y_k \in \mathcal{R}^{\mathrm{pre}},\, X_k > M_{n-1}\},	\\
E_n	& \defeq &	D((Y_{F_n+k})_{k \geq 0}),	\\
M_n	&\defeq &	\sup\{X_k:\, 0 \leq k < E_n\}
\end{eqnarray*}
where, again, $\inf \varnothing \defeq \infty$.
The $F_n$, $n \in \N$ are called the \emph{fresh times},
$F_1$ is the first time at which the walk visits a pre-regeneration point to the right of $\mathbf{0}$.
The time of the first return to this pre-regeneration point is $E_1$. It may be infinite.
The range on the $\x$-axis that was visited by the walk
before its first return to the pre-regeneration point $Y_{F_1}$ is given by $M_1$.
The $F_n$, $E_n$, $M_n$ for $n \geq 2$ can be interpreted similarly.
Let $K \defeq \inf\{k \in \N: F_k < \infty,\, E_k = \infty\}$.
Then $\tau_1 = F_K$.
}

\begin{proof}[Proof of Lemma \ref{Lem:time spent in the ith trap}]
Suppose that $\kappa < \lambdacrit/\lambda$.
Then, for any $\omega \in \Omega^*$ and any forwards-communicating $v$,
by the same argument that leads to (24) in \cite{Axelson-Fisk+H"aggstr"om:2009b},
\begin{equation}	\label{eq:lower bound for escape probability}	\textstyle
P_{\omega,\lambda}^v(Y_n \not = v \text{ for all } n \in \N)
\geq \frac{(\sum_{k=0}^{\infty} e^{-\lambda k})^{-1}}{e^{\lambda}+1+e^{-\lambda}}
= \frac{1-e^{-\lambda}}{e^{\lambda}+1+e^{-\lambda}}
\eqdef \pesc.
\end{equation}
This bound is uniform in the environment $\omega \in \Omega^*$. Denote by $v_{i}$ the entrance of the $i$th trap.  
By the strong Markov property,
$T_i$ can be decomposed into $M$ i.i.d.~excursions into the trap:
$T_i = T_{i,1}+ \ldots + T_{i,M}$. Since $v_i$ is forwards communicating, \eqref{eq:lower bound for escape probability} implies
that  $P_{\omega,\lambda}(M \geq n) \leq (1-\pesc)^{n-1}$, $n \in \N$.
Moreover, $T_{i,1}, \ldots, T_{i,j}$ are i.i.d.~conditional on $\{M \geq j\}$.
We now derive an upper bound for $E_{\omega,\lambda}[T_{i,j}^{\kappa} | M \geq j]$. To this end, we have to take into account the times the walk stays put.
Each time, the agile walk $(Z_n)_{n \geq 0}$ makes a step in the trap,
this step is preceded by a geometric number of times the lazy walk stays put.
This geometric random variable depends on the position inside the trap,
but is stochastically bounded by a geometric random variable $G$ with
$\mathrm{P}_{\!0}(G \geq k) = \gamma^k$ for $\gamma = (1+e^{\lambda})/(e^{\lambda}+1+e^{-\lambda})$.
Lemma \ref{Lem:Etau_m}(c) then gives
\begin{equation*}	\textstyle
E_{\omega,\lambda}[T_{i,j}^{\kappa} | M \geq j] \leq \frac{\gamma}{|\log \gamma|^{\kappa}} \big(2 \big(\frac{\kappa}{e}\big)^{\kappa}
		+ \frac{\Gamma(\kappa+1)}{|\log \gamma|}\big) c(\kappa,\lambda) L_i^{\kappa} e^{2\kappa\lambda L_i},
\end{equation*}
where $L_{i}$ is the number of steps made inside the $i$th trap.
Consequently, by Jensen's inequality and the strong Markov property,
\begin{align*}
E_{\omega,\lambda}[T_i^{\kappa}]
& =
E_{\omega,\lambda}\bigg[\bigg(\sum_{j=1}^M T_{i,j}\bigg)^{\kappa}\bigg]
\leq	E_{\omega,\lambda}\bigg[M^{(\kappa-1) \vee 0} \sum_{j=1}^M T_{i,j}^{\kappa}\bigg]	\\
&=
\sum_{j \geq 1} E_{\omega,\lambda}[ M^{(\kappa-1) \vee 0} \1_{\{M \geq j\}} T_{i,j}^{\kappa}]	\\
& =
\sum_{j \geq 1} E_{\omega,\lambda}^{v_i}[(j+M)^{(\kappa-1) \vee 0}] P_{\omega,\lambda}(M \geq j) E_{\omega,\lambda}[T_{i,j}^{\kappa} | M \geq j]	\\
& \leq
C(\kappa,\lambda) L_i^{\kappa} e^{2 \kappa \lambda L_i}
\end{align*}
for some constant $0 < C(\kappa,\lambda) < \infty$ which is independent of $\omega$.
For later use, we give an upper bound for the value of $C(\kappa,\lambda)$.
For this bound, by monotonicity, we can assume without loss of generality that $\kappa \geq 2$.
First observe that
\begin{align}
&E_{\omega,\lambda}^{v_i}[M^{\kappa-1}]
\leq
(\kappa-1) \sum_{k \geq 0} (k+1)^{\kappa-1} P_{\omega,\lambda}^{v_i}(M > k)	\notag	\\
&~\leq (\kappa-1) \bigg(1+ 2\sum_{k \geq 1} k^{\kappa-1} (1-\pesc)^k\bigg)	\notag	\\
&~\leq
1+ \frac{4}{|\log (1-\pesc)|^{\kappa-1}} \bigg(\frac{(\kappa-1)^{\kappa}}{e^{\kappa-1}} + \frac{\Gamma(\kappa+1)}{|\log (1-\pesc)|}\bigg)
\label{eq:geometric kappa-1 moment}
\end{align}
by \eqref{eq:kappa moment of geometric}.
Hence, again by \eqref{eq:kappa moment of geometric},
\begin{align*}
&\sum_{j \geq 1} E_{\omega,\lambda}^{v_i} [(j\!+\!M)^{\kappa-1}] P_{\omega,\lambda}(M \! \geq \! j)
\leq
2^{\kappa-2} \! \sum_{j \geq 1} (j^{\kappa-1}\!+\!E_{\omega,\lambda}^{v_i}[M^{\kappa-1}]) P_{\omega,\lambda}(M \! \geq \! j)	\\
&\leq~
2^{\kappa-2} \bigg(\!\!\bigg(\sum_{j \geq 1} j^{\kappa-1}P_{\omega,\lambda}(M \geq j)	\\
&\hphantom{\leq~}
+\bigg(1+ \frac{4}{|\log (1\!-\!\pesc)|^{\kappa-1}}
\bigg(\frac{(\kappa\!-\!1)^{\kappa}}{e^{\kappa-1}} + \frac{\Gamma(\kappa\!+\!1)}{|\log (1\!-\!\pesc)|}\bigg)\!\!\bigg)
\sum_{j \geq 1} P_{\omega,\lambda} (M \geq j) \bigg)	\\
&\leq~
\frac{2^{\kappa-2}}{1\!-\!\pesc} \frac{1}{|\log (1\!-\!\pesc)|^{\kappa-1}} \bigg(2 \Big(\frac{\kappa-1}{e}\Big)^{\kappa-1} + \frac{\Gamma(\kappa)}{|\log (1\!-\!\pesc)|}\bigg)	\\
&\hphantom{\leq~}
+ \frac{1}{\pesc}\bigg(1+ \frac{4}{|\log (1\!-\!\pesc)|^{\kappa-1}}
\bigg(\frac{(\kappa\!-\!1)^{\kappa}}{e^{\kappa-1}} + \frac{\Gamma(\kappa\!+\!1)}{|\log (1\!-\!\pesc)|}\bigg)\!\!\bigg).
\end{align*}
In conclusion,
\begin{align}
C(\kappa,\lambda)
\leq&
\frac{\gamma}{|\log \gamma|^{\kappa}} \bigg(2 \Big(\frac{\kappa}{e}\Big)^{\kappa} + \frac{\Gamma(\kappa\!+\!1)}{|\log \gamma|}\bigg) c(\kappa,\lambda)	\notag	\\
& \cdot
\bigg(\frac{2^{\kappa-2}}{1-\pesc} \frac{1}{|\log (1-\pesc)|^{\kappa-1}}
\bigg(2 \Big(\frac{\kappa-1}{e}\Big)^{\kappa-1} + \frac{\Gamma(\kappa)}{|\log (1-\pesc)|}\bigg)	\notag	\\
&
\hphantom{\cdot \bigg(}+ \frac{1}{\pesc}\bigg(\!1+ \frac{4}{|\log (1\!-\!\pesc)|^{\kappa-1}}
\bigg(\frac{(\kappa\!-\!1)^{\kappa}}{e^{\kappa-1}} + \frac{\Gamma(\kappa\!+\!1)}{|\log (1\!-\!\pesc)|}\bigg)\!\!\bigg)\!\!\bigg).	\label{eq:C(lambda,kappa)}
\end{align}
Since $\pesc = \frac{1-e^{-\lambda}}{e^{\lambda}+1+e^{-\lambda}}$
and $\gamma = \frac{1+e^{\lambda}}{e^{\lambda}+1+e^{-\lambda}}$
take values in $(0,1)$ for $\lambda > 0$, $C(\kappa,\lambda)$ is uniformly bounded on compact $\lambda$-intervals $\subseteq (0,\infty)$.
Taking expectations w.r.t.~$\Prmp^{\circ}$ yields:
\begin{eqnarray*}
\E^{\circ}_{\lambda} [T_i^{\kappa}]
& \leq &
\sum_{m \geq 1} \Prmp^{\circ}(L_i=m) C(\kappa,\lambda) m^{\kappa} e^{2 \kappa \lambda m}	\\
& = &
c(p) C(\kappa,\lambda) \sum_{m \geq 1} m^{\kappa} e^{2\kappa\lambda m} e^{-2\lambdacrit m}
~\eqdef~	C(p,\kappa,\lambda)	~<~	\infty
\end{eqnarray*}
since $\lambda\kappa < \lambdacrit$.
Since $C(\kappa,\lambda)$ is bounded on all compact $\lambda$-intervals $\subseteq (0,\infty)$,
$C(p,\kappa,\lambda)$ remains bounded on all compact $\lambda$-intervals $\subseteq (0,\lambdacrit/\kappa)$
(when $\kappa$ is fixed).
\end{proof}

\subsection{Quenched return probabilities.}

Recall that $Z^{\mathcal{B}} = (Z_0^{\mathcal{B}}, Z_1^{\mathcal{B}},\ldots)$
denotes the agile walk on the backbone $\mathcal{B}$.
For $v \in V$, let $\sigma_v \defeq \inf\{k \in \N: Z_k^{\mathcal{B}} = v\}$ and,
for $m \in \Z$, let $\sigma_m \defeq \sigma_{(m,0)} \wedge \sigma_{(m,1)}$.

\begin{lemma}	\label{Lem:stepping back}
Let $m \in \N$ and $v \in \mathcal{B}$ with $\x(v)=m$.
Then, for any $k > m$,
\begin{equation}	\label{eq:stepping back}
P_{\omega,\lambda}^{v}(\sigma_0 < \sigma_{k})
~\leq~	\frac{2(e^{2\lambda}-1)}{e^{\lambda}-1} \frac{1-e^{-2\lambda(k-m)}}{1-e^{-2 \lambda m}} e^{-2 \lambda m}
\end{equation}
uniformly for all $\omega \in \Omega_{\mathbf{0}}$ with $R_0^{\mathrm{pre}} = \mathbf{0}$.
In particular,
\begin{equation*}
P_{\omega,\lambda}^{v}(\sigma_0 < \infty)	~\leq~	\frac{2(e^{2\lambda}-1)}{e^{\lambda}-1} \frac{1}{e^{2 \lambda m}-1}
~\eqdef~	\frac{C(\lambda)}{e^{2 \lambda m}-1}.
\end{equation*}
\end{lemma}
\begin{proof}
The agile walk
$(Z^{\mathcal{B}}_n)_{n \geq 0}$ can be seen as the Markov chain induced by the (infinite) electric network with conductances
\begin{equation*}
C_{\mathcal{B}}(u,v) ~=~	\begin{cases}
			e^{\lambda (\x(u)+\x(v))}		&	\text{if } u,v \in \mathcal{B} \text{ and } \omega(\langle u,v \rangle) = 1,	\\
			0							&	\text{otherwise.}
			\end{cases}
\end{equation*}
We use Formula (4) of \cite{Berger+Gantert+Peres:2003}:
\begin{equation}
P_{\omega,\lambda}^{v}(\sigma_0 < \sigma_{k})
~\leq~	\frac{\mathcal{R}_{\mathcal{B}}(v \leftrightarrow \{(k,0),(k,1)\})}{\mathcal{R}_{\mathcal{B}}(v \leftrightarrow \mathbf{0})}
\end{equation}
where $\mathcal{R}_{\mathcal{B}}(v \leftrightarrow \mathbf{0})$ denotes the effective resistance between $v$ and $\mathbf{0}$ in the given electrical network
and $\mathcal{R}_{\mathcal{B}}(v\leftrightarrow \{(k,0),(k,1)\})$ is the effective resistance between $v$ and $\{(k,0),(k,1)\}$.
Since $v \in \mathcal{B}$, there is a non-backtracking path connecting $v$ and the set $\{(k,0),(k,1)\}$.
By Raleigh's monotonicity law \cite[Theorem 9.12]{Levin+Peres+Wilmer:2009},
$\mathcal{R}_{\mathcal{B}}(v\leftrightarrow \{(k,0),(k,1)\})$ is bounded from above by the resistance of that path.
By the series law, the latter is at most $\sum_{j=2m}^{2k-1} e^{-j\lambda} = e^{-2 \lambda m}(1-e^{-2 \lambda (k-m)})/(1-e^{-\lambda})$.
A lower bound for $\mathcal{R}_{\mathcal{B}}(v \leftrightarrow \mathbf{0})$
can be obtained from the Nash-Williams inequality \cite[Proposition 9.15]{Levin+Peres+Wilmer:2009}.
The $\Pi_j \defeq \{\langle(j-1,i),(j,i)\rangle: i=0,1\}$, $j=1,\ldots,m$ form disjoint edge-cutsets and hence the cited inequality gives
\begin{equation*}	\textstyle
\mathcal{R}_{\mathcal{B}}(v \leftrightarrow \mathbf{0})
\geq	\sum_{j=1}^m \big(\sum_{e \in \Pi_j} C_{\mathcal{B}}(e)\big)^{-1}
\geq	\sum_{j=1}^m(2 e^{\lambda(2j-1)})^{-1}
=	\frac{1}{2} \frac{1-e^{-2\lambda m}}{e^{\lambda}-e^{-\lambda}}.
\end{equation*}
The two bounds combined give \eqref{eq:stepping back}.
\end{proof}

\subsection{Uniform regeneration estimates.}

We are almost ready to prove Lemma \ref{Lem:moments of rho and tau}.
Before we do so, we derive a uniform upper bound
for the tails of $\rho_1$. In fact, for later use,
we prove an even stronger result.

\begin{lemma}	\label{Lem:uniform exponential regeneration point estimates}
For every compact interval $I = [\lambda_1,\lambda_2] \subseteq (0,\infty)$,
there are finite constants $C=C(I,p)$ and $\varepsilon = \varepsilon(I,p) > 0$
(depending only on $I,p$) such that
\begin{equation}	\label{eq:uniform exponential regeneration point estimates}
\sup_{n \in \N_0} \sup_{\lambda \in I} \Prob_{\lambda}(\rho_{\nu(n)}-X_n \geq k)	~\leq~	C(I,p) e^{-\varepsilon k}	\quad	\text{for all } k \in \N_0.
\end{equation}
The same statement holds true with $\Prob_{\lambda}$ replaced by $\Prob^{\circ}_{\lambda}$.
\end{lemma}
\begin{proof}
Let $D:V^{\N_0} \to \N_0 \cup \{\infty\}$ denote the time of the first return to the initial state,
that is, $D((y_n)_{n \in \N_0}) \defeq \inf\{n \in \N: y_n = y_0\}$
where, as usual, $\inf \varnothing \defeq \infty$.
Further, let $n \in \N_0$ and put $F_0(n) \defeq E_0(n) \defeq n$ and $M_0(n) \defeq \max_{j=0,\ldots,n} X_j$.
For $k \in \N$, define
\begin{eqnarray*}
F_k(n)	&  \defeq  &	\inf\{j \in \N_0:\,Y_j \in \mathcal{R}^{\mathrm{pre}},\, X_j > M_{k-1}(n)\},	\\
E_k(n)	&  \defeq  &	D((Y_{F_k(n)+j})_{j \geq 0}),	\\
M_k(n)	& \defeq  &	\sup\{X_j:\, 0 \leq j < E_k(n)\}
\end{eqnarray*}
where $\inf \varnothing = \infty$.
In particular, $F_{1}(n)$ is the first time after time $n$ that a pre-regeneration point is visited.
We call the $F_{k}(n)$ \emph{fresh times}.
Let $K(n) \defeq \inf\{k \in \N: F_k(n) < \infty, E_k(n) = \infty\}$.
Notice that $F_{K(n)}(n) = \tau_{\nu(n)}$ and, hence, $X_{F_{K(n)}(n)} = \rho_{\nu(n)}$.
Fix an interval $I = [\lambda_1,\lambda_2] \subseteq (0,\infty)$.
By \eqref{eq:lower bound for escape probability},
\begin{equation}	\label{eq:geometric bound for K(n)}	\textstyle
\Prob_{\lambda}(K(n) \geq k)	~\leq~	(1-\pesc)^{k-1},	\quad	k \in \N.
\end{equation}
We define	
\begin{equation*}
\H_{F_k(n)}	~=~	\sigma((F_k(n), Y_{0}, \ldots,Y_{F_k(n)}), \omega_i: \x(R_i^{\mathrm{pre}}) \leq X_{F_k(n)}),	\quad	k \in \N_0.
\end{equation*}
Then, for $k \geq 2$,
\begin{align}
\Prob_{\lambda}(X_{F_k(n)}-X_{F_{k-1}(n)} \in \cdot,E_{k}(n) < \infty \mid \H_{F_{k-1}(n)})	&	\notag	\\
~=~	\Prob^{\circ}_{\lambda}(X_{F} \in \cdot, E < \infty)	&
\text{ on } \{F_{k-1}(n) < \infty\}					\label{eq:conditional iid X_F_k(n)}
\end{align}
where $F \defeq F_1(0)$ and $E \defeq E_1(0)$.
Recall that $T'_{m:2m}$ denotes the event that $[m,2m)$ is contained in a trap piece. Thus, for $m \in \N$,
\begin{equation}	\label{eq:union bound}
\Prob^{\circ}_{\lambda}(X_F \geq 2m, E < \infty)
~\leq~
\Prmp^{\circ}(T'_{m:2m}) + \Prob^{\circ}_{\lambda}(M^{\mathcal{B}} \geq m, E < \infty)
\end{equation}
where $M^{\mathcal{B}} \defeq \sup\{X_k: k < E \text{ and } X_k \in \mathcal{B}\} = \sup\{\x(Z_k^{\mathcal{B}}): k < \sigma_{\mathbf{0}}\}$.
The last probability in \eqref{eq:union bound} can be bounded using Lemma \ref{Lem:stepping back}:
\begin{eqnarray*}
\Prob^{\circ}_{\lambda}(M^{\mathcal{B}} \geq m, E < \infty)
\leq
C(\lambda) e^{-2 \lambda  m}.
\end{eqnarray*}
Using that $\Prmp^{\circ}(T'_{m:2m}) \leq e^{-2\lambdacrit m}$ by Lemma \ref{Lem:trap cover},
we get that
\begin{equation}	\label{eq:distances between fresh points}
\Prob^{\circ}_{\lambda}(X_{F} \geq 2m, E < \infty)
\leq
C(\lambda) e^{-2 \lambda  m} + e^{-2\lambdacrit m}
\leq
C_1 e^{-2 (\lambda_1 \wedge \lambdacrit)  m}
\end{equation}
where $C_1 = 1+\max_{\lambda \in I} C(\lambda)$ depends only on $I$.
Further, for $m \in \N$,
\begin{align}
\Prob_{\lambda}(X_{F_1(n)} - X_n > 5m)
& \leq
\Prob_{\lambda}(M_0(n)-X_n \geq 2m)	\notag	\\
& \hphantom{\leq}
+ \Prob_{\lambda}\bigg(\min_{\substack{v \in \mathcal{R}^{\mathrm{pre}}:\\ \x(v)>M_0(n)}} \x(v) - M_0(n) > 3m \bigg).	\label{eq:Prob(X_F_1(n) large)}
\end{align}
Regarding the first probability on the right-hand side, notice that $M_0(n)-X_n \geq 2m$
requires an excursion of $(Y_k)_{k \geq 0}$ on the backbone at least to $\x$-coordinate $X_n+m$
and afterwards a return to $\x$-coordinate $X_n$ or the presence of a trap piece covering $[m,2m)$.
According to Lemma \ref{Lem:stepping back}, the probability of the first event is bounded by $C(\lambda)/(e^{2 \lambda m}-1)$,
while the probability of the second event is bounded by $e^{-2 \lambdacrit m}$ according to Lemma \ref{Lem:trap cover}.
Hence, $\Prob_{\lambda}(M_0(n)-X_n \geq 2m) \leq C(\lambda)/(e^{2 \lambda m}-1) + e^{-2 \lambdacrit m}$.
For the second probability, a standard geometric trials argument for the Markov chain $(({\tt T}_i,\eta_i))_{i \in \Z} = (({\tt T}_i,\omega(E^{i-1,>} \cap E^{i+1,<})))_{i \in \Z}$
from the proof of Lemma \ref{Lem:pre-regeneration point decomposition}
shows that
\begin{equation*}
\Prob_{\lambda}\Big(\min_{v \in \mathcal{R}^{\mathrm{pre}}: \x(v)>M_0(n)} \x(v) > m \Big)
~\leq~ c^m
\end{equation*}
for a suitable constant $c = c(p) \in (0,1)$, which depends only on $p$.
Hence,
\begin{eqnarray}
\Prob_{\lambda}(X_{F_1(n)} - X_n > 5m)
& \leq &
C(\lambda)/(e^{2 \lambda m}-1) + e^{-2 \lambdacrit m} + c^{3m}	\notag	\\
& \leq &
C_2 e^{-\varepsilon_1 m}	\label{eq:Prob(X_F_1(n) large) II}
\end{eqnarray}
where $C_2 < \infty$ and $\varepsilon_1 > 0$ are constants depending only on $I$ and $p$.
After these preparations, we are ready to estimate
$\Prob_{\lambda}(\rho_{\nu(n)}-X_n \geq k)$ uniformly in $\lambda \in I = [\lambda_1,\lambda_2]$ and $n \in \N_0$.
For $r>0$, using \eqref{eq:conditional iid X_F_k(n)}, we have
\begin{align}	\label{eq:Prob_lambda(rho_nu(n)-X_n >= k)}
\Prob_{\lambda}(\rho_{\nu(n)}-X_n \geq k)
~\leq~
\Prob_{\lambda}(K(n) > k/r) + \Prob_{\lambda}(\xi_1+\ldots+\xi_{\lfloor k/r \rfloor} \geq k)
\end{align}
where $\xi_1, \ldots, \xi_{\lfloor k/r \rfloor}$ are independent random variables with
$\xi_1$ having the same distribution as $X_{F_1(n)} - X_n$
and $\xi_2, \ldots, \xi_{\lfloor k/r \rfloor}$ having the same distribution as $X_{F} \1_{\{E < \infty\}}$ under $\Prob^{\circ}_{\lambda}$.
According to \eqref{eq:geometric bound for K(n)},
the first probability on the right-hand side of \eqref{eq:Prob_lambda(rho_nu(n)-X_n >= k)}
is bounded above by $(1-\pesc)^{\lfloor k/r \rfloor}$.
By Markov's inequality, for any $u>0$, the second probability is bounded by
\begin{align*}
\Prob_{\lambda}(\xi_1+\ldots+\xi_{\lfloor k/r \rfloor} \geq k)
~&\leq~
e^{-uk} \E_{\lambda}[\exp(u(\xi_1+\ldots+\xi_{\lfloor k/r \rfloor}))]	\\
~&\leq~
e^{-uk} \E_{\lambda}[e^{u\xi_1}] \E_{\lambda}[e^{u\xi_2}]^{k/r}.
\end{align*} 
By \eqref{eq:Prob(X_F_1(n) large) II},
\begin{eqnarray}
\E_{\lambda}[e^{u\xi_1}]
& = &
1+(e^u-1) \sum_{k \geq 0} e^{uk} \Prob_{\lambda}(\xi_1 > k)	\notag	\\
& \leq &
1+(e^u-1) 5e^{4u} \sum_{k \geq 0} e^{5uk} \Prob_{\lambda}(\xi_1 > 5k)	\notag	\\
& \leq &
1+5 C_2 e^{4u} (e^u-1) \sum_{k \geq 0} e^{5uk} e^{-\varepsilon_1 k}
~\eqdef~	C_3(u)	\label{eq:bound on E_lambda[e^u xi_1]}
\end{eqnarray}
where $C_3(u)$ is a positive constant depending only on $p, I$ and $u$.
Further, $C_3(u)$ is finite for all sufficiently small $u$.
Analogously, using \eqref{eq:Prob(X_F_1(n) large)} we find
\begin{equation}\label{eq:bound on E_lambda[e^u xi_2]}	\textstyle
\E_{\lambda}[e^{u\xi_2}]
\leq \exp\big(2 C_1 \frac{e^u(e^u-1)}{1-e^{2(u-(\lambda_1 \wedge \lambdacrit))}}\big)		
\end{equation}
for $u < \lambda_1 \wedge \lambdacrit$.
Now fix $u < \lambda_1 \wedge \lambdacrit$ so small that $C_3(u) < \infty$
and choose $r$ so large that
\begin{equation*}	\textstyle
\frac{2 C_1 e^u(e^u-1)}{r(1-e^{2(u-(\lambda_1 \wedge \lambdacrit))})} - u \eqdef -\varepsilon_1	~<~	0.
\end{equation*}
Then
\begin{equation*}	\textstyle
\Prob_{\lambda}(\xi_1+\ldots+\xi_{\lfloor k/r \rfloor} \geq k)
\leq
C_3(u) e^{-uk} \exp\Big(\frac{2 C_1 e^u(e^u-1)}{1-e^{2(u-(\lambda_1 \wedge \lambdacrit))}} \frac{k}{r}\Big)
=
C_3(u) e^{-\varepsilon_1 k}.
\end{equation*}
We use this estimate together with \eqref{eq:geometric bound for K(n)} in \eqref{eq:Prob_lambda(rho_nu(n)-X_n >= k)}
to conclude that
\begin{align*}
\Prob_{\lambda}(\rho_{\nu(n)}-X_n \geq k)
~&\leq~
(1-\pesc)^{\lfloor \frac kr \rfloor} + C_3(u) e^{-\varepsilon_1 k}
\end{align*}
for all $\lambda \in [\lambda_1,\lambda_2]$, $n \in \N_0$.
This implies \eqref{eq:uniform exponential regeneration point estimates} after some minor manipulations.

It remains to point out that the exact same argument works when $\Prob_{\lambda}$ is replaced by $\Prob^{\circ}_{\lambda}$.
\end{proof}

\subsection{Moments of regeneration points and times.}

We are now ready for the proof of Lemma \ref{Lem:moments of rho and tau}.

\begin{proof}[Proof of Lemma \ref{Lem:moments of rho and tau}]
In view of Lemma \ref{Lem:iid regeneration times and points},
we need to show that
\begin{equation}	\label{eq:epsilon-exponential moment}
\E^{\circ}_{\lambda}[\exp(\varepsilon \rho_1) \mid Y_k \neq \mathbf{0} \text{ for all } k \geq 1]	~<~	\infty
\end{equation}
for some $\varepsilon > 0$ and that
\begin{equation}	\label{eq:kappa moment}
\E^{\circ}_{\lambda}[\tau_1^{\kappa} \mid Y_k \neq \mathbf{0} \text{ for all } k \geq 1]	~<~	\infty
\quad	\text{iff}	\quad
\kappa < \lambdacrit/\lambda.
\end{equation}
From \eqref{eq:lower bound for escape probability}, we get
\begin{align*}
\E^{\circ}_{\lambda}[\exp&(\varepsilon \rho_1) \mid Y_k \neq \mathbf{0} \text{ for all } k \geq 1]	\\
& ~=~
\Ermp^\circ \big[E_{\omega,\lambda} [\exp(\varepsilon \rho_1) \mid Y_k \neq \mathbf{0} \text{ for all } k \geq 1]	\big]\\
& ~\leq~
\pesc^{-1} \Ermp^\circ \big[E_{\omega,\lambda} [\exp(\varepsilon \rho_1) \1_{\{Y_k \neq \mathbf{0} \text{ for all } k \geq 1\}}]\big]
~\leq~
\pesc^{-1} \E^{\circ}_{\lambda} [\exp(\varepsilon \rho_1)],
\end{align*}
and analogously
\begin{equation}	\label{eq:trivial bound E[tau_2-tau_1]}
\E^{\circ}_{\lambda}[\tau_1^{\kappa} \mid Y_k \neq \mathbf{0} \text{ for all } k \geq 1]
~\leq~	\pesc^{-1} \E^{\circ}_{\lambda}[\tau_1^{\kappa} \1_{\{Y_k \neq \mathbf{0} \text{ for all } k \geq 1\}}].
\end{equation}
Assertion (a) now follows from Lemma \ref{Lem:uniform exponential regeneration point estimates} with $I = \{\lambda\}$ and $n=0$.

The fact that $\E_{\lambda}[\tau_2-\tau_1] = \infty$ for $\kappa \geq \lambda$
follows from the lower bound in Lemma \ref{Lem:uniform lower and upper tail bounds} below.

Now assume that $\lambda < \lambdacrit/\kappa$.
We decompose
\begin{equation}	\label{eq:tau_1 decomposed}
\tau_1 ~=~ \tau_1^{\mathcal{B}} + \tau_1^{\mathrm{traps}}
\end{equation}
where $\tau_1^{\mathcal{B}} \defeq \#\{0 \leq k < \tau_1: Y_k \in \mathcal{B}\}$
and $\tau_1^{\mathrm{traps}} = \tau_1-\tau_1^{\mathcal{B}}$ is the time spent by the walk in the traps, that is, in $\Cluster_{\infty} \setminus \mathcal{B}$.
We proceed with a lemma that provides an estimate for $\tau_1^{\mathcal{B}}$:

\begin{lemma}	\label{Lem:speed of Y in backbone}
$\E^{\circ}_{\lambda}[(\tau_1^{\mathcal{B}})^{\gamma}\1_{\{Y_k \neq \mathbf{0} \text{ for all } k \geq 1\}}] < \infty$ for all $\gamma > 0$.
\end{lemma}

The proof of the lemma is postponed.
Taking its assertion for granted, it remains to prove that
$\E^{\circ}_{\lambda}[(\tau_1^{\mathrm{traps}})^{\kappa}\1_{\{Y_k \neq \mathbf{0} \text{ for all } k \geq 1\}}] < \infty$.
To this end, fix $r,s>1$ such that $\kappa \lambda s < \lambdacrit$ and $1/r+1/s=1$. Then
\begin{eqnarray*}
\E^{\circ}_{\lambda}[(\tau_1^{\mathrm{traps}})^{\kappa}\1_{\{Y_k \neq \mathbf{0} \text{ for all } k \geq 1\}}]
& \leq &
\E^{\circ}_{\lambda}\bigg[ \sum_{n \geq 1} \1_{\{\rho_1=n\}}\bigg(\sum_{i=1}^{n-1} T_i\bigg)^{\kappa}\bigg]	\\
& \leq &
\E^{\circ}_{\lambda}\bigg[ \sum_{n \geq 1} \1_{\{\rho_1=n\}} n^{\kappa-1} \sum_{i=1}^{n-1} T_i^{\kappa}\bigg]	\\
& \leq &
\sum_{n \geq 1} n^{\kappa-1} \Prob^{\circ}_{\lambda}(\rho_1=n)^{1/r} \bigg[ \sum_{i=1}^{n-1} \E^{\circ}_{\lambda}[ T_i^{\kappa s}]^{1/s}\bigg]
\end{eqnarray*}
where H\"older's inequality has been used in the last step.
From \eqref{eq:time spent in the ith trap} we infer
\begin{equation}	\label{eq:upper bound E[tau_1^traps^kappa]}	\textstyle
\E^{\circ}_{\lambda}[(\tau_1^{\mathrm{traps}})^{\kappa}]
\leq C(p,\kappa s, \lambda)^{1/s} \sum_{n \geq 1} n^{\kappa} \Prob^{\circ}_{\lambda}(\rho_1=n)^{1/r}.
\end{equation}
The latter sum is finite due to Lemma \ref{Lem:moments of rho and tau}.
\end{proof}

\begin{proof}[Proof of Lemma \ref{Lem:speed of Y in backbone}]
Fix $\gamma > 1$.
For every $v \in V$, let $N(v) \defeq \#\{k \geq 0: Y_k = v\}$ be the number of visits of $Y$ to $v$.
Then
\begin{align}
\E^{\circ}_{\lambda} & [(\tau_1^{\mathcal{B}})^{\gamma}\1_{\{Y_k \neq \mathbf{0} \text{ for all } k \geq 1\}}]
~\leq~
\E^{\circ}_{\lambda} \bigg[\bigg(\sum_{v \in \mathcal{B}, \, 0 \leq \x(v) < \rho_1} N(v)\bigg)^{\gamma}\bigg]		\notag	\\
& =~
\E^{\circ}_{\lambda} \bigg[\sum_{n \geq 1} \1_{\{\rho_1=n\}} \bigg(\sum_{v \in \mathcal{B}, \, 0 \leq \x(v) < n} N(v)\bigg)^{\gamma}\bigg]		\notag	\\
& \leq~
\sum_{n \geq 1} \E^{\circ}_{\lambda} \bigg[\1_{\{\rho_1=n\}} (2n)^{\gamma-1} \bigg(\sum_{v \in \mathcal{B}, \, 0 \leq \x(v) < n} N(v)^{\gamma} \bigg) \bigg]		\notag	\\
& \leq~
\sum_{n \geq 1} \E^{\circ}_{\lambda} [\1_{\{\rho_1=n\}} (2n)^{2(\gamma-1)}]^{1/2}  \bigg(\sum_{0 \leq \x(v) < n} \E^{\circ}_{\lambda}  [\1_{\{v \in \mathcal{B}\}} N(v)^{2\gamma}]^{1/2} \bigg)	\label{eq:estimate tau_1 backbone}
\end{align}
where the last inequality is a consequence of the Cauchy-Schwarz inequality.
Now arguing as in the paragraph following \eqref{eq:lower bound for escape probability},
one infers that, for $v \in \mathcal{B}$,
$P_{\omega,\lambda}(N(v) \geq k) \leq (1-\pesc)^{k-1}$ where $\pesc$ is as defined in \eqref{eq:lower bound for escape probability}.
Therefore,
\begin{eqnarray*}
\E^{\circ}_{\lambda}[\1_{\{v \in \mathcal{B}\}} N(v)^{2\gamma}]
& = &
\Ermp^\circ[ \1_{\{v \in \mathcal{B}\}} E_{\omega,\lambda}[N(v)^{2\gamma}]]	\\
& \leq &
2\gamma \Ermp^\circ\bigg[ \1_{\{v \in \mathcal{B}\}} \sum_{k \geq 1} k^{2\gamma-1}P_{\omega,\lambda}(N(v) \geq k) \bigg]	\\
& \leq &
2\gamma \sum_{k \geq 1} k^{2\gamma-1}(1-\pesc)^{k-1}	~\eqdef~	C^\mathcal{B}(\gamma,\lambda)^2 ~<~ \infty.
\end{eqnarray*}
Using this and Lemma \ref{Lem:moments of rho and tau}(a) in \eqref{eq:estimate tau_1 backbone} leads to:
\begin{equation}	\label{eq:upper bound E[tau_1^B^kappa]}	\textstyle
\E^{\circ}_{\lambda} [(\tau_1^{\mathcal{B}})^{\gamma}\1_{\{Y_k \neq \mathbf{0} \text{ for all } k \geq 1\}}]
\leq
C^\mathcal{B}(\gamma,\lambda) \sum_{n \geq 1} \E^{\circ}_{\lambda} [\1_{\{\rho_1=n\}} (2n)^{2\gamma}]^{1/2}
<	\infty.
\end{equation}
\end{proof}


\subsection{Further uniform regeneration estimates.}

In several proofs involving simultaneous limits in $\lambda$ and $n$,
we need uniform regeneration estimates.

For the next result, recall that $\nu(n) = \inf\{k \in \N: \tau_k > n\}$ for $n \in \N_0$.

\begin{lemma}	\label{Lem:uniform regeneration estimates}
\begin{itemize}
	\item[(a)]
		The functions
		$\lambda \mapsto \sup_{n \in \N_0} \E_{\lambda}[\rho_{\nu(n)}-X_n]$,
		$\lambda \mapsto \E_{\lambda}[\rho_1]$ and $\lambda \mapsto \E_{\lambda}^\circ[\rho_1]$
		are locally bounded on $(0,\infty)$.
	\item[(b)]
		The function $\lambda \mapsto \E_{\lambda}[\tau_1]$ is locally bounded on $(0,\lambdacrit)$.
	\item[(c)]
		The function $\lambda \mapsto \sup_{n \in \N_0} \E_{\lambda}^{\circ}[\tau_{\nu(n)}- n \mid Y_k \neq \mathbf{0} \text{ for all } k \geq 1]$
		is locally bounded on $(0,\lambdacrit/2)$.
		For every interval $I = [\lambda_1,\lambda_2] \subseteq (0,\lambdacrit)$ and every $1 < r < \frac{\lambdacrit}{\lambda_2} \wedge 2$,
		\begin{equation*}	\textstyle
		n^{-1/r} \sup_{\lambda \in I} \E_{\lambda}^{\circ}[\tau_{\nu(n)}- n \mid Y_k \neq \mathbf{0} \text{ for all } k \geq 1] \to 0
		\quad	\text{as } n \to \infty.
		\end{equation*}
\end{itemize}
\end{lemma}


We postpone the proof.
Lemma \ref{Lem:uniform regeneration estimates} allows us to finish the proof of Theorem \ref{Thm:continuity of the speed}:

\begin{proof}[Proof of Theorem \ref{Thm:continuity of the speed}]
Let $\lambda^* \in (0,\lambdacrit)$ and $1 < r < \frac{\lambdacrit}{\lambda^*} \wedge 2$.
As a consequence of Lemma \ref{Lem:uniform bound X_n-n vel},
we have
\begin{equation*}
\lim_{\substack{\lambda \to \lambda^*,\\ (\lambda-\lambda^*)^{r-1}n \to \infty}}
\bigg[\frac{\vel(\lambda)-\vel(\lambda^*)}{(\lambda-\lambda^*)^{r-1}} - \frac{\E_{\lambda}[X_n]-\E_{\lambda^*}[X_n]}{(\lambda-\lambda^*)^{r-1}n}\bigg]
~=~ 0.
\end{equation*}
Therefore, for arbitrary $\alpha > 0$,
\begin{equation*}
\lim_{\lambda \to \lambda^*}
\frac{\vel(\lambda)-\vel(\lambda^*)}{(\lambda-\lambda^*)^{r-1}}
=  \lim_{\substack{\lambda \to \lambda^*,\\ (\lambda-\lambda^*)^{r}n \to \alpha}} \frac{\E_{\lambda}[X_n]-\E_{\lambda^*}[X_n]}{(\lambda-\lambda^*)^{r-1}n}
= 0
\end{equation*}
by Proposition \ref{Prop:3rd step again}.

It remains to show that $\vel(\lambda)$ is continuous at $\lambda = \lambdacrit$,
that is, $\lim_{\lambda \uparrow \lambdacrit} \vel(\lambda)=0$.
By \eqref{eq:velocity}, we have
$\vel(\lambda) = \E_{\lambda}[\rho_2-\rho_1]/\E_{\lambda}[\tau_2-\tau_1]$.
Here,
\begin{equation*}	\textstyle
\E_{\lambda}[\rho_2-\rho_1] = \E_{\lambda}^\circ[\rho_1 \mid Y_k \neq \mathbf{0} \text{ for all } k \geq 1]
\leq \pesc^{-1} \E_{\lambda}^\circ [\rho_1]
\end{equation*}
where $\pesc$ is the escape probability bound defined in \eqref{eq:lower bound for escape probability},
see the beginning of the proof of Lemma \ref{Lem:moments of rho and tau} for details on this estimate.
The function $\lambda \mapsto \E_{\lambda}^\circ [\rho_1]$ is locally bounded on $(0,\infty)$
according to Lemma \ref{Lem:uniform regeneration estimates}.
Now let $\lambda < \lambdacrit$.
The probability under $\Prmp^\circ$ that there is a trap of length $m$
with trap entrance at $(1,0)$ is given by $\epsilon(p) e^{-2 \lambdacrit m}$
for a constant $\epsilon(p) > 0$ which depends only on $p$.
The walk steps into that trap immediately with probability $e^{2\lambda}/(e^{\lambda}+1+e^{-\lambda})^2$,
hence we obtain from Lemma \ref{Lem:Etau_m}(b) and the Markov property of $Y$ under $P_{\omega,\lambda}$
that
\begin{align*}
\E^{\circ}_{\lambda} [\tau_1 \mid Y_k \neq \mathbf{0} \text{ for all } k \geq 1]
&\geq \frac{e^{2\lambda}}{(e^{\lambda}+1+e^{-\lambda})^2}
\cdot \sum_{m=1}^\infty \epsilon(p) e^{-2 \lambdacrit m} 2 e^{2 \lambda (m-1)}	\\
&=\frac{2 \epsilon(p) e^{2(\lambda-\lambdacrit)}}{(e^{\lambda}+1+e^{-\lambda})^2}
\cdot \frac{1}{1-e^{-(\lambdacrit-\lambda)}}.
\end{align*}
This bound is of the order $(\lambdacrit-\lambda)^{-1}$ as $\lambda \to \lambdacrit$.
The proof is complete.
\end{proof}

\begin{lemma}	\label{Lem:uniform lower and upper tail bounds}
Let $p \in (0,1)$ be fixed. Then for every compact interval $I = [\lambda_1,\lambda_2] \subseteq (0,\infty)$
and every $\lambda^* > \lambda_2$,
there are positive and finite constants $\underline{C}(I,p)$ depending only on $p$ and $I$
and $\overline{C}(I,\lambda^*,p)$ depending only on $I,p,\lambda^*$ such that
\begin{equation}	\label{eq:lower and upper bounds tails of tau}
\underline{C}(I,p) k^{-\lambdacrit/\lambda_1}
\leq
\Prob_{\lambda}(\tau_2 -\tau_1 \geq k)
\leq
\overline{C}(I,p,\lambda^*) k^{-\lambdacrit/\lambda^*}
\end{equation}
for all $k \in \N$.
\end{lemma}

\begin{remark}	\label{Rem:uniform lower and upper tail bounds}
If one chooses $\lambda_1=\lambda_2=\lambda > 0$ in the above lemma,
then, with $\alpha = \lambdacrit/\lambda$ and arbitrary $\kappa<\alpha$,
the lemma gives that $\Prob_{\lambda}(\tau_2 -\tau_1 \geq k)$ is bounded below by
a constant times $k^{-\alpha}$ and bounded above by a constant times $k^{-\kappa}$.
The correct order is in fact $k^{-\alpha}$.
We refrain from proving this as we do not require this precision.
\end{remark}

\begin{proof}
Let $I = [\lambda_1,\lambda_2]$ be as in the lemma and $\lambda^* > \lambda_2$.

\noindent
We begin with the proof of the lower bound.
Under $\Prmp^\circ$, the cluster has a pre-regeneration point at $\mathbf{0}$ a.\,s.
Let $I_m$ denote the event that immediately to the right of the pre-regeneration point at $\mathbf{0}$
there is a trap of length $m$ with trap entrance at $(1,0)$.
Then $\Prmp^\circ(I_m) = \epsilon(p) e^{-2 \lambdacrit m}$
where $\epsilon(p)$ is a positive constant depending only on $p$.
For every $\omega \in I_m$, $m \in \N$,
the probability that the walk $(Y_n)_{n \geq 0}$ steps into the first trap
and then first hits the bottom of the trap before
returning to the trap entrance is given by
\begin{equation*}	\textstyle
\frac{e^{2\lambda}}{(e^{\lambda}+1+e^{-\lambda})^2} \cdot \frac{e^{-2\lambda}-e^{-2 \lambda m}}{1-e^{-2 \lambda m}},
\end{equation*}
where we have used the Gambler's ruin probabilities.
Once the walk hits the bottom of the trap,
it will make several attempts to return to the trap entrance
until it finally hits the trap entrance.
The probability that the walk then escapes without ever backtracking to the trap entrance
(and in particular to the origin) is bounded below by $\pesc$.
Denote the number of attempts to return to the trap entrance by $N$.
(More precisely, $N$ is the number of times the walk moves from the bottom of the trap
one step to the left).
Again using the Gambler's ruin probabilities,
we conclude that starting from the bottom of the trap,
the number of unsuccessful attempts to return to the trap entrance is $\geq k$
with probability
\begin{equation*}	\textstyle
\big(\frac{1-e^{-2\lambda(m-1)}}{1-e^{-2 \lambda m}}\big)^{k-1}.
\end{equation*}
Therefore, on $I_m$, we have
\begin{align*}
&P_{\omega,\lambda}(\tau_1 \geq k, Y_k \neq \mathbf{0} \text{ for all } k \geq 1)	\\
&~\geq \frac{e^{2\lambda}}{(e^{\lambda}+1+e^{-\lambda})^2} \cdot \frac{e^{-2\lambda}-e^{-2 \lambda m}}{1-e^{-2 \lambda m}} 
\cdot \Big(\frac{1-e^{-2\lambda(m-1)}}{1-e^{-2 \lambda m}}\Big)^{k-1} \cdot \pesc.
\end{align*}
Consequently, for every $m \in \N$, we have
\begin{align*}
&\Prob_{\lambda}^\circ(\tau_1 \geq k \mid Y_k \neq \mathbf{0} \text{ for all } k \geq 1)	\\
&\geq  \epsilon(p) \pesc \frac{e^{2\lambda}}{(e^{\lambda}+1+e^{-\lambda})^2}
\cdot \frac{e^{-2\lambda}-e^{-2 \lambda m}}{1-e^{-2 \lambda m}} e^{-2 \lambdacrit m} \Big(\frac{1-e^{-2\lambda(m-1)}}{1-e^{-2 \lambda m}}\Big)^{k-1}.
\end{align*}
The first three factors are clearly bounded away from $0$ as $\lambda$ varies in $[\lambda_1,\lambda_2]$.
The last three factors depend on $m$ and $k$. We may choose $m$ arbitrarily,
so we choose $m = \lceil \log k / (2 \lambda) \rceil \vee 2$.
The forth factor is increasing in $m$
and hence bounded below by $(e^{-2\lambda}-e^{-4 \lambda})/(1-e^{-4 \lambda})$,
which, in turn, is bounded away from $0$ for $\lambda \in [\lambda_1,\lambda_2]$.
The penultimate factor is decreasing in $m$ and thus bounded below by
\begin{equation*}	\textstyle
e^{-2 \lambdacrit (2+\log k / (2 \lambda))} = e^{-4 \lambdacrit} \cdot k^{-\lambdacrit/\lambda}.
\end{equation*}
If $k \geq k_0 \defeq \lfloor e^{2\lambda_2}\rfloor + 1$,
then we can bound the last factor from below by
\begin{equation*}	\textstyle
\big(\frac{1-e^{2\lambda} e^{-\log k }}{1-e^{-\log k}}\big)^{k-1}
\geq \big(\frac{1-e^{2\lambda}/k}{1-1/k}\big)^{k}
\geq e \big(1-\frac{e^{2\lambda}}{k_0}\big)^{k_0},
\end{equation*}
where we have used that, for $a \geq 1$, $(1-a/k)^k$ increases to $e^{-a}$ as $k \to \infty$.
The last term is again bounded away from $0$ for $\lambda \in [\lambda_1,\lambda_2]$.
Consequently, we infer that
\begin{equation*}
\Prob_{\lambda}(\tau_2 -\tau_1 \geq k) \geq  \underline{C}(I,p) k^{-\lambdacrit/\lambda_1}
\end{equation*}
for all $k \geq k_0$ and some $\underline{C}(I,p)$.
By replacing $\underline{C}(I,p)$ by a smaller positive constant if necessary,
we get the above estimate for all $k \geq 0$ from monotonicity arguments.

We now turn to the upper bound.
Let $k \geq 1$, $\lambda \in [\lambda_1,\lambda_2]$ and $\lambda^* > \lambda_2$.
Define $\kappa \defeq \lambdacrit/\lambda^*$.
From Markov's inequality, we get
\begin{equation*}	\textstyle
\Prob_{\lambda}(\tau_2-\tau_1 \geq k) \leq k^{-\kappa} \E_\lambda[(\tau_2-\tau_1)^{\kappa}].
\end{equation*}
It thus suffices to prove that
$\overline{C}(I,p,\lambda^*) \defeq \sup_{\lambda \in [\lambda_1,\lambda_2]} \E_\lambda[(\tau_2-\tau_1)^{\kappa}] < \infty$.
From \eqref{eq:trivial bound E[tau_2-tau_1]} and \eqref{eq:tau_1 decomposed}, we infer
\begin{equation*}	\textstyle
\E_\lambda[(\tau_2-\tau_1)^{\kappa}] \leq \pesc^{-1} \big(
\E^{\circ}_{\lambda}[(\tau_1^\mathcal{B}+\tau_1^{\mathrm{traps}})^{\kappa} \1_{\{Y_k \neq \mathbf{0} \text{ f.\,a.\ } k \geq 1\}}]
\big).
\end{equation*}
From the inequality $(x+y)^\kappa \leq (2^{\kappa-1} \vee 1) (x^\kappa+y^\kappa)$ for $x,y > 0$,
we conclude that it suffices to check that
\begin{equation}	\label{eq:uniform moment estimate tau_1^backbone}	\textstyle
\sup_{\lambda \in [\lambda_1,\lambda_2]} \E_\lambda[(\tau_1^\mathcal{B})^{\kappa} \1_{\{Y_k \neq \mathbf{0} \text{ f.\,a.\ } k \geq 1\}}] < \infty.
\end{equation}
and
\begin{equation}	\label{eq:uniform moment estimate tau_1^traps}	\textstyle
\sup_{\lambda \in [\lambda_1,\lambda_2]} \E_\lambda[(\tau_1^{\mathrm{traps}})^{\kappa} \1_{\{Y_k \neq \mathbf{0} \text{ f.\,a.\ } k \geq 1\}}] < \infty.
\end{equation}
Now notice that \eqref{eq:uniform moment estimate tau_1^backbone} follows from \eqref{eq:upper bound E[tau_1^B^kappa]}
in combination with Lemma \ref{Lem:uniform exponential regeneration point estimates},
while \eqref{eq:uniform moment estimate tau_1^traps} follows from \eqref{eq:upper bound E[tau_1^traps^kappa]}
in combination with Lemma \ref{Lem:time spent in the ith trap} and again Lemma \ref{Lem:uniform exponential regeneration point estimates}.
\end{proof}

\begin{proof}[Proof of Lemma \ref{Lem:uniform regeneration estimates}]
Part (a) is an immediate consequence of Lemma \ref{Lem:uniform exponential regeneration point estimates}.
We turn to part (b).
The local boundedness of $\lambda \mapsto \E_{\lambda}[\tau_1 \1_{\{Y_k \neq \mathbf{0} \text{ f.\,a. } k > 0\}}]$
follows from
\eqref{eq:upper bound E[tau_1^traps^kappa]}, \eqref{eq:upper bound E[tau_1^B^kappa]},
Lemma \ref{Lem:time spent in the ith trap}
and Lemma \ref{Lem:uniform exponential regeneration point estimates}
as below \eqref{eq:uniform moment estimate tau_1^traps}.
In fact, this argument yields the local boundedness in $\lambda$ of the expected time
spent to the right of the origin until the first regeneration time.
The time spent on the negative halfline can be estimated similarly
using the fact that backtracking to the left is (uniformly in $\lambda$)
exponentially unlikely due to two facts.
First an excursion on the backbone is short because of the drift to the right,
see Lemma \ref{Lem:stepping back}.
Backtracking to the left in a trap requires prior backtracking on the backbone
unless the origin is in a trap.
The probability of the event that this happens is exponentially small and independent of $\lambda$,
see Lemma \ref{Lem:trap cost}.
We refrain from providing more details and directly tend towards the more complicated assertion (c).
Fix an interval $I = [\lambda_1,\lambda_2] \subseteq (0,\lambdacrit)$.
Let $\lambda^*  > \lambda_2$.
By Lemma \ref{Lem:uniform lower and upper tail bounds},
there are constants
$\underline{C}(I,p), \overline{C}(I,p,\lambda^*) > 0$
such that
$\underline{C}(I,p) k^{-\lambdacrit/\lambda_1}
\leq
\Prob_{\lambda}(\tau_2 -\tau_1 \geq k)
\leq
\overline{C}(I,p,\lambda^*) k^{-\lambdacrit/\lambda^*}$
for every $\lambda \in I$.
Now let $(\xi_n)_{n \in \N}$ be i.i.d.\ nonnegative random variables
and $\eta$ be a nonnegative random variable with respect to a probability measure $\Prm$
with distributions given via the identities
\begin{equation*}	\textstyle
\Prm(\xi_1 \geq k) \defeq 1 \wedge (\underline{C}(I,p) k^{-\lambdacrit/\lambda_1}) \leq \Prob_{\lambda}(\tau_2 -\tau_1 \geq k),	\quad k \in \N
\end{equation*}
and
\begin{equation*}	\textstyle
\Prm(\eta \geq k) \defeq 1 \wedge (\overline{C}(I,p,\lambda^*) k^{-\lambdacrit/\lambda^*}) \geq \Prob_{\lambda}(\tau_2 -\tau_1 \geq k),	\quad k \in \N.
\end{equation*}
Let $S_n \defeq \xi_1+\ldots+\xi_n$, $n \in \N_0$
and denote by $\mathrm{U}$ the renewal measure of $(S_n)_{n \in \N_0}$ under $\Prm$.
As $S_n \to \infty$ a.\,s.\ under $\Prm$, the renewal measure $\mathrm{U}$ is locally bounded:
$\mathrm{U}(\{k\}) \leq \mathrm{U}(\{0\}) < \infty$ for every $k \in \N_0$.
Moreover, by stochastic domination,
$\mathrm{U}(\{k\})$ dominates $\mathbb{U}_{\lambda}$,
the renewal measure of $(\tau_j)_{j \geq 0}$ under $\Prob_{\lambda}^\circ(\cdot \mid Y_k \neq \mathbf{0} \text{ f.\,a. } k \geq 1)$,
for every $\lambda \in I$. Consequently,
\begin{equation*}	\textstyle
\mathbb{U}_{\lambda}(\{0,\ldots,k\}) \leq \mathrm{U}(\{0,\ldots,k\}) \leq k \mathrm{U}(\{0\})	\text{ for all } k \in \N_0.
\end{equation*}
Using this estimate, we infer for every $\lambda \in I$ and every $k \in \N$,
\begin{align}
\Prob_{\lambda}^\circ&(\tau_{\nu(n)} - n \geq k \mid Y_i \neq \mathbf{0} \text{ for all } i \geq 1)	\notag	\\
&= \sum_{j \geq 0} \Prob_{\lambda}^\circ(\tau_j \leq n, \tau_{j+1} \geq n+k\mid Y_i \neq \mathbf{0} \text{ for all } i \geq 1)	\notag	\\
&= \int_{\{0,\ldots,n\}} \Prob_{\lambda}(\tau_2-\tau_1 \geq n+k-i) \, \mathbb{U}_{\lambda}(\mathrm{d}i)	\notag	\\
& \leq \int_{\{0,\ldots,n\}} \Prm(\eta \geq n+k-i) \, \mathrm{U}(\mathrm{d}i)	\notag	\\
& \leq \mathrm{U}(\{0\}) \sum_{i=0}^n \Prm(\eta \geq k+i).	\label{eq:uniform upper bound}
\end{align}
Now first suppose $\lambda_2 < \lambdacrit/2$.
Then we can choose $\lambda^* \in(\lambda_2, \lambdacrit/2)$.
Since $\Prm(\eta \geq j) \leq \overline{C}(I,p,\lambda^*) j^{-\lambdacrit/\lambda^*}$,
the sum in \eqref{eq:uniform upper bound} is bounded by
\begin{equation*}	\textstyle
\sum_{j \geq k} \Prm(\eta \geq j)
\leq
\overline{C}(I,p,\lambda^*) \sum_{j \geq k} j^{-\lambdacrit/\lambda^*}
\leq \overline{C}(I,p,\lambda^*) \frac{\lambda^*}{\lambdacrit-\lambda^*} (k-1)^{-\lambdacrit/\lambda^*+1},
\end{equation*}
for $k \geq 2$.
Summing over all $k \geq 0$ (using trivial bounds for $k=0,1$),
and using $\lambda^* < \lambdacrit/2$ yields the first assertion in (c).
Next suppose that $\lambda_2 < \lambdacrit$ and $1 < r < \frac{\lambdacrit}{\lambda_2} \wedge 2$.
Choose $\lambda^*\in(\lambda_2, \lambdacrit)$ such that $r < \lambdacrit/\lambda^* < 2$.
Then we infer from \eqref{eq:uniform upper bound}
\begin{align*}
n^{-1/r} & \sup_{\lambda \in I} \E_{\lambda}^{\circ}[\tau_{\nu(n)}- n \mid Y_k \neq \mathbf{0} \text{ for all } k \geq 1]	\\
& \leq n^{-1/r} \sum_{k \geq 0} \mathrm{U}(\{0\}) \sum_{i=0}^n \Prm(\eta \geq k+i)	\\
& \leq \mathrm{U}(\{0\}) n^{-1/r} \bigg(3\Erm[\eta] + \sum_{i=3}^n \sum_{k \geq i} \overline{C}(I,p,\lambda^*) k^{-\lambdacrit/\lambda^*}\bigg)	\\
& \sim \mathrm{U}(\{0\}) \overline{C}(I,p,\lambda^*) n^{-1/r} \sum_{i=3}^n \sum_{k \geq i} k^{-\lambdacrit/\lambda^*}.
\end{align*}
Here,
\begin{align*}
n^{-1/r} \sum_{i=3}^n \sum_{k \geq i} k^{-\lambdacrit/\lambda^*}
&\leq n^{-1/r} \sum_{i=2}^{n-1} \frac{i^{-\lambdacrit/\lambda^*+1}}{\lambdacrit/\lambda^*-1}	\\
&\leq \frac{1}{(\lambdacrit/\lambda^*-1)(2-\lambdacrit/\lambda^*)} n^{-\lambdacrit/\lambda^*+2 - 1/r}
\to 0
\end{align*}
by the choice of $\lambda^*$.
\end{proof}

\section{acknowledgements}
The research of M.\;Meiners was supported by DFG SFB 878 ``Geometry, Groups and Actions''
and by short visit grant 5329 from the European Science Foundation (ESF)
for the activity entitled `Random Geometry of Large Interacting Systems and Statistical Physics'.
The research was partly carried out during visits of M.\;Meiners to Tech\-nische Universit\"at Graz and to Aix-Marseille Universit\'e,
during visits of M.\;Meiners and S.\;M\"uller to Tech\-nische Universit\"at M\"unchen,
and during visits of N.\;Gantert to Technische Universit\"at Darmstadt.
Grateful acknowledgement is made for hospitality to all four universities.

\begin{appendix}
\section{Auxiliary results}  

Throughout the paper, we repeatedly estimate the expectation of the $\kappa$th power of a geometric random variable.
For convenience, we provide this estimate in the following lemma.

\begin{lemma}	\label{Lem:unimodular sum<->integral}
Suppose that $f: [0,\infty) \to [0,\infty)$ is unimodal with maximizer $x^* \geq 0$.
Then
\begin{equation}	\label{eq:unimodular sum<->integral}
\sum_{k \geq 0} f(k)	~\leq~	2f(x^*) + \int_0^{\infty} f(x) \, \dx.
\end{equation}
In particular, for any $r \in (0,1)$ and $\kappa > 0$
\begin{equation}	\label{eq:kappa moment of geometric}
\sum_{k \geq 0} k^{\kappa} r^k	~\leq~	\frac{1}{|\log r|^{\kappa}} \bigg(2 \Big(\frac{\kappa}{e}\Big)^{\!\kappa} + \frac{\Gamma(\kappa+1)}{|\log r|}\bigg).
\end{equation}
\end{lemma}
\begin{proof}
Since $f$ is increasing on $[0,x^*]$ and decreasing on $[x^*,\infty)$, we have
\begin{equation*}
\sum_{k=0}^{\lfloor x^* \rfloor-1} f(k) ~\leq~ \int_0^{\lfloor x^* \rfloor} f(x) \, \dx
\quad	\text{and}	\quad
\sum_{k=\lfloor x^* \rfloor+2} f(k) ~\leq~ \int_{\lfloor x^* \rfloor+1} f(x) \, \dx
\end{equation*}
The estimate \eqref{eq:unimodular sum<->integral} now follows
from the fact that $f(\lfloor x^* \rfloor)+f(\lfloor x^* \rfloor+1) \leq 2f(x^*)$.\smallskip

\noindent
In order to show \eqref{eq:kappa moment of geometric}, set $f(x) \defeq x^{\kappa} r^x$, $x \geq 0$
and observe that $f$ assumes its maximum at $x^* = \kappa/|\log r|$.
The result now follows from the identities
\begin{equation*}
\int_0^{\infty} \!\! f(x) \, \dx = \frac{\Gamma(\kappa+1)}{|\log r|^{\kappa+1}}
\quad	\text{and}	\quad
f(x^*) = \Big(\frac{\kappa}{|\log r|}\Big)^{\!\kappa} e^{-\kappa}.
\end{equation*}
\end{proof}

\end{appendix}


\bibliographystyle{spmpsci}
\bibliography{RWRE}

\end{document}